\documentclass[10pt,twoside,reqno]{amsart}
 \usepackage[top=1.4in, bottom=1.4in, left=1in, right=1in]{geometry}
\usepackage{graphicx}
\usepackage{amsmath,latexsym}
\usepackage{amsfonts,amssymb}
\usepackage{amsmath}
\usepackage{amscd, amsthm}
\usepackage{stmaryrd}
\usepackage{fancyhdr}
\usepackage{commath,enumerate}
\usepackage{indentfirst, hyperref}
\usepackage{mathrsfs}
\usepackage{mathabx} 
\usepackage[toc,page]{appendix}

\usepackage{pgf,tikz}
\usepackage{mathrsfs}
\usetikzlibrary{arrows,calc,positioning}
\usepackage{float}
\numberwithin{figure}{section}
\numberwithin{equation}{section}
\newtheorem{theorem}{Theorem}[section]
\newtheorem{lemma}[theorem]{Lemma}
\newtheorem{proposition}[theorem]{Proposition}

\newtheorem{definition}[theorem]{Definition}
\newtheorem{remark}[theorem]{Remark}

\allowdisplaybreaks[4]

\usepackage{scalerel,stackengine}
\stackMath
\newcommand\reallywidehat[1]{%
\savestack{\tmpbox}{\stretchto{%
  \scaleto{%
    \scalerel*[\widthof{\ensuremath{#1}}]{\kern-.6pt\bigwedge\kern-.6pt}%
    {\rule[-\textheight/2]{1ex}{\textheight}}
  }{\textheight}%
}{0.5ex}}%
\stackon[1pt]{#1}{\tmpbox}%
}
\newcommand{\Rm}[1]{
  \textup{\uppercase\expandafter{\romannumeral#1}}
}

\let\olddefinition\definition
\renewcommand{\definition}{\olddefinition\normalfont}
\let\oldremark\remark
\renewcommand{\remark}{\oldremark\normalfont}

\newcommand{\C}{\mathbb{C}}

\newcommand{\etab}{\boldsymbol{\eta}}

\newcommand{\F}{\mathcal{F}}
\newcommand{\G}{\mathbf{G}}

\newcommand{\N}{\mathbb{N}}
\newcommand{\Nz}{\mathbb{N}_0}
\newcommand{\Nc}{\mathcal{N}}
\renewcommand{\Mc}{\mathcal{M}}
\newcommand{\med}{\mathrm{med}}

\newcommand{\R}{\mathbb{R}}

\newcommand{\Rc}{\mathcal{R}}
\newcommand{\s}{\mathbf{s}}
\newcommand{\T}{\mathbb{T}}
\newcommand{\Tb}{\mathbf{T}}
\newcommand{\Z}{\mathbb{Z}}
\newcommand{\Time}{T}
\newcommand{\cutoff}{\mathfrak{b}}

\newcommand{\cutoffxi}{\mathfrak{d}}

\def\vp{\varphi}
\def\ve{\varepsilon}
\def\px{\partial_x}
\def\pt{\partial_t}
\def\S{\mathcal{S}}

\newcommand{\diff}{\,\mathrm{d}}

\DeclareMathOperator{\sgn}{sgn}

\def\ds{\displaystyle}
\def\bel{\begin{equation}\label}
\def\beq{\begin{equation}}
\def\eeq{\end{equation}}
\def\bega{\begin{array}}
\def\enda{\end{array}}

\makeatletter
\def\@tocline#1#2#3#4#5#6#7{\relax
  \ifnum #1>\c@tocdepth 
  \else
    \par \addpenalty\@secpenalty\addvspace{#2}%
    \begingroup \hyphenpenalty\@M
    \@ifempty{#4}{%
      \@tempdima\csname r@tocindent\number#1\endcsname\relax
    }{%
      \@tempdima#4\relax
    }%
    \parindent\z@ \leftskip#3\relax \advance\leftskip\@tempdima\relax
    \rightskip\@pnumwidth plus4em \parfillskip-\@pnumwidth
    #5\leavevmode\hskip-\@tempdima
      \ifcase #1
       \or\or \hskip 1em \or \hskip 2em \else \hskip 3em \fi%
      #6\nobreak\relax
    \hfill\hbox to\@pnumwidth{\@tocpagenum{#7}}\par
    \nobreak
    \endgroup
  \fi}
\makeatother

\author{John K. Hunter}
\address{Department of Mathematics, University of California at Davis}
\email{jkhunter@ucdavis.edu}
\thanks{JKH was supported by the NSF under grant numbers DMS-1616988 and DMS-1908947}

\author{Jingyang Shu}
\address{Department of Mathematics, Temple University}
\email{jyshu@temple.edu}
\author{Qingtian Zhang}
\address{Department of Mathematics, West Virginia University}
\email{qingtian.zhang@mail.wvu.edu }
\title[SQG Front Equation]{Global Solutions of a Surface Quasi-Geostrophic Front Equation}
\date{\today}

\begin{document}\

\begin{abstract}
We consider a nonlinear, spatially-nonlocal initial value problem in one space dimension on $\mathbb{R}$ that describes the motion of surface quasi-geostrophic (SQG) fronts. We prove that the initial value problem has a unique local smooth solution under a convergence condition on the multilinear expansion of the nonlinear term in the equation, and, for sufficiently smooth and small initial data, we prove that the solution is global.
\end{abstract}

\maketitle

\tableofcontents

\section{Introduction}
In this paper, we prove the existence of global small, smooth solutions of the following initial value problem
\begin{equation}\label{fsqgivp}
\left\{\begin{aligned}
&\ds \varphi_t(x, t) + \int_\R \left[\varphi_x(x, t) - \varphi_x(x + \zeta, t)\right] \bigg\{\frac{1}{|\zeta|} - \frac{1}{\sqrt{\zeta^2 + [\varphi(x, t) - \varphi(x + \zeta, t)]^2}}\bigg\} \diff{\zeta} = 2 \log|\px| \varphi_x(x, t),
\\
 &\ds \varphi(x, 0) = \varphi_0(x),
\end{aligned}\right.\end{equation}
where $\vp \colon \R \times \R_+ \to \R$ is defined for $x\in \R$,  $t\in \R_+$, and $\log |\partial_x|$
is the Fourier multiplier operator with symbol $\log |\xi|$.
Our main result is stated in Theorem~\ref{global}.

This initial value problem describes front solutions of the surface quasi-geostrophic (SQG) equation
\begin{align}\label{SQG}
\theta_t+u\cdot \nabla\theta=0,\qquad
u=(-\Delta)^{-1/2}\nabla^{\perp}\theta,
\end{align}
where $(-\Delta)^{-1/2}$ is a fractional inverse Laplacian on $\R^2$, and $\nabla^\perp = (-\partial_y,\partial_x)$.
The SQG equation arises as a description of quasi-geostrophic flows confined to a surface \cite{sqg_lap, Ped87}. After the incompressible Euler equation, it is the most physically important member of a family of two-dimensional active scalar problems for $\theta$ with a divergence-free transport velocity $u=(-\Delta)^{-\alpha/2}\nabla^{\perp}\theta$ and $0<\alpha \le 2$. The case $\alpha=2$ gives the vorticity-stream function formulation of the incompressible Euler equation \cite{majda}, while $\alpha=1$ gives the SQG equation.

The SQG equation is also of interest from
an analytical perspective because it has similar features to the three-dimensional incompressible Euler equation \cite{sqg}; in both cases, the question of singularity formation in smooth solutions remains open. The SQG equation has global weak solutions \cite{Mar08, Res}, and, as for the Euler equation, non-unique weak solutions of the SQG initial value problem may be constructed by convex integration \cite{BSV19, IM20a}. The SQG equation also has a nontrivial family of global smooth solutions \cite{CCGS18}.

By SQG front solutions, we mean piecewise-constant solutions of \eqref{SQG} with
\[
\theta(x,y, t)=\begin{cases} \theta_+ & \text{if $y>\vp(x,t)$},\\
\theta_- & \text{if $y<\vp(x,t)$},
\end{cases}
\]
where $\theta_+$ and $\theta_-$ are distinct constants, in which $\theta$ has a jump discontinuity across a front located at $y=\vp(x,t)$ with $x\in \R$; in \eqref{fsqgivp}, the jump is
normalized to $\theta_+ - \theta_- = 2\pi$.
We assume that the front is a graph and do not consider questions related to the breaking or filamentation of the front.

We contrast these front solutions with SQG patches, in which
\[
\theta(x,y, t)=\begin{cases} \theta_+ & \text{if $(x,y)\in \Omega(t)$},\\
0 & \text{if $(x,y)\notin \Omega(t)$},
\end{cases}
\]
where $\Omega(t) \subset \R^2$ is a bounded, simply-connected region with smooth boundary.
Contour dynamics equations for the motion of patches in SQG, Euler, and generalized SQG (with arbitrary values of $0<\alpha\le 2$) are straightforward to write down, although they require an appropriate regularization of a locally non-integrable singularity in the Green's function of $(-\Delta)^{\alpha/2}$ when $0<\alpha \le 1$.
Local well-posedness of the contour dynamics equations for SQG and generalized SQG patches is proved in \cite{CCG18, Gan08},
and generalized SQG patches in the more locally singular regime $0<\alpha<1$ are studied in \cite{CCCGW12, KR20a, KR20b}.

The boundary of a vortex patch in the Euler equation remains globally smooth in time \cite{BeCo, Che, Che1},  but this question remains open for SQG patches. Splash singularities cannot occur in a smooth boundary of an SQG patch \cite{GS14}, while numerical results suggest the formation of complex, self-similar singularities in a single patch \cite{Dri, SD19}, and a curvature blow up when two patches touch \cite{CFMR05}. Singularity formation in the boundary of generalized SQG patches has been proved in the presence of a rigid boundary when $\alpha$ is sufficiently close to $2$ \cite{KiRyYaZl, KYZ17}, and a class of nontrivial global smooth solutions for SQG patches is constructed in \cite{CCGS16a, CCGS16b, GS18}.

When $0<\alpha <1$, it is straightforward to derive contour dynamics equations for fronts in the same way as one does for patches. In that case, C\'{o}rdoba \textit{et.~al.}\cite{CGI17} prove the global well-posedness of the initial-value problem on $\R$ for small, smooth generalized SQG fronts.

When $1\le \alpha\le 2$, additional problems arise in the formulation of contour dynamics equations for fronts as a result of the slow decay
of the Green's function and the lack of compact support of $\theta$.
Front equations, including \eqref{fsqgivp}, are derived by a regularization procedure in \cite{HSh17}, and a detailed derivation of
\eqref{fsqgivp} from the SQG equation is given in \cite{HSZ19}. Unlike the front equations with $\alpha\ne 1$, the SQG front equation requires both `ultraviolet' and `infrared' regularization in the front equation to account for the
failure of both local and global integrability of the SQG Green's function $G(r)= 1/r$ on $\R$. This failure leads to the logarithmic derivatives in \eqref{fsqgivp}, rather than the fractional derivatives that occur for generalized SQG fronts with $\alpha\ne 1$.

In the case of spatially periodic fronts with $x\in \T = \R/2 \pi \Z$, one can write down front equations directly by using the Green's function of $(-\Delta)^{\alpha/2}$ on the cylinder $\T\times \R$.
Local well-posedness for spatially-periodic SQG front-type equations is proved in \cite{Ro05} for $C^\infty$-solutions by a Nash-Moser method, and in \cite{FR11} for analytic solutions by a Cauchy-Kowalewski method. Almost sharp fronts, across which $\theta$ is continuous, are studied in \cite{CFR04, FLR12, FR12, FR15}.

The local well-posedness in Sobolev spaces of a cubically nonlinear approximation of \eqref{fsqgivp} for spatially periodic solutions is proved in \cite{HSZ1}. In this paper, we consider the fully nonlinear equation \eqref{fsqgivp} on $\R$. The problem on $\R$ differs from the problem on $\T$ in two respects. First, the logarithmic multiplier $\log|\xi|$ is unbounded at low frequencies, which does not occur on $\T$ when $\xi\in \Z\setminus\{0\}$ is discrete and nonzero. Second, the linearized equation on $\R$ provides dispersive decay, which allows us to get global solutions for sufficiently small, smooth initial data. In this paper, we do not attempt to obtain a sharp regularity result for these solutions.

 The general strategy for proving the global existence of small solutions of dispersive equations is to prove an energy estimate together with a dispersive decay estimate.
 Energy estimates for \eqref{fsqgivp} in the usual $H^s$-Sobolev spaces lead to a logarithmic loss of derivatives \cite{HSh17}.
 However, as shown in \cite{HSZ1} for spatially periodic solutions of the cubic approximation, we can obtain good energy estimates in suitably weighted $H^s$-spaces
 by para-linearizing the equation and using the linear dispersive term to control the logarithmic loss of derivatives from the nonlinear term.

The proof of the dispersive estimates is more delicate. The linear part of the equation provides $t^{-1/2}$ decay for the $L^\infty$-norm of the solution, but this is not sufficient to close the global energy estimates for the full equation, since the $O(t^{-1})$ contribution from the cubically nonlinear term is not integrable in time. We therefore need
to analyze the nonlinear dispersive behavior in more detail. We do this by
the method of space-time resonances introduced by Germain, Masmoudi and Shatah \cite{Germain,GMS09, GMS12}, together with
estimates for weighted $L^\infty_\xi$-norms --- the so-called $Z$-norms --- developed by Ionescu and his collaborators
 \cite{CGI17, DIP17, DIPP16,IP13, IP14, IP15, IPu16}.

Our $Z$-norm estimates in Section~\ref{sec-Znorm} involve a detailed frequency-space analysis. The most difficult part is the estimate of the cubically nonlinear terms.
In most regions of frequency space, these terms are nonresonant, and we can use integration-by-parts in either the spatial or temporal frequency variables to estimate the corresponding oscillatory integrals. In regions of space-time resonances, we use the method of modified scattering to account for the nonlinear, long-time asymptotics of the solutions \cite{IP12, ozawa}.

In \cite{CGI17}, where the authors prove global well-posedness of the initial-value problem for the generalized SQG front equation with $0<\alpha < 1$, the linearized equation $\vp_t=\partial_x|\partial_x|^{1 - \alpha}\vp$ has a scaling invariance, with
dispersion relation $\tau=\xi|\xi|^{1 - \alpha}$, and it commutes with
the vector field  $x\partial_x+ (2 - \alpha) t\partial_t$. This commutation provides a key ingredient in the dispersive estimates. The SQG
equation considered here corresponds to the limiting case $\alpha=1$, and its linearized  dispersion relation is $\tau=2\xi\log|\xi|$. The linearized equation $\vp_t=2\log|\partial_x|\vp_x$ is not scale-invariant, but it has a combined scaling-Galilean invariance
and commutes with the scaling-Galilean vector field $\S = (x+2t)\partial_x+t\partial_t$, which we use to obtain dispersive estimates.

This paper is organized as follows. In Section~\ref{sec-weyl}, we collect some fundamental facts and estimates that we use later. In Section~\ref{sec:expand_para}, we expand and para-linearize the nonlinear terms in the evolution equation.  In Section~\ref{sec-apriori}, we derive the weighted energy estimates and state a local existence and uniqueness result in Theorem~\ref{th:loc_exist}.
In Section~\ref{sec-global}, we state the global existence result in Theorem~\ref{global}. Finally, in Sections \ref{sec-sharp}--\ref{sec-Znorm} we carry out the three key steps in the proof of global existence: linear dispersive estimates; scaling-Galilean estimates; and nonlinear dispersive estimates.

\section{Preliminaries}\label{sec-weyl}
\subsection{Para-differential calculus }

In this section, we state several lemmas for Fourier multiplier operators
that follow from the Weyl para-differential calculus.
Further discussion of the Weyl calculus and para-products can be found in \cite{BCD11, Che1, Hor, Tay00}.

We denote the Fourier transform of $f \colon \R\to \C$ by $\hat f \colon \R\to \C$, where $\hat f= \F f$ is given by
\[
f(x)=\int_{\R} \hat f(\xi) e^{i\xi x} \diff\xi,  \qquad \hat f(\xi)=\frac1{2\pi} \int_{\R}f(x) e^{-i\xi x}\diff{x}.
\]
For $s\in \R$, we denote by $H^s(\R)$ the space of Schwartz distributions $f$ with $\|f\|_{H^s} < \infty$, where
\[
\|f\|_{H^s} = \left[\int_\R \left(1+|\xi|^2\right)^s |\hat{f}(\xi)|^2\, \diff{\xi}\right]^{1/2}.
\]
Throughout this paper, we use $A\lesssim B$ to mean there is a constant $C$ such that $A\leq C B$, and $A\gtrsim B$ to mean there is a constant $C$ such that $A\geq C B$. We use $A\approx B$ to mean that $A\lesssim B$ and $B\lesssim A$.

Let $\chi \colon \R \to \R$ be a smooth function such that
\begin{equation}
\text{$\chi$ is supported in the interval $\{\xi\in \R \mid |\xi|\leq 1/10\}$, and
$\chi(\xi) = 1$ on $\{\xi\in \R \mid |\xi|\leq 3/40\}$}.
\label{defchi}
\end{equation}
If $f$ is a Schwartz distribution and $a \colon \R \times \R \to \C$ is a symbol, then
we define a Weyl paraproduct $T_a f$ by
\begin{equation}
\label{weyldef}
\F \left[T_a f\right](\xi)= \int_{\R} \chi\bigg(\frac{|\xi-\eta|^2}{1+|\xi+\eta|^2}\bigg) \tilde{a}\Big(\xi-\eta, \frac{\xi + \eta}{2}\Big)\hat f(\eta)\diff\eta,
\end{equation}
where $\tilde{a}(\xi,\eta)$ denotes the partial Fourier transform of $a(x, \eta) $ with respect to $x$. For $r_1, r_2 \in \N_0 = \N \cup \{0\}$, we define a normed symbol space by
\begin{align}
\label{defMc}
\begin{split}
\Mc_{(r_1, r_2)} &= \{a \colon \R \times \R \to \C \mid \|a\|_{\Mc_{(r_1, r_2)}} < \infty\},
\\
\|a\|_{\Mc_{(r_1, r_2)}} &= \sup_{(x,\xi) \in \R^2} \left\{\sum_{\alpha=0}^ {r_1} \sum_{\beta=0}^{r_2}(1+|\xi|)^{\beta}\left| \px^\alpha\partial_\xi^\beta a(x, \xi)\right|\right\}.
\end{split}
\end{align}
The following lemma is proved in Appendix~\ref{sec:B}.
\begin{lemma}
\label{lem:taest}
Let $s\in \R$. If $a \in \Mc_{(1,1)}$ and $f \in {H}^s(\R)$, then $T_af\in H^s(\R)$ and
\[
\|T_a f\|_{H^s} \lesssim \|a\|_{\Mc_{(1,1)}} \|f\|_{H^s}.
\]
\end{lemma}

Next, we prove some commutator estimates. We denote by $\log_+|\px|$ the Fourier multiplier with symbol
\[
\log_+|\xi| = \begin{cases} \log|\xi| & \text{for $|\xi| > 1$}, \\
0 & \text{for $|\xi| \le 1$}.\end{cases}
\]
\begin{lemma}\label{Commu}
Let $s\in \R$. Suppose that $f\in H^s(\R)$, $a\in \Mc_{(2,1)}$, and $b\in \Mc_{(1,2)}$. Then
\begin{align}\label{LTa}
\|[\log_+|\px|, T_{a}]f\|_{ H^s}&\lesssim \|a\|_{\Mc_{(2, 1)}} \|f\|_{H^{s-1}},\\\label{xL}
\|[x, \log_+|\px|]f\|_{H^s}&\lesssim \|f\|_{H^{s-1}},\\ \label{xTa}
\|[x, T_{b}]f\|_{H^s}&\lesssim \|b\|_{\Mc_{(1, 2)}} \|f\|_{H^{s}}, \\\label{Txb}
\|x T_{b}f-T_{xb} f\|_{H^s}&\lesssim \|b\|_{\Mc_{(1, 2)}} \|f\|_{H^{s}}.
\end{align}
\end{lemma}

\begin{proof}
{\bf 1.} By the definition \eqref{weyldef} of the Weyl para-product, we have for $\xi\ne 0$ that
\begin{align}
\begin{split}
\F \left[\log_+|\px| T_a v\right](\xi)&= \log_+|\xi| \int_{\R}\chi\bigg(\frac{|\xi-\eta|^2}{1+|\xi+\eta|^2}\bigg)\tilde a\bigg(\xi-\eta, \frac{\xi+\eta}{2}\bigg)\hat v(\eta)\diff\eta
\\
&= \int_{\R}\log_+|\xi-\eta+\eta|\chi\bigg(\frac{|\xi-\eta|^2}{1+|\xi+\eta|^2}\bigg)\tilde a\bigg(\xi-\eta,\frac{\xi+\eta}2\bigg)\hat v(\eta)\diff\eta.
\end{split}
\label{FLT}
\end{align}
If $(\xi,\eta)$ belongs to the support of $\chi({|\xi-\eta|^2}/{(1 + |\xi+\eta|^2)})$, then we claim that
\begin{equation}
\left|\frac{\xi-\eta}{\eta}\right| \leq \frac{17}{18} \text{ when } |\eta|\ge 2.
\label{ineq1}
\end{equation}
To prove this claim, we observe that
\[
\frac{|\xi-\eta|^2}{1+|\xi+\eta|^2}\leq \frac1{10}
\]
implies that
\[
9\left|\frac{\xi-\eta}{\eta} - \frac{2}{9}\right|^2 \le  \frac{40}{9} + \frac{1}{\eta^2} \le \frac{169}{36},
\]
and it follows that
\[
\left|\frac{\xi-\eta}{\eta}\right|\leq \left|\frac{\xi-\eta}{\eta} - \frac{2}{9}\right| + \frac{2}{9} \le  \frac{17}{18}.
\]

We introduce a smooth cutoff function $\iota(\eta)$ supported in $\{|\eta|\le 3\}$ with $\iota(\eta) =1$ on $\{|\eta| \le 2\}$. In view of \eqref{ineq1}, when $|\eta| > 2$ we can use
\[
\log|\xi-\eta+\eta|=\log|\eta|+\log\left|1+\frac{\xi-\eta}{\eta}\right|,
\]
and we obtain from \eqref{FLT} that, for $|\xi|>1$,
\begin{align*}
\F  \left[\log_+|\px| T_a f\right](\xi)
&=\int_{\R}(1-\iota(\eta)) \bigg[\log|\eta|+\log\left|1+\frac{\xi-\eta}{\eta}\right|\bigg] \chi\bigg(\frac{|\xi-\eta|^2}{1+|\xi+\eta|^2}\bigg)\tilde a\left(\xi-\eta,\frac{\xi+\eta}2\right)\hat f(\eta)\diff\eta\\
& \qquad +\log_+|\xi| \int_{\R}\iota(\eta)\chi\bigg(\frac{|\xi-\eta|^2}{1+|\xi+\eta|^2}\bigg)\tilde a\bigg(\xi-\eta, \frac{\xi+\eta}{2}\bigg)\hat f(\eta)\diff\eta.
\end{align*}
We also have
\[
\F \left[T_a \log_+|\px| f\right](\xi)=\int_{\R}\log_+|\eta|\chi\bigg(\frac{|\xi-\eta|^2}{1+|\xi+\eta|^2}\bigg)\tilde a\left(\xi-\eta,\frac{\xi+\eta}2\right) \hat f(\eta)\diff\eta.
\]
By taking the difference of the previous two equations, we get
\begin{align}
\begin{split}
&\F  \left[\log_+|\px| T_a f\right](\xi)-\F \left[T_a \log_+|\px| f\right](\xi)
\\
= ~ & \int_{\R} (1-\iota(\eta))\bigg[\log\left|1+\frac{\xi-\eta}{\eta}\right|\bigg] \chi\bigg(\frac{|\xi-\eta|^2}{1+|\xi+\eta|^2}\bigg)\tilde a\left(\xi-\eta,\frac{\xi+\eta}2\right)\hat f(\eta)\diff\eta\\
&\qquad+\int_{\R}\iota(\eta)(\log_+|\xi| -\log_+|\eta|)\chi\bigg(\frac{|\xi-\eta|^2}{1+|\xi+\eta|^2}\bigg)\tilde a\bigg(\xi-\eta, \frac{\xi+\eta}{2}\bigg)\hat f(\eta)\diff\eta.
\end{split}
\label{LTcom}
\end{align}

The integrand in the first integral on the right-hand side of \eqref{LTcom} is supported on
\[
\{(\xi,\eta)\mid \text{$|\eta|>2$, $|\xi-\eta|/|\eta|<17/18$}\}.
\]
Thus, if $\mathcal P(\xi,\eta)$ is a smooth cut-off function supported in a small neighborhood of this set and equal to $1$ on the set, then the first integral can be written as
\begin{align*}
\int_{\R}\mathcal P(\xi,\eta) \bigg[\frac{\eta}{\xi-\eta}\log\left|1+\frac{\xi-\eta}{\eta}\right|\bigg] \chi\bigg(\frac{|\xi-\eta|^2}{1+|\xi+\eta|^2}\bigg)(\xi-\eta)\tilde a\left(\xi-\eta,\frac{\xi+\eta}2\right)\left[\frac{1-\iota(\eta)}{\eta} \hat f(\eta)\right]\diff\eta.
\end{align*}
We define
\[
\tilde A(\zeta_1,\zeta_2)=\frac{2\zeta_2-\zeta_1}{2i\zeta_1}\log\left|1+\frac{2\zeta_1}{2\zeta_2-\zeta_1}\right|
\widetilde{\partial_1a}(\zeta_1,\zeta_2)\mathcal P\left(\zeta_2+\frac{\zeta_1}2, \zeta_2-\frac{\zeta_1}2\right),
\]
so that
\[
A(x,\zeta_2)=\frac 12\partial_x^{-1}(2\zeta_2+i\partial_x)\log\left|1-2i\partial_x(2\zeta_2+i\partial_x)^{-1}\right| \mathcal P\left( \zeta_2-\frac{i\partial_x}2, \zeta_2+\frac{i\partial_x}2\right)
\partial_xa(x,\zeta_2).
\]
Then the first integral on the right-hand-side of \eqref{LTcom} can be written in terms of a para-differential operator with symbol $A$ as
\begin{align*}
\int_{\R}  \chi\bigg(\frac{|\xi-\eta|^2}{1+|\xi+\eta|^2}\bigg)\tilde{A}\left(\xi-\eta,\frac{\xi+\eta}2\right)\left[\frac{1-\iota(\eta)}{\eta} \hat f(\eta)\right]\diff\eta = \F\left[T_A g\right](\xi),\qquad g= \F^{-1}\left[\frac{1-\iota}{\eta} \hat f\right].
\end{align*}

By Lemma~\ref{lem:taest}, we have
\[
\left\|T_A g\right\| \lesssim \|A\|_{\Mc_{(1,1)}}\|g\|_{H^s} \lesssim\|A\|_{\Mc_{(1,1)}}\|f\|_{H^{s-1}}.
\]
Because of the cut-off function $\mathcal P$, we see that
the support of $\tilde A(\zeta_1,\zeta_2)$ is contained in
\[
|2\zeta_2-\zeta_1|>4,\qquad \left|\frac{2\zeta_1}{2\zeta_2-\zeta_1}\right|<\frac{17}{18}.
\]
So $\partial_1a(\cdot,\zeta_2)\mapsto A(\cdot,\zeta_2)$ is a zeroth order pseudo-differential operator. By carrying out a dyadic decomposition and using Bernstein's inequality \cite{BCD11}, we obtain that
\begin{align*}
\|A\|_{\Mc_{(1,1)}}
\lesssim \|a\|_{\Mc_{(2,1)}}.
\end{align*}
It follows that the first term on the right-hand side of \eqref{LTcom} satisfies the estimate \eqref{LTa}.

For the second term on the right-hand side of \eqref{LTcom}, the cutoff functions $\chi$, $\iota$ ensure that $|\xi| < 6$, $|\eta| < 3$.
Therefore we have the $H^s$-estimate
\begin{align*}
&\left\|(1+|\xi|^2)^{s/2}\int_{\R}\iota(\eta)(\log_+|\xi| -\log_+|\eta|)\chi\bigg(\frac{|\xi-\eta|^2}{1+|\xi+\eta|^2}\bigg)\tilde a\bigg(\xi-\eta, \frac{\xi+\eta}{2}\bigg)\hat f(\eta)\diff\eta\right\|_{L^2_\xi}\\
\lesssim~ &\left\|(1+|\xi|^2)^{s/2}\log_+|\xi|\int_{\R}\chi\bigg(\frac{|\xi-\eta|^2}{1+|\xi+\eta|^2}\bigg)\tilde a\bigg(\xi-\eta, \frac{\xi+\eta}{2}\bigg)[ \iota(\eta)\hat f(\eta)]\diff\eta\right\|_{L^2_\xi}\\
&+\left\|(1+|\xi|^2)^{s/2}\int_{\R}\chi\bigg(\frac{|\xi-\eta|^2}{1+|\xi+\eta|^2}\bigg)\tilde a\bigg(\xi-\eta, \frac{\xi+\eta}{2}\bigg)\left[ \iota(\eta)\log_+|\eta| \hat f(\eta)\right]\diff\eta\right\|_{L^2_\xi}\\
\lesssim~ &\left\|\int_{\R}\chi\bigg(\frac{|\xi-\eta|^2}{1+|\xi+\eta|^2}\bigg)\tilde a\bigg(\xi-\eta, \frac{\xi+\eta}{2}\bigg)[\iota(\eta)\hat f(\eta)]\diff\eta\right\|_{L^2_\xi}\\
&+\left\|\int_{\R}\chi\bigg(\frac{|\xi-\eta|^2}{1+|\xi+\eta|^2}\bigg)\tilde a\bigg(\xi-\eta, \frac{\xi+\eta}{2}\bigg)\left[\iota(\eta) \log_+|\eta| \hat f(\eta)\right]\diff\eta\right\|_{L^2_\xi}
\\
=~& \left\|T_a g\right\|_{L^2} + \left\|T_a h\right\|_{L^2},\qquad g = \F^{-1}[\iota \hat{f}], \quad h =  \F^{-1}[\iota\log_+|\eta|  \hat{f}],
\end{align*}
and Lemma~\ref{lem:taest} implies that the second term also satisfies\eqref{LTa}.

{\bf 2.} Taking Fourier transforms, we get that
\[
\F\Big([x,\log_+|\px|]f\Big) = i\partial_\xi[\log_+|\xi| \hat f(\xi)]-\log_+|\xi| (i\partial_\xi \hat f(\xi))=\frac{i}{\xi}\mathbf1_{|\xi|>1}\hat f(\xi),
\]
where $\mathbf1_{|\xi|>1}$ is the indicator function for the set $\{|\xi|>1\}$.
Then \eqref{xL} follows.

{\bf 3.} To prove \eqref{xTa}, we compute that
\begin{align*}
\F \left[[x, T_{b}]f\right](\xi)
& =i\partial_\xi \widehat{T_b f}(\xi)-\widehat{T_b(x f)}(\xi)\\
& =i \int_{\R} \partial_\xi\left[\chi\left(\frac{|\xi-\eta|^2}{1+|\xi+\eta|^2}\right) \tilde{b}\left(\xi-\eta, \frac{\xi + \eta}{2}\right)\right]\hat f(\eta)\diff\eta\\
& \qquad - i\int_{\R} \chi\left(\frac{|\xi-\eta|^2}{1+|\xi+\eta|^2}\right) \tilde{b}\left(\xi-\eta, \frac{\xi + \eta}{2}\right)\partial_\eta\hat f(\eta)\diff\eta.
\end{align*}
We rewrite the first integral above as
\[\begin{aligned}
&\int_{\R} \partial_\xi \left[\chi\left(\frac{|\xi-\eta|^2}{1+|\xi+\eta|^2}\right) \tilde{b}\left(\xi-\eta, \frac{\xi + \eta}{2}\right)\right]\hat f(\eta)\diff\eta\\
=&\int_{\R} (\partial_{\xi_1}+\partial_{\xi_2})  \left[\chi\left(\frac{|\xi_1-\eta|^2}{1+|\xi_2+\eta|^2}\right) \tilde{b}\left(\xi_1-\eta, \frac{\xi_2 + \eta}{2}\right)\right]\Bigg|_{\xi_1=\xi_2=\xi}\hat f(\eta)\diff\eta\\
=&\int_{\R} (2\partial_{\xi_2}-\partial_{\eta})  \left[\chi\left(\frac{|\xi_1-\eta|^2}{1+|\xi_2+\eta|^2}\right) \tilde{b}\left(\xi_1-\eta, \frac{\xi_2 + \eta}{2}\right)\right]\Bigg|_{\xi_1=\xi_2=\xi}\hat f(\eta)\diff\eta\\
=&\int_{\R} 2\partial_{\xi_2}  \bigg[\chi\bigg(\frac{|\xi_1-\eta|^2}{1+|\xi_2+\eta|^2}\bigg) \tilde{b}\Big(\xi_1-\eta, \frac{\xi_2 + \eta}{2}\Big)\bigg]\bigg|_{\xi_1=\xi_2=\xi}\hat f(\eta)\\
&\qquad\qquad\qquad\qquad\qquad\qquad +\bigg[\chi\bigg(\frac{|\xi-\eta|^2}{1+|\xi+\eta|^2}\bigg) \tilde{b}\Big(\xi-\eta, \frac{\xi + \eta}{2}\Big)\bigg]\partial_{\eta}\hat f(\eta) \diff\eta.
\end{aligned}
\]

It follows that
\begin{align}\nonumber
\F [x, T_{b}]f&= 2i\int_{\R} \partial_{\xi_2}  \left[\chi\left(\frac{|\xi_1-\eta|^2}{1+|\xi_2+\eta|^2}\right) \tilde{b}\left(\xi_1-\eta, \frac{\xi_2 + \eta}{2}\right)\right]\Bigg|_{\xi_1=\xi_2=\xi}\hat f(\eta)\diff\eta
\\\nonumber
&= 2 i \int_{\R}  \frac{2|\xi-\eta|^2(\xi+\eta)}{[1+|\xi+\eta|^2]^2}\chi'\left(\frac{|\xi-\eta|^2}{1+|\xi+\eta|^2}\right) \tilde{b}\left(\xi-\eta, \frac{\xi + \eta}{2}\right)\hat f(\eta)\diff\eta
\\
&\quad + i\int_{\R} \chi\left(\frac{|\xi-\eta|^2}{1+|\xi+\eta|^2}\right) \widetilde{\partial_2b}\left(\xi-\eta, \frac{\xi + \eta}{2}\right)\hat f(\eta)\diff\eta.
\label{xTb}
\end{align}

From \eqref{ineq1}, in the support of the cut-off function $\chi$ we have
\begin{align*}
&\frac1{18}|\eta|\leq |\xi| \leq \frac{35}{18}|\eta| \quad \text{when $|\eta|>2$},
\qquad \text{and} \qquad |\xi|<6\quad \text{when $|\eta|<2$}.
\end{align*}
Thus, the first integral on the right-hand-side of \eqref{xTb} satisfies
\begin{align*}
&\left\|(1+|\xi|^2)^{s/2}\int_{\R}  \frac{2|\xi-\eta|^2(\xi+\eta)}{[1+|\xi+\eta|^2]^2}\chi'\left(\frac{|\xi-\eta|^2}{1+|\xi+\eta|^2}\right) \tilde{b}\left(\xi-\eta, \frac{\xi + \eta}{2}\right)\hat f(\eta)\diff\eta\right\|_{L^2_\xi}\\
\lesssim &\left\|\int_{\R}  \frac{2|\xi-\eta|^2(\xi+\eta)}{[1+|\xi+\eta|^2]^2}\chi'\left(\frac{|\xi-\eta|^2}{1+|\xi+\eta|^2}\right) \tilde{b}\left(\xi-\eta, \frac{\xi + \eta}{2}\right)[\iota(\eta)\hat f(\eta)]\diff\eta\right\|_{L^2_\xi}\\
&\quad+\left\|\int_{\R}  \frac{2|\xi-\eta|^2(\xi+\eta)}{[1+|\xi+\eta|^2]^2}\chi'\left(\frac{|\xi-\eta|^2}{1+|\xi+\eta|^2}\right) \tilde{b}\left(\xi-\eta, \frac{\xi + \eta}{2}\right)[(1-\iota(\eta))(1+|\eta|^2)^{s/2}\hat f(\eta)]\diff\eta\right\|_{L^2_\xi}.
\end{align*}
These terms can be expressed in terms of a Weyl pseudo-differential operator $B^w$ in \eqref{weyl2} with symbol
\[
B(x,\xi) = \frac{4\xi\partial_x}{(1+4\xi^2)^2}\chi'\left(\frac{-\partial_x^2}{1+4\xi^2}\right)b(x,\xi).
\]
Using Theorem \ref{L2PDO} and Bernstein's inequality, we then get that
\begin{align*}
\left\|(1+|\xi|^2)^{s/2}\int_{\R}  \frac{2|\xi-\eta|^2(\xi+\eta)}{[1+|\xi+\eta|^2]^2}\chi'\left(\frac{|\xi-\eta|^2}{1+|\xi+\eta|^2}\right) \tilde{b}\left(\xi-\eta, \frac{\xi + \eta}{2}\right)\hat f(\eta)\diff\eta\right\|_{L^2_\xi} &\lesssim
\left\|B\right\|_{\Mc_{(1,1)}}\|f\|_{H^s}
\\
&\lesssim
\left\|b\right\|_{\Mc_{(1,1)}}\|f\|_{H^s}.
\end{align*}
The second integral on the right-hand-side of \eqref{xTb} is the paraproduct $\F[T_{\partial_2b} f]$.
By using Lemma~\ref{lem:taest} and the previous estimate, we then obtain \eqref{xTa}.

{\bf 4. } We compute that
\[\begin{aligned}
&\F(xT_bf-T_{xb}f)\\
=&i\int_{\R} \partial_\xi \left[\chi\left(\frac{|\xi-\eta|^2}{1+|\xi+\eta|^2}\right) \tilde{b}\left(\xi-\eta, \frac{\xi + \eta}{2}\right)\right]\hat f(\eta)\diff\eta-i\int_{\R} \left[\chi\left(\frac{|\xi-\eta|^2}{1+|\xi+\eta|^2}\right) \partial_1\tilde{b}\left(\xi-\eta, \frac{\xi + \eta}{2}\right)\right]\hat f(\eta)\diff\eta\\
=&i\int_{\R}\left[\partial_\xi \chi\left(\frac{|\xi-\eta|^2}{1+|\xi+\eta|^2}\right) \tilde{b}\left(\xi-\eta, \frac{\xi + \eta}{2}\right)+\frac12 \chi\left(\frac{|\xi-\eta|^2}{1+|\xi+\eta|^2}\right)\partial_2 \tilde{b}\left(\xi-\eta, \frac{\xi + \eta}{2}\right) \right]\hat f(\eta)\diff\eta.
\end{aligned}
\]
The first term satisfies
\[\begin{aligned}
\left\|\int_{\R}\partial_\xi \chi\left(\frac{|\xi-\eta|^2}{1+|\xi+\eta|^2}\right) \tilde{b}\left(\xi-\eta, \frac{\xi + \eta}{2}\right)\hat f(\eta)\diff\eta\right\|_{H^s}\lesssim \|b\|_{\Mc_{(1,1)}}\|f\|_{H^s},
\end{aligned}
\]
and the second term satisfies
\[\begin{aligned}
\left\|\int_{\R}\chi\left(\frac{|\xi-\eta|^2}{1+|\xi+\eta|^2}\right)\partial_2 \tilde{b}\left(\xi-\eta, \frac{\xi + \eta}{2}\right)\hat f(\eta)\diff\eta\right\|_{H^s}\lesssim \|b\|_{\Mc_{(1,2)}}\|f\|_{H^s},
\end{aligned}
\]
which proves \eqref{Txb}.
\end{proof}

Finally, we give an expansion of $|D|$ acting on para-products (\emph{cf.} \cite{kato-ponce}).

\begin{lemma}\label{lem-DsT}
Let $s\in \R$, $s\geq2$. If $a \in \Mc_{(3,1)}$ and $f\in H^s(\R)$, then
\[
|D|^sT_a f=T_a |D|^sf+sT_{Da}|D|^{s-2}Df + \Rc,
\]
where $\Rc$ satisfies
\[
\|\Rc\|_{L^2} \lesssim \|a\|_{\Mc_{(3, 1)}} \|f\|_{H^{s - 2}(\R)},
\]
and $Da$ means that the differential operator $D$ acts on the function $x\mapsto a(x, \xi)$ for fixed $\xi$.
\end{lemma}
\begin{proof}
By the definition of the Weyl paraproduct
\[
\begin{aligned}
\F (|D|^sT_a f)(\xi)&=  |\xi|^s \int_{\R} \chi\left(\frac{|\xi-\eta|^2}{1+|\xi+\eta|^2}\right) \tilde{a}\left(\xi-\eta, \frac{\xi + \eta}{2}\right)\hat f(\eta)\diff\eta
\\
&= \int_{\R}  |\xi-\eta+\eta|^s\chi\left(\frac{|\xi-\eta|^2}{1+|\xi+\eta|^2}\right) \tilde{a}\left(\xi-\eta, \frac{\xi + \eta}{2}\right)\hat f(\eta)\diff\eta,
\end{aligned}
\]
where $\tilde{a}$ denotes the partial Fourier transform of $a$ in the first variable. The low frequency part satisfies the reminder estimate, so it can be absorbed into $\Rc$, and we only need to consider the high frequency part with $|\eta|>2$. In that case, \eqref{ineq1} is satisfied on the support of $\chi\left({|\xi-\eta|^2}/({1 + |\xi+\eta|^2})\right)$. Define $b(x)=(1+x)^s-1-sx$. Then
\[
|\xi-\eta+\eta|^s=|\eta|^s\left|1+\frac{\xi-\eta}{\eta}\right|^s = |\eta|^s\left[1+s\frac{\xi-\eta}{\eta}+b\left(\frac{\xi-\eta}{\eta}\right)\right].
\]
In the expression for $\F\big[|D|^s T_a f\big]$, we get
\[
\begin{aligned}
\F & \big[|D|^s T_a f\big](\xi) \\
&= \int_{\R}  |\eta|^s\left[1+s\frac{\xi-\eta}{\eta}+b\left(\frac{\xi-\eta}{\eta}\right)\right]\chi\bigg(\frac{|\xi-\eta|^2}{1+|\xi+\eta|^2}\bigg) \tilde{a}\Big(\xi-\eta, \frac{\xi + \eta}{2}\Big)\hat f(\eta)\diff\eta.
\end{aligned}
\]
Then we only need to estimate
\begin{align}\label{remhigh}
 \int_{\R}|\eta|^{2} b\left(\frac{\xi-\eta}{\eta}\right)\chi\bigg(\frac{|\xi-\eta|^2}{1+|\xi+\eta|^2}\bigg) \tilde{a}\Big(\xi-\eta, \frac{\xi + \eta}{2}\Big)[(1-\iota(\eta)) |\eta|^{s - 2}\hat f(\eta)]\diff\eta.
\end{align}
Define the symbol $A$ by
\[
\tilde A(\zeta_1,\zeta_2)=\left|\frac{2\zeta_2-\zeta_1}2\right|^{2}\left(1-\iota\Big(\frac{2\zeta_2-\zeta_1}2\Big)\right)b\left(\frac{2\zeta_1}{2\zeta_2-\zeta_1}\right)\tilde a(\zeta_1,\zeta_2).
\]
Then \eqref{remhigh} can be viewed as a para-differential operator with symbol $A$. By considering the supports of $\chi$, $\iota$
and using Bernstein's inequality, we see that
\[
\|A\|_{\Mc_{(1,1)}}\lesssim \|a\|_{\Mc_{(3, 1)}}.
\]
The result then follows by applying Lemma~\ref{lem:taest} to \eqref{remhigh}.
\end{proof}

\subsection{Fourier multipliers}

Let $\psi \colon\R\to [0,1]$ be a smooth function supported in $[-8/5, 8/5]$ and equal to $1$ in $[-5/4, 5/4]$.
For any $k\in \mathbb Z$, we define
\begin{align}
\label{defpsik}
\begin{split}
\psi_k(\xi)&=\psi(\xi/2^k)-\psi(\xi/2^{k-1}), \qquad \psi_{\leq k}(\xi)=\psi(\xi/2^k),\qquad \psi_{\geq k}(\xi)=1-\psi(\xi/2^{k-1}),\\
\tilde\psi_k(\xi)&=\psi_{k-1}(\xi)+\psi_k(\xi)+\psi_{k+1}(\xi),
\end{split}
\end{align}
and denote by $P_k$, $P_{\leq k}$, $P_{\geq k}$, and $\tilde{P}_k$  the Fourier multiplier operators with symbols $\psi_k, \psi_{\leq k}, \psi_{\geq k}$, and $\tilde{\psi}_k$, respectively. Notice that $\psi_k(\xi)=\psi_0(\xi/2^k)$, $\tilde\psi_k(\xi)=\tilde\psi_0(\xi/2^k)$.

It is easy to check that
\begin{equation}
\label{psi-L2}
 \|\psi_k\|_{L^2}\approx  2^{k/2}, \qquad  \|\psi_k'\|_{L^2}\approx 2^{-k/2}.
\end{equation}

We will need the following interpolation lemma, whose proof can be found in \cite{IPu16}.
\begin{lemma}\label{interpolation}
For any $k\in\mathbb Z$ and $f\in L^2(\R)$, we have
\[
\|\widehat{P_kf}\|_{L^\infty}^2\lesssim \|P_k f\|_{L^1}^2\lesssim 2^{-k}\|\hat f\|_{L^2_\xi}\left[2^k\|\partial_\xi\hat f\|_{L^2_\xi}+\|\hat f\|_{L^2_\xi}\right].
\]
\end{lemma}

We will also use an estimate for multilinear Fourier multipliers proved in \cite{IP15}. Before stating the estimate,
we introduce some notation.

We define a norm on symbols $\kappa \colon \R^d \to \C$ by
\[
\|\kappa\|_{S^\infty}=\|\F^{-1}\kappa\|_{L^1},
\]
and define the symbol class
\begin{align}
\label{Sinf}
S^\infty =\left\{\kappa \colon \R^d\to \C \mid \text{$\kappa$ continuous and
$\|\kappa\|_{S^\infty}<\infty$}\right\}.
\end{align}

Given $\kappa \in S^\infty$, we define a multilinear operator $M_\kappa$ acting on Schwartz functions  $f_1,\dotsc, f_m \in \mathcal{S}(\R)$ by
\begin{align*}
M_{\kappa}(f_1,\dotsc,f_m)(x)=\int_{\R^{m}} e^{ix(\xi_1+\dotsb+\xi_m)}\kappa(\xi_1, \dotsc, \xi_m)\hat f_1(\xi_1) \dotsm \hat f_m(\xi_m)\diff{\xi_1} \dotsm \diff{\xi_m}.
\end{align*}

\begin{lemma}\label{multilinear}
(i)~If $\kappa_1, \kappa_2\in S^\infty$, then $\kappa_1\kappa_2\in S^\infty$.

(ii)~Suppose that $1 \le p_1, \dotsc, p_m\leq \infty$, $1 \le p \le \infty$, satisfy
\[
\frac1{p_1}+\frac1{p_2}+ \dotsb +\frac1{p_m}=\frac1p.
\]
 If $\kappa \in S^\infty$, then
\[
\|M_{\kappa}\|_{L^{p_1}\times \dotsb \times L^{p_m}\to L^p}\lesssim \|\kappa\|_{S^\infty}.
\]
(iii)~ Assume $p, q, r\in [1, \infty]$ satisfy $1/p+1/q + 1/r=1$, and $m \in S^\infty_{\eta_1, \eta_2} L^\infty_\xi$. Then, for any $f \in L^p(\R)$, $g \in L^q(\R)$, and $h \in L^r(\R)$,
\[
\left\|\int_{\R^2} m(\eta_1,\eta_2, \xi)\hat f(\eta_1)\hat g(\eta_2)\hat h(\xi-\eta_1-\eta_2) \diff\eta_1\diff\eta_2 \right\|_{L^\infty_\xi}\lesssim \|m\|_{S^\infty_{\eta_1,\eta_2}L^\infty_\xi}\|f\|_{L^p}\|g\|_{L^q}\|h\|_{L^r}.
\]
\end{lemma}

In particular, using interpolation, we can estimate the $S^{\infty}$-norm of a symbol $m(\eta_1,\eta_2)$ in $C_c^\infty$ by
\begin{align}\label{SymEst}
\begin{split}
\|m\|_{S^\infty}\lesssim  \|m\|_{L^1}^{1/4} \|\partial_{\eta_i}^2m\|_{L^1}^{1/2}\|\partial_{\eta_1}^2\partial_{\eta_2}^2m\|_{L^1}^{1/4}
\qquad \text{where $i = 1,2$}.
\end{split}
\end{align}

\section{Reformulation of the equation}
\label{sec:expand_para}

\subsection{Expansion of the equation}
In this section, we expand the nonlinearity in the SQG front equation
\begin{equation}
\label{full-sqg}
\varphi_t(x, t) + \int_\R \left[\varphi_x(x, t) - \varphi_x(x + \zeta, t)\right] \bigg\{\frac{1}{|\zeta|} - \frac{1}{\sqrt{\zeta^2 + [\varphi(x, t) - \varphi(x + \zeta, t)]^2}}\bigg\} \diff{\zeta} = 2 \log|\px| \varphi_x(x, t)
\end{equation}
for fronts with small slopes $|\vp_x| \ll 1$. As we will show, \eqref{full-sqg} can be rewritten as
\begin{align}
\label{Tn-sqg}
\begin{split}
&\varphi_t(x, t) - \sum_{n = 1}^\infty \frac{c_n}{2n + 1} \px \int_{\R^{2n+1}} \Tb_n(\etab_n) \hat\varphi(\eta_1, t) \hat\varphi(\eta_2, t) \dotsm \hat\varphi(\eta_{2n+1}, t) e^{i (\eta_1 + \eta_2 + \dotsb + \eta_{2n + 1}) x}\diff{\etab_n}
\\
&\qquad\qquad= 2 \log|\px| \varphi_x(x, t),
\end{split}
\end{align}
where $\etab_n = (\eta_1, \eta_2, \dotsc, \eta_{2n+1})$, and
\begin{equation}
\label{Tnintdef}
\begin{aligned}
\Tb_n(\etab_n) &= \int_\R \frac{\prod_{j = 1}^{2n+1} (1 - e^{ i \eta_j \zeta})}{|\zeta|^{2n+1}} \diff{\zeta},
\qquad c_n = \frac{\sqrt{\pi}}{\Gamma\left(\frac 12 - n\right) \Gamma(n + 1)}.
\end{aligned}
\end{equation}
We remark that $c_n = O(n^{-1/2})$ as $n \to \infty$.

In fact, if we expand the nonlinearity in \eqref{full-sqg} around $\vp_x(x, t) = 0$, we find that
\[
\begin{aligned}
& \int_\R \bigg[\frac{\varphi_x(x, t) - \varphi_x(x + \zeta, t)}{|\zeta|} - \frac{\varphi_x(x, t) - \varphi_x(x + \zeta, t)}{\sqrt{\zeta^2 + (\varphi(x, t) - \varphi(x + \zeta, t))^2}}\bigg] \diff{\zeta}\\
 = ~& - \sum_{n = 1}^\infty c_n \int_\R \frac{\left[\varphi_x(x, t) - \varphi_x(x + \zeta, t)\right] \cdot [\varphi(x, t) - \varphi(x + \zeta, t)]^{2n}}{|\zeta|^{2n + 1}} \diff{\zeta}\\
= ~& - \sum_{n = 1}^\infty \frac{c_n}{2n + 1} \px \int_\R \left[\frac{\varphi(x, t) - \varphi(x + \zeta, t)}{|\zeta|}\right]^{2n + 1} \diff{\zeta}.
\end{aligned}
\]
Writing
\[
f_n(x) = \int_\R \left[\frac{\varphi(x) - \varphi(x + \zeta)}{|\zeta|}\right]^{2n + 1} \diff{\zeta},
\qquad
\varphi(x) = \int_\R \hat\varphi(\eta) e^{i \eta x} \diff{\eta},
\]
we have
\[
f_n(x) = \int_{\R^{2n+1}} \Tb_n(\etab_n) \hat\varphi(\eta_1) \hat\varphi(\eta_2) \dotsm \hat\varphi(\eta_{2n+1}) e^{i (\eta_1 + \eta_2 + \dotsb + \eta_{2n + 1}) x}\diff{\etab_n},
\]
which gives \eqref{Tn-sqg}.

Isolating the lowest degree nonlinear term  in \eqref{Tn-sqg}, which is cubic, we can also write \eqref{full-sqg} as
\begin{align}
\label{cub-sqg}
\begin{split}
&\varphi_t(x,t) + \frac{1}{6} \px \int_{\R^{3}} \Tb_1(\eta_1,\eta_2,\eta_3) \hat\varphi(\eta_1, t) \hat\varphi(\eta_2, t) \hat\varphi(\eta_{3}, t) e^{i (\eta_1 + \eta_2 +  \eta_{3}) x}\diff{\eta_1}\diff{\eta_2}\diff{\eta_3}
\\
&\qquad\qquad\qquad\qquad + \Nc_{\geq 5}(\varphi)(x,t) = 2 \log|\px| \varphi_x(x,t),
\end{split}
\end{align}
where $\Nc_{\geq 5}(\vp)$ denotes the nonlinear terms of quintic degree or higher
\begin{equation}
\label{N>=5}
\Nc_{\geq 5}(\vp)(x, t) = - \sum_{n = 2}^\infty \frac{c_n}{2n + 1} \partial_x \int_{\R^{2n+1}} \Tb_n(\etab_n) \hat\varphi(\eta_1, t) \hat\varphi(\eta_2, t) \dotsm \hat\varphi(\eta_{2n+1}, t) e^{i (\eta_1 + \eta_2 + \dotsb + \eta_{2n + 1}) x}\diff{\etab_n}.
\end{equation}
Equation \eqref{cub-sqg} will be used in Section \ref{sec-Znorm} in order to carry out nonlinear dispersive estimates, where the main difficulty is controlling the slowest decay in time caused by the lowest degree, cubic nonlinearity.

In the appendix, we evaluate the integrals in \eqref{Tnintdef}
and show that we can write \eqref{Tn-sqg} in the alternative form
\begin{equation}
\label{expd-sqg}
\begin{aligned}
\varphi_t + \px \left\{\sum_{n = 1}^\infty \sum_{\ell = 1}^{2n + 1} (-1)^{\ell + 1} d_{n, \ell}\varphi^{2n - \ell + 1}\px^{2n} \log|\px| \varphi^\ell\right\} = 2 \log|\px| \varphi_x,
\end{aligned}
\end{equation}
where the constants $d_{n, \ell}$ are given in \eqref{ddef}.
We will not use \eqref{expd-sqg} in this paper since it makes sense classically only for $C^\infty$-solutions and does not make explicit the fact that, owing to a cancelation of derivatives, the nonlinear flux in \eqref{expd-sqg} involves at most logarithmic derivatives of $\vp$. However, we remark that if the quintic and higher-order terms in \eqref{expd-sqg} are neglected, then the equation becomes
\begin{align*}
\varphi_t + \frac 12 \px \bigg\{\varphi^2 \log|\px| \varphi_{xx} - \varphi \log|\px| (\vp^2)_{xx} + \frac 13 \log|\px| (\vp^3)_{xx}\bigg\} = 2 \log|\px| \varphi_x,
\end{align*}
which is the cubic approximation for the front equation that is derived in \cite{HSh17} and analyzed in \cite{HSZ1}.

\subsection{Para-linearization of the equation}
\label{sec-simeqn}

In this section, we para-linearize the SQG front equation \eqref{Tn-sqg} and put it in a form that allows us to make weighted energy estimates. This form
extracts a nonlinear term $L(T_{B^{\log}[\vp]}\vp)$ from the flux that is responsible for the logarithmic loss of derivatives in the dispersionless equation.

We use Weyl para-differential calculus to decompose the nonlinearity in \eqref{full-sqg}. In the following, we use $C(n,s)$
to denote
a positive constant depending only on $n$ and $s$, which may change from line to line.

\begin{proposition}
Suppose that $\varphi(\cdot, t) \in H^s(\R)$ with $s \geq 4$ and  $\|\varphi_x\|_{W^{2, \infty}} + \|L \varphi_x\|_{W^{2, \infty}}$ is sufficiently small. Then \eqref{full-sqg} can be written as
\begin{equation}
\label{fsqgeq}
\vp_t + \px T_{B^0[\vp]} \vp + \Rc(\vp) = L \big[(2 - T_{B^{\log}[\vp]}) \vp\big]_x,
\end{equation}
where the symbols $B^0[\vp]$ and $B^{\log}[\vp]$ are given by the following multilinear expansions in $\vp_x$:
\begin{align}
\label{defBlog}
\begin{split}
B^{\log}[\vp](\cdot, \xi) &=\sum_{n = 1}^\infty B^{\log}_n[\vp](\cdot, \xi), \quad B^{0}[\vp](\cdot, \xi) =\sum_{n = 1}^\infty  B^{0}_n[\vp](\cdot, \xi),\\
B^{\log}_n[\vp](\cdot, \xi) &= -\F^{-1}_\zeta \biggl\{ 2c_n \int_{\R^{2n}} \delta\biggl(\zeta - \sum_{j = 1}^{2n} \eta_j\biggr) \prod_{j = 1}^{2n} \biggl[i \eta_j \hat{\vp}(\eta_j) \chi\Big(\frac{(2n + 1)\eta_j}{\xi}\Big)\biggr] \diff{\hat{\etab}_n}\biggr\},\\
B^{0}_n[\vp](\cdot, \xi) &= \F^{-1}_\zeta \biggl\{ 2c_n \int_{\R^{2n}} \delta\bigg(\zeta - \sum_{j = 1}^{2n} \eta_j\bigg) \prod_{j = 1}^{2n} \biggl[i \eta_j \hat{\vp}(\eta_j) \chi\Big(\frac{(2n + 1)\eta_j}{\xi}\Big)\biggr]\biggr.
\\
&\hspace{2.3in} \biggl.\cdot \int_{[0, 1]^{2n}} \log\bigg|\sum_{j = 1}^{2n} \eta_j s_j\bigg| \diff{\hat{\s}_n} \diff{\hat{\etab}_n}\biggr\}.
\end{split}
\end{align}
Here,  $c_n$ is given by \eqref{Tnintdef}, $\delta$ is the delta-distribution, $\chi$ is the cutoff function in \eqref{defchi}, $\hat{\etab}_n = (\eta_1, \eta_2, \dotsc, \eta_{2n})$, and $\hat{\s}_n = (s_1, \dotsc, s_{2n})$. The operators $T_{B^{\log}[\vp]}$ and $T_{B^{0}[\vp]}$ are self-adjoint and their symbols satisfy the estimates
\begin{equation}
\label{B-est}
\begin{aligned}
\|B^{\log}[\vp]\|_{\Mc_{(j, 2)}} &\lesssim \sum_{n = 1}^\infty C(n,s) |c_n| \|\varphi_x\|_{W^{j, \infty}}^{2n},\quad j=2,3,\\
\|B^{0}[\vp]\|_{\Mc_{(2, 2)}} &\lesssim \sum_{n = 1}^\infty  C(n,s) |c_n| \Big(\|L\varphi_x\|_{W^{2, \infty}}^{2n} + \|\varphi_x\|_{W^{2, \infty}}^{2n}\Big),
\end{aligned}
\end{equation}
while the remainder term $\Rc$ satisfies
\begin{equation}
\|\Rc(\vp)\|_{H^s} \lesssim \|\vp\|_{H^s} \left\{\sum_{n = 1}^\infty  C(n, s) |c_n| \Big(\|\varphi_x\|_{W^{2, \infty}}^{2n} + \|L \varphi_x\|_{W^{2, \infty}}^{2n}\Big)\right\},
\label{Rcest}
\end{equation}
where the constants $C(n,s)$ have at most exponential growth in $n$.
\end{proposition}

\begin{proof}
We define
\begin{align*}
f_n(x) &= \int_{\R^{2n+1}} \Tb_n(\etab_n) \hat\varphi(\eta_1) \hat\varphi(\eta_2) \dotsm \hat\varphi(\eta_{2n+1}) e^{i (\eta_1 + \eta_2 + \dotsb + \eta_{2n + 1}) x}\diff{\etab_n}.
\end{align*}
In view of \eqref{Tn-sqg} and the commutator estimate \eqref{LTa}, we only need to prove that
\[
- \sum_{n = 1}^\infty \frac{c_n}{2n + 1} \px f_n(x)=\partial_x T_{B^0[\vp]}\vp+\px[(T_{B^{\log}[\vp]})L\vp]+\Rc,
\]
where $\Rc$ satisfies \eqref{Rcest}, and to do this it suffices to prove for each $n$ that
\begin{align*}
 &\frac{c_n}{2n + 1} \px f_n(x)= - \partial_x T_{B^0_n[\vp]}\vp - \px[(T_{B^{\log}_n[\vp]})L\vp]+\Rc_n,
\\
&\|\Rc_n\|_{H^s}\lesssim C(n,s)|c_n|  \Big(\|\varphi_x\|_{W^{3, \infty}}^{2n} + \|L \varphi_x\|_{W^{3, \infty}}^{2n}\Big) \|\vp\|_{H^s}.
\end{align*}

By symmetry, we can assume that $|\eta_{2n+1}|$ is the largest frequency in the expression of $f_n$. Then
\begin{align}
\label{intprod}
\begin{split}
 &\frac{c_n}{2n + 1} \px f_n(x)=c_n\px \int\limits_{\substack{|\eta_{2n+1}|\geq |\eta_j|\\ \text{ for all } j=1, \dotsc, 2n}} \Tb_n(\etab_n) \hat\varphi(\eta_1) \hat\varphi(\eta_2) \dotsm \hat\varphi(\eta_{2n+1}) e^{i (\eta_1 + \eta_2 + \dotsb + \eta_{2n + 1}) x}\diff{\etab_n}\\
 &\qquad=c_n\px \int_{\R}\int\limits_{\substack{ |\eta_j|\leq |\eta_{2n+1}|\\ \text{ for all } j=1, \dotsc, 2n}} \Tb_n(\etab_n) \hat\varphi(\eta_1) \hat\varphi(\eta_2) \dotsm \hat\varphi(\eta_{2n}) e^{i (\eta_1 + \eta_2 + \dotsb + \eta_{2n }) x}\diff{\hat\etab_n}
 \hat\varphi(\eta_{2n+1})e^{ix\eta_{2n+1 } } \diff{\eta_{2n+1}}.
\end{split}
\end{align}

To proceed, we split the above integral into two parts corresponding to the lower and higher frequencies of
$\eta_{2n+1}$.
Define $\mathbf U_n(\etab_n)= \Tb_n(\etab_n) \chi(\eta_{2n+1})$, and $\boldsymbol{\Lambda}_n(\etab_n) = \Tb_n(\etab_n)-\mathbf U_n(\etab_n)$.
For the lower frequency part, we have
\begin{align*}
&\px \int_{\R}\int\limits_{\substack{ |\eta_j|\leq |\eta_{2n+1}|\\ \text{ for all } j=1, \dotsc, 2n}} \mathbf{U}_n(\etab_n) \hat\varphi(\eta_1) \hat\varphi(\eta_2) \dotsm \hat\varphi(\eta_{2n}) e^{i (\eta_1 + \eta_2 + \dotsb + \eta_{2n }) x}\diff{\hat\etab_n}
 \hat\varphi(\eta_{2n+1})e^{ix\eta_{2n+1 } } \diff{\eta_{2n+1}}\\
 =~&\px \int_{\R}\int\limits_{\substack{ |\eta_j|\leq |\eta_{2n+1}|\\ \text{ for all } j=1, \dotsc, 2n}} \chi(\eta_{2n+1})
  \int_\R \frac{\prod_{j = 1}^{2n+1} (1 - e^{ i \eta_j \zeta})}{|\zeta|^{2n+1}} \diff{\zeta}\hat\varphi(\eta_1) \hat\varphi(\eta_2) \dotsm \hat\varphi(\eta_{2n}) \\
&\hspace{8cm}  \cdot e^{i (\eta_1 + \eta_2 + \dotsb + \eta_{2n }) x}\diff{\hat\etab_n}
 \hat\varphi(\eta_{2n+1})e^{ix\eta_{2n+1 } } \diff{\eta_{2n+1}}\\
 =~&\px \int_{\R} \mathbf{A}_n(x, \eta_{2n+1})\hat\varphi(\eta_{2n+1})e^{ix\eta_{2n+1 } } \diff{\eta_{2n+1}},
\end{align*}
where the symbol $\mathbf{A}_n$ is defined by
\[
\mathbf{A}_n(x,\eta_{2n+1})=\int\limits_{\substack{ |\eta_j|\leq |\eta_{2n+1}|\\ \text{ for all } j=1, \dotsc, 2n}}\chi(\eta_{2n+1})
  \int_\R \frac{\prod_{j = 1}^{2n+1} (1 - e^{ i \eta_j \zeta})}{|\zeta|^{2n+1}} \diff{\zeta}~\hat\varphi(\eta_1) \hat\varphi(\eta_2) \dotsm \hat\varphi(\eta_{2n}) e^{i (\eta_1 + \eta_2 + \dotsb + \eta_{2n }) x}\diff{\hat\etab_n}.
\]

Using an $L^2$-boundedness theorem for pseudo-differential operators (Theorem 1.1 of \cite{Bou99}), Lemma \ref{multilinear}, and the compact support of the cutoff functions, we obtain
\begin{align*}
\left\|\px^n \int_{\R} \mathbf{A}_n(x, \eta_{2n+1})\hat\varphi(\eta_{2n+1})e^{ix\eta_{2n+1 } } \diff{\eta_{2n+1}}\right\|_{L^2}\lesssim \sum\limits_{i, j\leq 1}\|\partial_x^i\partial_{\eta_{2n+1}}^j\mathbf A_n\|_{L^\infty}\|\vp\|_{L^2}\\
\lesssim  \left\|\frac{\mathbb{I}_n(\hat{\etab}_n,\eta_{2n+1}) \chi(\eta_{2n+1})}{\prod_{j=1}^{2n} (1+|\eta_j|)} \prod_{j=1}^{2n} (1+|\eta_j|)
  \int_\R \frac{\prod_{j = 1}^{2n+1} (1 - e^{ i \eta_j \zeta})}{|\zeta|^{2n+1}} \diff{\zeta}\right\|_{S^\infty_{\hat{\etab}_n}L^\infty_{\eta_{2n+1}}}\|\vp\|_{W^{1,\infty}}^{2n}\|\vp\|_{L^2},
\end{align*}
where
\begin{equation}
\mathbb{I}_n(\hat{\etab}_n,\eta_{2n+1}) = \begin{cases}1 & \text{if $|\eta_j|\leq |\eta_{2n+1}|$ for $j=1, \dotsc, 2n$},\\
0 & \text{otherwise}.\end{cases}
\label{defIn}
\end{equation}
Thus, the lower frequency part satisfies the estimate \eqref{Rcest}, and this term can be absorbed in $\Rc$ in \eqref{fsqgeq}.

Next, we consider the higher frequency part in\eqref{intprod}, which we write as
\begin{align}
\begin{split}
&c_n\px \int_{\R}\int\limits_{\substack{ |\eta_j|\leq |\eta_{2n+1}|\\ \text{ for all } j=1, \dotsc, 2n}} \boldsymbol{\Lambda}_n(\etab_n) \prod\limits_{j=1}^{2n}\left\{\chi\left(\frac{(2n+1)\eta_j}{\eta_{2n+1}}\right) + \left[1 - \chi\left(\frac{(2n+1)\eta_j}{\eta_{2n+1}}\right)\right]\right\}\hat\varphi(\eta_j)
\\
&\hspace{2.5in} \cdot  e^{i (\eta_1 + \eta_2 + \dotsb + \eta_{2n }) x}\diff{\hat\etab_n}\hat\varphi(\eta_{2n+1})e^{ix\eta_{2n+1 } } \diff{\eta_{2n+1}}.
\end{split}
\label{hfint}
\end{align}
We expand the product in the above integral, and consider two cases depending on whether a term in the expansion
contains only factors of $\chi$ or contains at least one factor $1-\chi$. In the first case, the frequency $\eta_{2n+1}$ is much larger than
all of the other frequencies, and we can extract a logarithmic derivative acting on the highest frequency; in the second case at least one other frequency is comparable to  $\eta_{2n+1}$, and we get a remainder term by distributing derivatives on comparable frequencies.

{\bf Case I.} When we take only factors of $\chi$ in the expansion of the product in \eqref{hfint}, we get the integral
\begin{align}\label{expandfn1}
c_n\px \int_{\R}\int\limits_{\substack{ |\eta_j|\leq |\eta_{2n+1}|\\ \text{ for all } j=1, \dotsc, 2n}} \boldsymbol{\Lambda}_n(\etab_n) \prod\limits_{j=1}^{2n}\chi\left(\frac{(2n+1)\eta_j}{\eta_{2n+1}}\right)\hat\varphi(\eta_j) e^{i (\eta_1 + \eta_2 + \dotsb + \eta_{2n }) x}\diff{\hat\etab_n}\hat\varphi(\eta_{2n+1})e^{ix\eta_{2n+1 } } \diff{\eta_{2n+1}}.
\end{align}

From \eqref{Tnintdef}, we can write $\boldsymbol{\Lambda}_n = [1-\chi(\eta_{2n+1})]  \Tb_n$ as an integral with respect to
$\s_n = (s_1, s_2, \dotsc, s_{2n + 1})$,
\[
\begin{aligned}
\boldsymbol{\Lambda}_n(\etab_n) &= -(1-\chi(\eta_{2n+1})) \int_\R \sgn{\zeta} \int_{[0, 1]^{2n + 1}} \prod_{j = 1}^{2n + 1} i \eta_j e^{i \eta_j s_j \zeta} \diff{\s_n} \diff{\zeta}\\
&= 2 (-1)^n (1-\chi(\eta_{2n+1}))\bigg(\prod_{j = 1}^{2n + 1} \eta_j\bigg) \int_{[0, 1]^{2n + 1}} \frac{1}{\sum_{j = 1}^{2n + 1} \eta_j s_j} \diff{\s_n}\\
&= 2(1-\chi(\eta_{2n+1})) \bigg(\prod_{j = 1}^{2n} (i \eta_j)\bigg) \int_{[0, 1]^{2n}} \log\bigg|1 + \sum_{j = 1}^{2n} \frac{\eta_j}{\eta_{2n + 1}} s_j\bigg| - \log\bigg|\sum_{j = 1}^{2n} \frac{\eta_j}{\eta_{2n + 1}} s_j\bigg| \diff{\hat{\s}_n}\\
&= 2(1-\chi(\eta_{2n+1})) \log|\eta_{2n + 1}| \cdot \prod_{j = 1}^{2n} (i \eta_j) - 2 \bigg(\prod_{j = 1}^{2n} (i \eta_j)\bigg) \int_{[0, 1]^{2n}} \log\bigg|\sum_{j = 1}^{2n} \eta_j s_j\bigg| \diff{\hat{\s}_n}\\
&\qquad +(1-\chi(\eta_{2n+1}))\bigg(\prod_{j = 1}^{2n} (i \eta_j)\bigg) \int_{[0, 1]^{2n}} \log\bigg|1 + \sum_{j = 1}^{2n} \frac{\eta_j}{\eta_{2n + 1}} s_j\bigg|\diff{\hat{\s}_n}.
\end{aligned}
\]
Substitution of this expression into \eqref{expandfn1} gives the following three terms
\begin{align}\label{Tlog}
&c_n\px \int_{\R}\int\limits_{\substack{ |\eta_j|\leq |\eta_{2n+1}|\\ \text{ for all } j=1, \dotsc, 2n}} \boldsymbol{\Lambda}_n^{\log}(\etab_n) \prod\limits_{j=1}^{2n}\chi\left(\frac{(2n+1)\eta_j}{\eta_{2n+1}}\right)\hat\varphi(\eta_j) e^{i (\eta_1 + \eta_2 + \dotsb + \eta_{2n }) x}\diff{\hat\etab_n}\hat\varphi(\eta_{2n+1})e^{ix\eta_{2n+1 } } \diff{\eta_{2n+1}},
\\
\label{T0}
&c_n\px \int_{\R}\int\limits_{\substack{ |\eta_j|\leq |\eta_{2n+1}|\\ \text{ for all } j=1, \dotsc, 2n}} \boldsymbol{\Lambda}_n^{0}(\etab_n) \prod\limits_{j=1}^{2n}\chi\left(\frac{(2n+1)\eta_j}{\eta_{2n+1}}\right)\hat\varphi(\eta_j) e^{i (\eta_1 + \eta_2 + \dotsb + \eta_{2n }) x}\diff{\hat\etab_n}\hat\varphi(\eta_{2n+1})e^{ix\eta_{2n+1 } } \diff{\eta_{2n+1}},
\\
\label{remainder1}
&c_n\px \int_{\R}\int\limits_{\substack{ |\eta_j|\leq |\eta_{2n+1}|\\ \text{ for all } j=1, \dotsc, 2n}} \boldsymbol{\Lambda}_n^{\leq -1}(\etab_n)(\etab_n) \prod\limits_{j=1}^{2n}\chi\left(\frac{(2n+1)\eta_j}{\eta_{2n+1}}\right)\hat\varphi(\eta_j) e^{i (\eta_1 + \eta_2 + \dotsb + \eta_{2n }) x}\diff{\hat\etab_n}\hat\varphi(\eta_{2n+1})e^{ix\eta_{2n+1 } } \diff{\eta_{2n+1}},
\end{align}
where
\[
\begin{aligned}
\boldsymbol{\Lambda}_n^{\log}(\etab_n) &= 2(1-\chi(\eta_{2n+1})) \log|\eta_{2n + 1}| \cdot \prod_{j = 1}^{2n} (i \eta_j),\\
\boldsymbol{\Lambda}_n^{0}(\etab_n) &= - 2(1-\chi(\eta_{2n+1})) \bigg(\prod_{j = 1}^{2n} (i \eta_j)\bigg) \int_{[0, 1]^{2n}} \log\bigg|\sum_{j = 1}^{2n} \eta_j s_j\bigg| \diff{\hat{\s}_n},
\\
\boldsymbol{\Lambda}_n^{\leq -1}(\etab_n) &= 2(1-\chi(\eta_{2n+1})) \bigg(\prod_{j = 1}^{2n} (i \eta_j)\bigg) \int_{[0, 1]^{2n}} \log\bigg|1 + \sum_{j = 1}^{2n} \frac{\eta_j}{\eta_{2n + 1}} s_j\bigg| \diff{\hat{\s}_n}.
\end{aligned}
\]

We claim that the terms \eqref{Tlog} and \eqref{T0} can be rewritten as
\begin{equation}
- \partial_xT_{B_n^{\log}[\vp]}\log_+|\px|\vp+\Rc_1 \quad \text{and}\quad -\partial_x T_{B^0_n[\vp]}\vp+\Rc_2,
\label{TlT0}
\end{equation}
where $\Rc_1$ and $\Rc_2$ satisfy the estimate \eqref{Rcest}.
Indeed,
\begin{align*}
\F\left[\partial_xT_{B_n^{\log}[\vp]}\log_+|\px|\vp\right](\xi)&= - 2 c_n i \xi \int_{\R} \chi\left(\frac{|\xi-\eta|^2}{1+|\xi+\eta|^2}\right) \log_+|\eta| \int_{\R^{2n}} \delta\bigg(\xi-\eta - \sum_{j = 1}^{2n} \eta_j\bigg)
\\
&\hskip1in \cdot\prod_{j = 1}^{2n} \left[i \eta_j \hat{\vp}(\eta_j) \chi\Big(\frac{2(2n + 1)\eta_j}{\xi+\eta}\Big)\right] \diff{\hat{\etab}_n} \hat \vp(\eta)\diff \eta,
\end{align*}
while the Fourier transform of \eqref{Tlog} is
\begin{align*}
&2 c_n i \xi \int_{\R}\int\limits_{\substack{ |\eta_j|\leq |\eta_{2n+1}|\\ \text{ for all } j=1, \dotsc, 2n}} \delta\left(\xi-\sum\limits_{j=1}^{2n+1}\eta_j\right) (1-\chi(\eta_{2n+1})) \log|\eta_{2n + 1}|
\\
&\hskip1in\cdot  \prod\limits_{j=1}^{2n}\left[\chi\left(\frac{(2n+1)\eta_j}{\eta_{2n+1}}\right)(i \eta_j) \hat\varphi(\eta_j) \right] \diff{\hat\etab_n}\hat\varphi(\eta_{2n+1})\diff{\eta_{2n+1}}.
\end{align*}
The difference of the above two integrals is
\begin{align}\label{error}
\begin{split}
& 2 c_n i \xi \int_{\R^{2n+1}}\mathbf \delta\left(\xi-\sum\limits_{j=1}^{2n+1}\eta_j\right)  \log|\eta_{2n + 1}|
\cdot \biggl[ \mathbb{I}_n(\hat\etab_n,\eta_{2n+1}) \prod\limits_{j=1}^{2n}\chi\left(\frac{(2n+1)\eta_j}{\eta_{2n+1}}\right)(i \eta_j) \hat\varphi(\eta_j)(1-\chi(\eta_{2n+1}))
\\
&\hskip1in- \chi\left(\frac{|\xi-\eta_{2n+1}|^2}{1+|\xi+\eta_{2n+1}|^2}\right)\mathbf{1}_{|\eta_{2n+1}|>1}\prod_{j = 1}^{2n} \bigg(i \eta_j \hat{\vp}(\eta_j) \chi\Big(\frac{2(2n + 1)\eta_j}{\xi+\eta_{2n+1}}\Big)\bigg)\biggr]\diff{\hat\etab_n}\hat\varphi(\eta_{2n+1})\diff{\eta_{2n+1}},
\end{split}
\end{align}
where $\mathbb{I}_n$ is given by \eqref{defIn}.

When $\etab_n$ satisfies
\begin{equation}\label{etancon}
|\eta_j|\leq \frac1{40} \frac1{2n+1}|\eta_{2n+1}|\qquad \text{for $j=1,2,\dotsc, 2n$},
\end{equation}
we have $\mathbb I_n= 1$ and $\chi\left(\frac{(2n+1)\eta_j}{\eta_{2n+1}}\right)=1$. In addition, since $\xi=\sum_{j=1}^{2n+1}\eta_j$, we have
\begin{align*}
&\frac{|\xi-\eta_{2n+1}|^2}{1+|\xi+\eta_{2n+1}|^2}\leq\frac{|\xi-\eta_{2n+1}|}{|\xi+\eta_{2n+1}|}
=\frac{\left|\sum\limits_{j=1}^{2n}\eta_{j}\right|}{\left|\sum\limits_{j=1}^{2n}\eta_{j}+2\eta_{2n+1}\right|}\leq \frac{\frac1{40} |\eta_{2n+1}|}{(2-\frac1{40})|\eta_{2n+1}|}=\frac1{79}<\frac3{40},
\\
&\frac{2(2n + 1)|\eta_j|}{|\xi+\eta_{2n+1}|}\leq\frac{\frac1{20}|\eta_{2n+1}|}{(2-\frac1{40})|\eta_{2n+1}|}=\frac2{79}<\frac{3}{40},
\end{align*}
so
\[
\chi\left(\frac{|\xi-\eta_{2n+1}|^2}{1+|\xi+\eta_{2n+1}|^2}\right)=1,\qquad
\chi\left(\frac{2(2n + 1)\eta_j}{\xi+\eta_{2n+1}}\right)=1.
\]

Therefore the integrand of \eqref{error} is supported outside the set \eqref{etancon},
and there exists $j_1\in \{1, \dotsc, 2n\}$, such that
$|\eta_{j_1}|> \frac1{40} \frac1{2n+1}|\eta_{2n+1}|$.
Since $|\eta_{2n+1}|$ is the largest frequency, we see that $|\eta_{j_1}|$ and $|\eta_{2n+1}|$ are comparable in the error term. Therefore, the $H^s$-norm of \eqref{error} is bounded by
\[
\|\vp\|_{H^s}C(n,s)|c_n|  \Big(\|\varphi_x\|_{W^{2, \infty}}^{2n} + \|L \varphi_x\|_{W^{2, \infty}}^{2n}\Big).
\]
It follows that \eqref{Tlog} can be written as in \eqref{TlT0}. A similar calculation applies to \eqref{T0}.

Next, we estimate the symbols $B^{\log}_n[\vp]$ and $B^{0}_n[\vp]$. First, we notice that they are real-valued, so that $T_{B^{\log}_n[\vp]}$ and $T_{B^{0}_n[\vp]}$ are self-adjoint.
Again, without loss of generality, we assume $|\eta_{2n}| = \max_{1 \leq j \leq 2n} |\eta_{j}|$ and observe that
\[
\begin{aligned}
& \int_{[0, 1]^{2n}} \log\bigg|\sum_{j = 1}^{2n} \eta_j s_j\bigg| \diff{\hat{\s}_n}\\
= ~~& \log|\eta_{2n}| + \int_{[0, 1]^{2n - 1}} \bigg\{ \bigg(\sum_{j = 1}^{2n - 1} \frac{\eta_j}{\eta_{2n}} s_j\bigg) \log\bigg|1 + \frac{1}{\sum_{j = 1}^{2n - 1} \frac{\eta_j}{\eta_{2n}} s_j}\bigg| + \log\bigg|1 + \frac{1}{\sum_{j = 1}^{2n - 1} \frac{\eta_j}{\eta_{2n}} s_j}\bigg| - 1 \bigg\}\diff{\s_{n - 1}}\\
= ~~& \log|\eta_{2n}| + O(1).
\end{aligned}
\]
Thus, using Young's inequality, we obtain from \eqref{defBlog} the estimate \eqref{B-est},
where the constants $C(n,s)$ have at most exponential growth in $n$.

To estimate the third term \eqref{remainder1}, we observe that on the support of the functions $\chi\left(\frac{(2n+1)\eta_j}{\eta_{2n+1}}\right)$, we have
\[
\frac{|\eta_j|}{|\eta_{2n+1}|}\leq \frac1{10(2n+1)}.
\]
Since $s_j\in [0,1]$, a Taylor expansion gives
\[
\left|\boldsymbol{\Lambda}_n^{\leq -1}(\etab_n) \right|\lesssim   \frac{\left[\prod_{j = 1}^{2n} |\eta_j|\right]\left[\sum_{j = 1}^{2n} |\eta_j|\right]}{|\eta_{2n + 1}|}.
\]
Therefore the $H^s$-norm of \eqref{remainder1} is bounded by $C(n,s)|c_n|\|\vp\|_{H^s}\|\vp_x\|_{W^{2,\infty}}^{2n}$, where $C(n,s)$ has at most exponential growth in $n$.

{\bf Case II.} When there is at least one factor of the form $1-\chi$ in the expansion of the product in the integral \eqref{hfint},
we get a term of the form
\begin{align}\label{expandfn2}
\begin{split}
f_{(n)}(x) = c_n\px \int_{\R}\int\limits_{\substack{ |\eta_j|\leq |\eta_{2n+1}|\\ \text{ for all } j=1, \dotsc, 2n}} \boldsymbol{\Lambda}_n(\etab_n) \prod\limits_{k=1}^{\ell}\left[1-\chi\left(\frac{(2n+1)\eta_{j_k}}{\eta_{2n+1}}\right)\right] \prod\limits_{k=\ell+1}^{2n}\chi\left(\frac{(2n+1)\eta_{j_k}}{\eta_{2n+1}}\right)\\
\cdot \bigg(\prod_{j = 1}^{2 n} \hat\varphi(\eta_j)\bigg) e^{i (\eta_1 + \eta_2 + \dotsb + \eta_{2n }) x}\diff{\hat\etab_n}\hat\varphi(\eta_{2n+1})e^{ix\eta_{2n+1 } } \diff{\eta_{2n+1}},
\end{split}
\end{align}
where $1\le \ell \le 2n$ is an integer, and $\{j_k : k=1,\dotsc, 2n\}$ is a permutation of $\{1, \dotsc, 2n\}$.

$1-\chi\left(\frac{(2n+1)\eta_{j_1}}{\eta_{2n+1}}\right)$ is compactly supported on
\[
\frac{|\eta_{j_1}|}{|\eta_{2n+1}|}\geq \frac3{40(2n+1)}.
\]
By assumption, $\eta_{2n+1}$ has the largest absolute value, so
\[
\frac{3}{40(2n+1)} |\eta_{2n+1}| \leq |\eta_{j_1}|\leq |\eta_{2n+1}|,
\]
meaning that the frequencies $|\eta_{j_1}|$ and $|\eta_{2n+1}|$ are comparable.

Without loss of generality, we assume that $|\eta_{j_1}| \leq |\eta_{j_2}| \leq \dotsb \leq |\eta_{j_{2n}}| \leq |\eta_{2n + 1}|$, define $\eta_{j_{2n + 1}} = \eta_{2n + 1}$, and, using \eqref{Tnintdef}, split the integral for $\boldsymbol{\Lambda}_n$ into three parts:
\begin{align*}
\boldsymbol{\Lambda}_n(\etab_n) = \boldsymbol{\Lambda}_n^{low}(\etab_n) + \sum_{k = 1}^{2n} \boldsymbol{\Lambda}_n^{med, (k)}(\etab_n) + \boldsymbol{\Lambda}_n^{high}(\etab_n),
\end{align*}
where
\begin{align}
\boldsymbol{\Lambda}_n^{low}(\etab_n) & = \left[1-\chi(\eta_{2n+1})\right]\int_{|\eta_{2n+1} \zeta| < 2} \frac{\prod_{j = 1}^{2n + 1} \left(1 - e^{i \eta_j \zeta}\right)}{\zeta^{2 n + 1}}  \sgn{\zeta} \diff{\zeta}, \label{Tn-low}\\
\boldsymbol{\Lambda}_n^{med, (k)}(\etab_n) & = \left[1-\chi(\eta_{2n+1})\right]\int_{{2}/{|\eta_{j_{k + 1}}|} \le |\zeta| \le {2}/{|\eta_{j_k}|}} \frac{\prod_{j = 1}^{2n + 1} \left(1 - e^{i \eta_j \zeta}\right)}{\zeta^{2 n + 1}}  \sgn{\zeta} \diff{\zeta}, \label{Tn-med}\\
\boldsymbol{\Lambda}_n^{high}(\etab_n) & = \left[1-\chi(\eta_{2n+1})\right]\int_{|\eta_{j_1} \zeta| > 2} \frac{\prod_{j = 1}^{2n + 1} \left(1 - e^{i \eta_j \zeta}\right)}{\zeta^{2 n + 1}}  \sgn{\zeta} \diff{\zeta}. \label{Tn-high}
\end{align}

To estimate \eqref{Tn-low}, we notice that
\begin{align*}
|\boldsymbol{\Lambda}_n^{low}(\etab_n)| & \le \prod_{k = 1}^{2n+1} |\eta_{k}| \cdot \int_{|\eta_{2n+1} \zeta| < 2} \bigg(\prod_{k = 1}^{2n+1} \frac{|1 - e^{i \eta_{j_k} \zeta}|}{|\eta_{j_k} \zeta|}\bigg)  \diff{\zeta} \leq C(n, s) \bigg(\prod_{k = 1}^{2n} |\eta_{j_k}| \bigg).
\end{align*}

For each $1 \le k \le 2n $, we consider two cases. If $k \neq 2n$, we estimate \eqref{Tn-med} as
\begin{align*}
|\boldsymbol{\Lambda}_n^{med, (k)}(\etab_n)| & \leq \prod_{\ell = 1}^{k} |\eta_{j_\ell}| \cdot \int_{{2}/{|\eta_{j_{k + 1}}|} \le |\zeta| \le {2}/{|\eta_{j_k}|}} \bigg(\prod_{\ell = 1}^k \frac{|1 - e^{i \eta_{j_\ell} \zeta}|}{|\eta_{j_\ell} \zeta|}\bigg) \cdot \frac{\prod_{\ell = k + 1}^{2n + 1} |1 - e^{i \eta_{j_\ell} \zeta}|}{|\zeta|^{2n + 1 - k}}  \diff{\zeta}\\
& \leq 2^{2n + 1 - k} \prod_{\ell = 1}^{k} |\eta_{j_\ell}| \cdot \int_{{2}/{|\eta_{j_{k + 1}}|} \le |\zeta| \le {2}/{|\eta_{j_k}|}} |\zeta|^{-2n - 1 + k } \diff{\zeta}\\
& \leq \frac{2}{2n  - k } \left(|\eta_{j_k}|^{2n  - k } + |\eta_{j_{k + 1}}|^{2n  - k }\right) \prod_{\ell = 1}^{k} |\eta_{j_\ell}|\\
& \leq 2  \bigg(\prod_{k = 1}^{2n} |\eta_{j_k}|\bigg).
\end{align*}

If $k = 2n$, we have
\begin{align*}
|\boldsymbol{\Lambda}_n^{med, (k)}(\etab_n)| & \leq 2 \prod_{\ell = 1}^{2n} |\eta_{j_\ell}| \cdot \int_{{2}/{|\eta_{j_{2n + 1}}|} \le |\zeta| \le {2}/{|\eta_{j_{2n}}|}} \frac{1}{|\zeta|} \diff{\zeta}\\
& = 4 \prod_{\ell = 1}^{2n} |\eta_{j_\ell}| \cdot \log\bigg|\frac{\eta_{j_{2n + 1}}}{\eta_{j_{2n}}}\bigg| \leq C(n, s) \prod_{\ell = 1}^{2n} |\eta_{j_\ell}|,
\end{align*}
where the last line follows from the fact that $|\eta_{j_{2n}}|$ and $|\eta_{j_{2n + 1}}|$ are comparable.

As for \eqref{Tn-high}, we have
\begin{align*}
|\boldsymbol{\Lambda}_n^{high}(\etab_n)| & \leq |\eta_{j_1}| \int_{|\eta_{j_1} \zeta| > 2} \bigg(\prod_{k = 2}^{2n + 1} \frac{|1 - e^{i \eta_{j_k} \zeta}|}{|\zeta|} \bigg) \cdot \frac{|1 - e^{i \eta_{j_1} \zeta}|}{|\eta_{j_1} \zeta|} \diff{\zeta}\\
& \le 2^{2n} |\eta_{j_1}| \int_{|\eta_{j_1} \zeta| > 2} \frac{\diff{\zeta}}{|\zeta|^{2n }}\\
& \le \frac{4}{2n - 1} \bigg(\prod_{k = 1}^{2n} |\eta_{j_k}|\bigg) .
\end{align*}

Collecting these estimates, we find that
\[
|\boldsymbol{\Lambda}_n(\etab_n)| \leq C(n, s) \bigg(\prod_{k = 1}^{2n} |\eta_{j_k}|\bigg) .
\]

Using the $L^2$-boundedness theorem for pseudo-differential operators, we can bound the $H^s$-norm of $f_{(n)}$ in \eqref{expandfn2} by
\begin{align*}
\|f_{(n)}\|_{H^s}\lesssim \sum\limits_{j, k = 0, 1}\|\px^j \partial_{\eta_{2n+1}}^k \mathbf P_n\|_{L^\infty_{x,\eta_{2n+1}}} \|\vp\|_{H^s},
\end{align*}
where
\begin{align*}
&\mathbf P_n(x,\eta_{2n+1})=\left(i\sum\limits_{j=1}^{2n+1}\eta_{j}\right)\int_{\R^{2n}}\mathbb{I}_n(\hat\etab_n, \eta_{2n+1}) \Tb_n(\etab_n)  \prod\limits_{k=1}^{\ell}\left[1-\chi\left(\frac{(2n+1)\eta_{j_k}}{\eta_{2n+1}}\right)\right]
\\
&\qquad\qquad\qquad\qquad\qquad  \cdot \prod\limits_{k=\ell+1}^{2n}\chi\left(\frac{(2n+1)\eta_{j_k}}{\eta_{2n+1}}\right)
\prod_{j = 1}^{2 n} \hat\varphi(\eta_j) e^{i (\eta_1 + \eta_2 + \dotsb + \eta_{2n }) x}\diff{\hat\etab_n}.
\end{align*}
Considering the support of the cut-off functions, we therefore have
\[
\|f_{(n)}\|_{H^s}\lesssim \|\vp\|_{H^s} \bigg(\sum_{n = 1}^\infty  C(n, s) |c_n| \|\varphi_x\|_{W^{2, \infty}}^{2n}\bigg).
\]
So we have proved that the equation can be written as
\[
\vp_t + \px T_{B^0[\vp]} \vp + \Rc(\vp) =2L\vp_x  -\big[( T_{B^{\log}[\vp]}) \log_+|\px| \vp\big]_x.
\]
Then the proposition follows by the commutator estimate \eqref{LTa} and the fact that $\mathbf 1_{|\xi|<1}\px\log|\px|$ is
bounded from $H^s(\R)$ to $H^s(\R)$.
\end{proof}

\section{Energy estimates and local well-posedness}\label{sec-apriori}

In this section, we prove a local existence and uniqueness result for the initial value problem \eqref{fsqgivp},
which is stated in Theorem~\ref{th:loc_exist}.

As noted in the introduction, standard $H^s$-estimates
do not close, so we introduce a weighted energy $E^{(s)}$.
The solutions we construct satisfy $\|T_{B^{\log}[\vp]}\|_{L^2\to L^2} < 2$, and then $(2 - T_{B^{\log}[\vp]})$ is a positive, self-adjoint operator on $L^2$.
We can therefore define homogeneous and nonhomogeneous weighted energies that are equivalent to the $H^s$-energies by
\begin{equation}
E^{(s)}(t) = \int_\R |D|^s \varphi(x, t) \cdot \Big(2 - T_{B^{\log}[\vp]}\Big)^{2s + 1} |D|^s \vp(x, t) \diff{x},
\qquad\tilde{E}^{(s)}(t) = \sum_{j=0}^s E^{(j)}(t).
\label{weighted_energy1}
\end{equation}
For simplicity, we consider only integer norms with $s\in \N$.

We begin by stating an \emph{a priori} energy estimate. In the following, we use $F$ to denote an increasing, continuous, non-negative function, which might change from line to line.

\begin{proposition}\label{apriori}
Let $s \ge 5$ be an integer and $\vp_0 \in {H}^s(\R)$. Then there exist constants
$\tilde{C} > 0$, depending only on $s$, and $T>0$ such that the following statement holds: If
$\vp_0$ satisfies
\[
\big\|T_{B^{\log}[\vp_0]}\|_{L^2 \to L^2} \leq C,\qquad
\sum_{n = 1}^\infty \tilde{C}^n |c_n| \Big(\|\px\varphi_0\|_{W^{3, \infty}}^{2n} + \|L\px \varphi_0\|_{W^{2, \infty}}^{2n}\Big) < \infty
\]
for some constant $0 < C < 2$, where $c_n$ is defined in \eqref{Tnintdef}, then a solution $\vp\in C([0,T]; H^s(\R))$ of \eqref{fsqgivp} with initial data $\vp(\cdot, 0) = \vp_0$ satisfies
\begin{align}
\tilde{E}^{(s)}(t) &\le \tilde{E}^{(s)}(0) + \int_0^t \big(\|\varphi_x(\tau)\|_{W^{3, \infty}} + \|L \varphi_x(\tau)\|_{W^{2, \infty}}\big)^2 F\left(\|\vp_x(\tau)\|_{W^{3,\infty}} + \|L \vp_x(\tau)\|_{W^{2,\infty}}\right) \tilde{E}^{(s)}(\tau) \diff{\tau},
\label{apest}
\\
&\big\|T_{B^{\log}[\vp(t)]}\|_{L^2 \to L^2} < 2,\qquad
\sum_{n = 1}^\infty \tilde C^n |c_n| \Big(\|\varphi_x(t)\|_{W^{3, \infty}}^{2n} + \|L \varphi_x(t)\|_{W^{2, \infty}}^{2n}\Big) < \infty,
\nonumber
\end{align}
for all $t\in [0,T]$.
In \eqref{apest}, $\tilde{E}^{(s)}$ is defined in \eqref{weighted_energy1}, and $F(\cdot)$ is an increasing, continuous, non-negative function such that
\begin{equation}
F\left(\|\varphi_x\|_{W^{3, \infty}} + \|L\varphi_x\|_{W^{2, \infty}}\right)\approx \sum_{n = 0}^\infty \tilde{C}^n |c_n| \Big(\|\varphi_x\|_{W^{3, \infty}}^{2n} + \|L \varphi_x\|_{W^{2, \infty}}^{2n}\Big).
\label{defF}
\end{equation}
\end{proposition}

This result follows from the more general result in Proposition~\ref{linearEstimate}. In order to prove that proposition, we first state a lemma.
\begin{lemma}\label{lem-dfT} Suppose that $s \ge 5$ is an integer. If $\vp\in H^s(\R)$ is a solution of \eqref{fsqgeq} and $\psi\in C_t^1L_x^2$, then
\[
\pt (2 - T_{B^{\log}[\vp]})^s \psi = (2 - T_{B^{\log}[\vp]})^s \psi_t - s (2 - T_{B^{\log}[\vp]})^{s - 1} T_{\pt B^{\log}[\vp]} \psi + \Rc_2(\psi),
\]
where the remainder term satisfies
\[
\|\Rc_2(\psi)\|_{H^1} \lesssim \|\psi\|_{L^2} \bigg(\sum_{n = 1}^\infty C(n,s)  |c_n| \Big(\|\varphi_x\|_{W^{2, \infty}}^{2n} + \| \varphi_{xt}\|_{W^{1, \infty}}^{2n}\Big)\bigg)
\]
for constants $C(n,s)$ with at most exponential growth in $n$.
\end{lemma}
\begin{proof}
Since $s$ is an integer, we can calculate the time derivative as
\begin{align*}
\pt (2 - T_{B^{\log}[\vp]})^s \psi &= T_{\pt B^{\log}[\vp]} (2 - T_{B^{\log}[\vp]})^{s - 1} \psi + (2 - T_{B^{\log}[\vp]}) T_{\pt B^{\log}[\vp]} (2 - T_{B^{\log}[\vp]})^{s - 2} \psi\\
&\hspace{1in} + \dotsb + (2 - T_{B^{\log}[\vp]})^{s - 1} T_{\pt B^{\log}[\vp]} \psi + (2 - T_{B^{\log}[\vp]})^{s} \psi_t.
\end{align*}
By Lemma \ref{lem:taest},
\[
\left\|[T_{\pt B^{\log}[\vp]}, (2 - T_{B^{\log}[\vp]})]f\right\|_{L^2}\lesssim \|f\|_{L^2} \bigg(\sum_{n = 1}^\infty C(n,s) |c_n| \Big(\|\varphi_x\|_{W^{1, \infty}}^{2n} + \| \varphi_{xt}\|_{W^{1, \infty}}^{2n}\Big)\bigg).
\]

By taking $f = (2-T_{B^{\log}[\vp]})^{s-2}\psi, (2-T_{B^{\log}[\vp]})^{s-3}\psi, \dotsc, (2-T_{B^{\log}[\vp]})\psi$ and applying the above estimate repeatedly, we obtain the conclusion.
\end{proof}

We then have the following estimate for a linearization of \eqref{fsqgeq}.
\begin{proposition}\label{linearEstimate}
Assume $s\geq 5$, $\vp_0\in H^s(\R)$, and $\Upsilon\in C^{}([0,T];H^{s}(\R))$. Suppose that $u_x \in C([0,T]; W^{3, \infty}(\R))$, $u_t\in C([0,T]; W^{3,\infty}(\R))$, and  $Lu_x\in C([0,T]; W^{3, \infty}(\R))$ satisfy
\[
\big\|T_{B^{\log}[u(t)]}\|_{L^2 \to L^2} < 2,\qquad
\sum_{n = 1}^\infty \tilde C^n |c_n| \Big(\|u_x(t)\|_{W^{3, \infty}}^{2n} + \|L u_x(t)\|_{W^{2, \infty}}^{2n}+\|u_{tx}(t)\|_{W^{1, \infty}}^{2n}\Big) < \infty
\qquad
\text{for all $t \in [0, T]$.}
\]
Consider the initial value problem
\begin{align*}
\vp_t + \px T_{B^0[u]} \vp + \Upsilon(x,t) = L \big[(2 - T_{B^{\log}[u]}) \vp\big]_x, \quad  \vp(x,0)=\vp_0(x).
\end{align*}
Then this linear problem has a unique solution $\vp\in C([0,T]; H^s(\R))$,
and the linearized energy
\begin{align}
\label{uenergy}
E^{(s)}_u=\int_\R |D|^s \varphi(x, t) \cdot \Big(2 - T_{B^{\log}[u]}\Big)^{2s + 1} |D|^s \vp(x, t) \diff{x},
\qquad\tilde{E}^{(s)}_u(t) = \sum_{j=0}^s E^{(j)}_u(t),
\end{align}
satisfies
\begin{align}
\label{apest1}
\begin{split}
\tilde{E}^{(s)}_u(t)  &\le \tilde{E}^{(s)}_u(0)
+ \int_0^t  F\left(\|u_x(\tau)\|_{W^{3,\infty}} + \|u_{tx}(\tau)\|_{W^{1,\infty}}+\|L u_x(t)\|_{W^{2, \infty}}\right)
\\
&\qquad\cdot \biggl\{
\left(\|u_x(\tau)\|_{W^{3, \infty}} + \|u_{tx}(\tau)\|_{W^{1, \infty}}+\|L u_x(t)\|_{W^{2, \infty}}\right)^2
\|\vp(\tau)\|_{H^s}^2
+ \|\Upsilon(\tau)\|_{H^s}\|\vp(\tau)\|_{H^s}\biggr\} \diff{\tau},
\end{split}
\end{align}
where $F$ is an increasing, continuous, non-negative function that satisfies
\begin{equation}
F\left(\|u_x(\tau)\|_{W^{3,\infty}} + \|u_{tx}(\tau)\|_{W^{1,\infty}}+\|L u_x(t)\|_{W^{2, \infty}}\right)\approx \sum_{n = 0}^\infty \tilde{C}^n |c_n| \Big(\|u_x\|_{W^{3, \infty}}^{2n} + \|u_{tx}\|_{W^{1, \infty}}^{2n}+ \|Lu_{x}(\tau)\|_{W^{2, \infty}}^{2n}\Big).
\label{defFu}
\end{equation}
for some constant $\tilde{C} > 0$ depending only on $s$.
\end{proposition}

\begin{proof}
We apply the operator $|D|^s$ to equation \eqref{fsqgeq} to get
\beq\label{Dseqn}
|D|^s \vp_t + \px |D|^s T_{B^0[u]} \vp + |D|^s \Upsilon = \px L |D|^s \big[(2 - T_{B^{\log}[u]}) \vp\big].
\eeq
Using Lemma \ref{lem-DsT}, we find that
\[
\begin{aligned}
|D|^s\left[(2 - T_{B^{\log}[u]}) \vp\right]&=2|D|^s\vp-|D|^s(T_{B^{\log}[u]} \vp)\\
&=2|D|^s\vp-T_{B^{\log}[u]} |D|^s\vp
- sT_{\px B^{\log}[u]} |D|^{s-2} \vp_x
+ \Rc_{2},
\end{aligned}
\]
where
\[
\|\px \Rc_{2}\|_{L^2} \lesssim \left(\sum_{n = 1}^\infty C(n,s) |c_n| \|u_x\|_{W^{3, \infty}}^{2n}\right) \|\vp\|_{H^{s-1}}.
\]
Thus, we can write the right-hand side of \eqref{Dseqn} as
\[
\begin{aligned}
\px L |D|^s & \left[(2 - T_{B^{\log}[u]}) \vp\right]\\
& = \px L \left[(2 - T_{B^{\log}[u]}) |D|^s\vp - s T_{\px B^{\log}[u]} |D|^{s-2} \vp_x\right]+\Rc_{3}\\
& =L \left\{(2 - T_{B^{\log}[u]}) |D|^s\vp_x - T_{\px B^{\log}[u]} |D|^s\vp + s T_{\px B^{\log}[u]} |D|^{s} \vp \right\} + \Rc_{3}\\
&= L\left\{(2 - T_{B^{\log}[u]}) |D|^s\vp_x + (s - 1) T_{\px B^{\log}[u]} |D|^s \vp \right\} + \Rc_{3},
\end{aligned}\]
where
\[
\|\Rc_3\|_{L^2}\lesssim \bigg(\sum_{n = 1}^\infty C(n,s) |c_n| \|u_x\|_{W^{3, \infty}}^{2n} \bigg) \|\vp\|_{H^{s-1}}.
\]

Applying $(2 - T_{B^{\log}[u]})^s$ to \eqref{Dseqn} and
commuting $(2 - T_{B^{\log}[u]})^s$ with $L$ up to remainder
terms, we obtain that
\begin{equation}
\label{JDseqn}
\begin{aligned}
& (2 - T_{B^{\log}[u]})^s |D|^s \vp_t + (2 - T_{B^{\log}[u]})^s \px |D|^s T_{B^0[\vp]} \vp+(2 - T_{B^{\log}[u]})^s |D|^s\Upsilon\\
&\qquad = L\left\{(2 - T_{B^{\log}[u]})^{s + 1} |D|^s\vp_x + (s - 1) (2 - T_{B^{\log}[u]})^s T_{\px B^{\log}[u]} |D|^s \vp \right\} + \Rc_{4}\\
&\qquad=  \px L\left\{(2 - T_{B^{\log}[u]})^{s + 1} |D|^s \vp\right\} + \Rc_{5},
\end{aligned}
\end{equation}
where $\|\Rc_{5}\|_{L^2} \lesssim \left(\sum_{n = 1}^\infty C(n,s) |c_n| \|u_x\|_{W^{3, \infty}}^{2n}\right) \|\vp\|_{H^s}$.

By Lemma \ref{lem-dfT}, with $\psi = |D|^s\vp$, the time derivative of $E^{(s)}_u(t)$ in \eqref{uenergy} is
\bel{dtE}\begin{aligned}
\frac{\diff}{\diff{t}} E^{(s)}_u(t) & = - \int_{\R} (2s+1) |D|^s\vp\cdot (2 - T_{B^{\log}[u]})^{2s} T_{\pt B^{\log}[u]} |D|^s\vp\diff{x}\\
& + 2\int_{\R}|D|^s\vp \cdot (2 - T_{B^{\log}[u]})^{2s+1} |D|^s\vp_t\diff{x} + \int_{\R}\Rc_2 \cdot |D|^s\vp \diff{x}.
\end{aligned}\eeq
We will estimate each of the terms on the right-hand side of \eqref{dtE}, where the second term requires the most work.

The first term on the right-hand side of \eqref{dtE} can be estimated by
\[
\begin{aligned}
&\left| \int_{\R} (2s+1)|D|^s\vp \cdot (2 - T_{B^{\log}[u]})^{2s} T_{\pt B^{\log}[u]} |D|^s\vp\diff{x} \right|\\
&\qquad\lesssim  \bigg(\sum_{n = 1}^\infty C(n,s) |c_n| \Big(\|u_x\|_{W^{1, \infty}}^{2n} + \| u_{tx}\|_{W^{1, \infty}}^{2n}\Big)\bigg) \|\vp\|_{H^s}^2.
\end{aligned}
\]
Using Lemma~\ref{lem-dfT}, we can estimate the third term  on the right-hand side of \eqref{dtE} by
\begin{align*}
 \int_{\R}\Rc_2 \cdot |D|^s\vp \diff{x} \lesssim  \bigg(\sum_{n = 1}^\infty C(n,s) |c_n| \Big(\|u_{x}\|_{W^{2, \infty}}^{2n}+\|u_{tx}\|_{W^{1, \infty}}^{2n}  \Big)\bigg) \|\vp\|_{H^s} \|\vp\|_{H^{s - 1}}.
 \end{align*}

To estimate the second term  on the right-hand side \eqref{dtE}, we multiply \eqref{JDseqn} by $(2 - T_{B^{\log}[u]})^{s+1} |D|^s \vp$, integrate the result with respect to $x$, and use the self-adjointness of
$(2 - T_{B^{\log}[u]})^{s + 1}$, which gives
\[
\int_{\R} |D|^s \vp \cdot (2 - T_{B^{\log}[u]})^{2s + 1} |D|^s \vp_t \diff{x}
= \Rm{1} + \Rm{2} + \Rm{3}+\Rm{4},
\]
where
\begin{align*}
\Rm{1} &= - \int_{\R} |D|^s\vp\cdot (2 - T_{B^{\log}[u]})^{2s+1} |D|^s \partial_x T_{B^{0}[\vp]} \vp \diff{x},
\\
\Rm{2} &= \int_\R (2 - T_{B^{\log}[u]})^{s+1} |D|^s \vp \cdot \px L (2 - T_{B^{\log}[u]})^{s+1} |D|^s \vp \diff{x},
\\
\Rm{3}&=\int _\R(2 - T_{B^{\log}[u]})^{s+1} |D|^s \vp \cdot \Rc_{5} \diff{x},\\
\Rm{4}&=\int_{\R}  |D|^s \vp \cdot (2 - T_{B^{\log}[u]})^{2s + 1} |D|^s \Upsilon \diff{x}.
\end{align*}
We have $\Rm{2}=0$, since $\partial_x L$ is skew-symmetric, and
\[
\Rm{3} \lesssim  \bigg(\sum_{n = 1}^\infty C(n,s) |c_n| \|u_x\|_{W^{3, \infty}}^{2n}\bigg) \|\vp\|_{H^s}^2,
\qquad \Rm{4} \lesssim  F( \|u_x\|_{W^{3, \infty}}) \|\vp\|_{H^s}\|\Upsilon\|_{H^s},
\]
since $\|\Rc_{5}\|_{L^2} \lesssim  \left(\sum_{n = 1}^\infty C(n,s) |c_n| \|u_x\|_{W^{3, \infty}}^{2n} \right)\|\vp\|_{H^s}$ and $(2 - T_{B^{\log}[u]})^{s+1}$ is bounded on $L^2$.

\noindent {\bf Term} $\Rm{1}$ {\bf estimate.} We write  $\Rm{1} = -\Rm{1}_a + \Rm{1}_b$, where
\begin{align*}
\Rm{1}_a & =\int_{\R}|D|^s\vp\cdot (2 - T_{B^{\log}[u]})^{2s+1} \partial_x  T_{B^0[u]} |D|^s \vp \diff{x},\\
 \Rm{1}_b&= \int_{\R} |D|^s\vp\cdot (2 - T_{B^{\log}[u]})^{2s+1} \partial_x [T_{B^0[u]}, |D|^s] \vp \diff{x}.
\end{align*}
By a commutator estimate and \eqref{B-est}, the second integral satisfies
\[
|\Rm{1}_b| \lesssim  \bigg(\sum_{n = 1}^\infty C(n,s) |c_n| \Big(\|u_x\|_{W^{2, \infty}}^{2n} + \|Lu_{x}\|_{W^{2, \infty}}^{2n}\Big)\bigg) \|\vp\|_{H^s}^2.
\]
To estimate the first integral, we write it as
\begin{align*}
\Rm{1}_a &= \Rm{1}_{a_1} - \Rm{1}_{a_2},
\end{align*}
where
\begin{align*}
\Rm{1}_{a_1}& = \int_{\R} |D|^s \vp \cdot [(2 - T_{B^{\log}[u]})^{2s + 1}, \partial_x] \left(T_{B^{0}[u]} |D|^s \vp\right) \diff{x},
\\
\Rm{1}_{a_2}  &= \int_{\R} |D|^s\vp_x \cdot(2 - T_{B^{\log}[u]})^{2s+1} \left(T_{B^0[u]} |D|^s\vp\right) \diff{x}.
\end{align*}

\noindent {\bf Term} $\Rm{1}_{a_1}$ {\bf estimate.} A Kato-Ponce type commutator estimate and \eqref{B-est} gives
\[
|\Rm{1}_{a_1}| \lesssim  \bigg(\sum_{n = 1}^\infty C(n,s) |c_n| \Big(\|u_x\|_{W^{2, \infty}}^{2n} + \|Lu_{x}\|_{W^{1, \infty}}^{2n}\Big)\bigg) \|\vp\|_{H^s}^2.
\]

\noindent {\bf Term} $\Rm{1}_{a_2}$ {\bf estimate.} We have
\bel{eqBs}\begin{aligned}
\Rm{1}_{a_2}
&=\int_{\R} \left(T_{B^0[u]} |D|^s \vp\right) \cdot \left\{\partial_x\left((2 - T_{B^{\log}[u]})^{2s + 1} |D|^s \vp\right)-\left[\partial_x, (2 - T_{B^{\log}[u]})^{2s + 1}\right] |D|^s \vp \right\}\diff{x}\\
&=- \int_{\R} \partial_x \left(T_{B^0[u]} |D|^s \vp\right) \cdot (2 - T_{B^{\log}[u]})^{2s + 1} |D|^s \vp \diff{x}\\
& \qquad - \int_\R \left(T_{B^0[u]} |D|^s \vp\right) \cdot \left[\partial_x, (2 - T_{B^{\log}[u]})^{2s + 1}\right] |D|^s \vp \diff{x}\\
&=- \int_{\R} \left(T_{B^0[u]} |D|^s \vp_x + \left[\partial_x, T_{B^0[u]}\right] |D|^s \vp\right) \cdot (2 - T_{B^{\log}[u]})^{2s + 1} |D|^s \vp \diff{x}\\
& \qquad - \int_\R \left(T_{B^0[u]} |D|^s \vp\right) \cdot \left[\partial_x, (2 - T_{B^{\log}[u]})^{2s + 1}\right] |D|^s \vp \diff{x}.
\end{aligned}\eeq
Using commutator estimates and \eqref{B-est}, we get that
\[
\begin{aligned}
\left\|\left[\partial_x,T_{B^0[u]}\right] |D|^s\vp\right\|_{L^2} &\lesssim \bigg(\sum_{n = 1}^\infty C(n,s) |c_n| \left(\|u_x\|_{W^{2,\infty}}^2 + \|Lu_{x}\|_{W^{2,\infty}}^2\right)\bigg) \|\vp\|_{H^s},\\
\left\|\left[\partial_x,(2 - T_{B^{\log}[u]})^{2s+1}\right] |D|^s\vp\right\|_{L^2} &\lesssim  \bigg(\sum_{n = 1}^\infty C(n,s) |c_n| \Big(\|u_x\|_{W^{2, \infty}}^{2n} \Big)\bigg) \|\vp\|_{H^s},\\
\left\|\partial_x \left[(2 - T_{B^{\log}[u]})^{2s+1}, T_{B^0[u]}\right]|D|^s\vp\right\|_{L^2} &\lesssim  \bigg(\sum_{n = 1}^\infty C(n,s) |c_n| \Big(\|u_x\|_{W^{2, \infty}}^{2n} + \|Lu_{x}\|_{W^{2, \infty}}^{2n}\Big)\bigg) \|\vp\|_{H^s}.
\end{aligned}
\]
Since $T_{B^0[\vp]}$ is self-adjoint, we can rewrite \eqref{eqBs} as
\[
\begin{aligned}
\Rm{1}_{a_2}
& = - \Rm{1}_{a_2} + \Rc_{6},
\end{aligned}\]
with
\[
|\Rc_{6}| \lesssim  \bigg(\sum_{n = 1}^\infty C(n,s) |c_n| \Big(\|u_x\|_{W^{2, \infty}}^{2n} + \|Lu_{x}\|_{W^{2, \infty}}^{2n}\Big)\bigg) \|\vp\|_{H^s}^2,
\]
and we conclude that
\[
|\Rm{1}_{a_2}| \lesssim  \bigg(\sum_{n = 1}^\infty C(n,s) |c_n| \Big(\|u_x\|_{W^{2, \infty}}^{2n} + \|Lu_{x}\|_{W^{2, \infty}}^{2n}\Big)\bigg) \|\vp\|_{H^s}^2.
\]

This completes the estimate of the terms on the right hand side of \eqref{dtE}. Collecting the above estimates and using the interpolation inequalities for $E^{(0)}_u$ and $E^{(s)}_u$, we obtain \eqref{apest1}

We observe that there exists a constant $\tilde C(s) > 0$ such that $C(n,s)\lesssim \tilde C(s)^n$. The series in \eqref{defF} then converges whenever $\|u_x\|_{W^{3, \infty}} + \|u_{tx}\|_{W^{1, \infty}}+ \|Lu_{x}(\tau)\|_{W^{2, \infty}}$ is sufficiently small, and we can choose $F$ to be an increasing, continuous, non-negative function that satisfies \eqref{defFu}.
\end{proof}

For $s\geq 5$, define a map
\[
\G \colon C([0,T]; H^s(\R))\to C([0,T]; H^s(\R))
\]
by $\G(u)=\vp$ where
\begin{align*}
\vp_t + \px T_{B^0[u]} \vp + \Rc(u)  = L \big[(2 - T_{B^{\log}[u]}) \vp\big]_x, \quad  \vp(x,0)=\vp_0(x),
\end{align*}
with the same $\Rc(\cdot)$ as the one in \eqref{fsqgeq}.
We will prove that $\G$ is a contraction mapping for sufficiently small $T > 0$, which implies the existence and uniqueness of
local solutions of the initial value problem for \eqref{fsqgeq}.

We first prove the following proposition.
\begin{proposition}[Boundedness]\label{bdd}
Assume $\vp_0$ satisfies the assumptions of Proposition~\ref{apriori} and $\|\vp_0\|_{H^s}\leq \bar C$ for some $\bar C > 0$. Define
\begin{align*}
X_T &= \bigg\{u\in C([0,T]; H^s(\R)) ~\big|~ \|u\|_{L^\infty_t(0,T;H^s(\R))}\leq 2\bar C,\quad \big\|T_{B^{\log}[u(t)]}\|_{L^2 \to L^2} < 2,\\
& \qquad \sum_{n = 1}^\infty \tilde C^n |c_n| \Big(\|u_x(t)\|_{W^{3, \infty}}^{2n} + \|L u_x(t)\|_{W^{2, \infty}}^{2n}+ \|u_{tx}(t)\|_{W^{1, \infty}}^{2n}\Big) < \infty,
\forall t \in [0, T].\bigg\}.
\end{align*}
 Then there exists $T > 0$, such that $\G \colon X_T \to  X_T$.
\end{proposition}
\begin{proof}
Taking $\Upsilon = \Rc(u)$ in Proposition~\ref{linearEstimate} and using \eqref{Rcest}, we obtain that
\begin{align}
\begin{split}
\tilde{E}^{(s)}_u(t) & \le \tilde{E}^{(s)}_u(0) + \int_0^t \big(\|u_x(\tau)\|_{W^{3, \infty}} + \|u_{tx}(\tau)\|_{W^{1, \infty}}+ \|Lu_{x}(\tau)\|_{W^{2, \infty}}\big)^2\\
& \hspace{.75in}\cdot F\left(\|u_x(\tau)\|_{W^{3,\infty}} + \|u_{tx}(\tau)\|_{W^{1,\infty}}+ \|Lu_{x}(\tau)\|_{W^{2, \infty}}\right) (\|\vp(\tau)\|_{H^s}^2+\|u(\tau)\|_{H^s}^2)\diff{\tau},
\end{split}
\end{align}
where $F$ is a positive continuous function.
Since $\|\vp_0\|_{H^s}\leq\bar C$, and $\|\vp(\cdot, t)\|_{H^s}\approx \left[\tilde E_{u}^{(s)}(t)\right]^{1/2}$, $\|u(t)\|_{H^s}$ are continuous in time,  there exists $T>0$ such that $\G(u)=\vp\in X_T$.
\end{proof}

Now, we prove Proposition~\ref{apriori}.
\begin{proof}[Proof of Proposition~\ref{apriori}]
We take $u=\vp$ and $\Upsilon = \Rc(\vp)$ in Proposition~\ref{linearEstimate}.
Since $\|2 - T_{B^{\log}[\vp_0]}\|_{L^2 \to L^2} \ge 2-C$,  and $\|B^{\log}[\vp](\cdot, t)\|_{\Mc_{(2,1)}}$
and $F\left(\|\varphi_x\|_{W^{3, \infty}} + \|L\varphi_x\|_{W^{2, \infty}}\right)$ are continuous in time, there exist $\Time>0$ and $m>0$, depending only on the initial data, such that
\[
\|2-T_{B^{\log}[\vp(t)]}\|_{L^2 \to L^2} \geq m \qquad \text{for $0\le t\le\Time$}.
\]
We therefore obtain that
\[
m^{2s+1}\|\vp\|_{H^s}^2\leq \tilde{E}^{(s)}\leq 2^{2s+1} \|\vp\|_{H^s}^2,
\]
so, using the remainder estimate \eqref{Rcest}, we get \eqref{apest} and \eqref{defF} from \eqref{apest1} and \eqref{defFu}.
\end{proof}

We construct a sequence of approximate solutions by
\[
\vp^{(0)}(x,t)=\vp_0(x),\qquad \vp^{(i)}=\G(\vp^{(i-1)})\quad \text{for $i\in \N$}.
\]
By Proposition \ref{bdd}, the set $\{\vp^{(i)} \mid i\in \N\}$ is bounded in $X_T$. We prove the convergence of this sequence by showing that $\G$ is a contraction mapping with respect to a low norm.

\begin{proposition}[Contraction]
\label{Contraction}
For sufficiently small $T > 0$,
$\G \colon X_T\to X_T$ defined above is a contraction mapping with respect to $\|\cdot\|_{L^\infty_t H^3_x}$.
\end{proposition}
\begin{proof}
For $u, v \in W$, let $\vp$ and $\psi$ be solutions of the equations
\begin{align*}
&\vp_t + \px T_{B^0[u]} \vp + \Rc(u) = L \big[(2 - T_{B^{\log}[u]}) \vp\big]_x,
\\
&\psi_t + \px T_{B^0[v]} \psi + \Rc(v) = L \big[(2 - T_{B^{\log}[v]}) \psi\big]_x,
\end{align*}
with the same initial data.
Taking their difference, we have
\begin{align*}
(\vp-\psi)_t+\px T_{B^0[u]} (\vp-\psi) &= L \big[(2 - T_{B^{\log}[u]}) (\vp-\psi)\big]_x+\px (T_{B^0[v]}-T_{B^0[u]}) \psi \\
& \qquad - L \big[(T_{B^{\log}[u]} - T_{B^{\log}[v]}) \psi\big]_x + \Rc(v)-\Rc(u).
\end{align*}
Applying Proposition~\ref{linearEstimate} with
\[
\Upsilon=-\px (T_{B^0[v]}-T_{B^0[u]}) \psi + L \big[(T_{B^{\log}[u]} + T_{B^{\log}[v]}) \psi\big]_x - \Rc(v)+\Rc(u),
\]
we obtain that, for $k\leq 3$,
\begin{align*}
 &\frac{\diff}{\diff t}\int_{\R} |D|^k(\vp-\psi)(2 - T_{B^{\log}[u]})^{2k+1}|D|^k(\vp-\psi)\diff x \\
 &\lesssim (\|u_x\|_{W^{3, \infty}} + \|u_{tx}\|_{W^{1, \infty}}+\|Lu_x\|_{W^{2, \infty}})^2 F_1(\|u_x\|_{W^{3, \infty}} + \|u_{tx}\|_{W^{1, \infty}}+\|Lu_x\|_{W^{2, \infty}})\|\vp-\psi\|_{H^k}^2 \\*
& + F_2(\|u_x\|_{W^{3, \infty}} + \|u_{tx}\|_{W^{1, \infty}}+\|Lu_x\|_{W^{2, \infty}}+\|v_x\|_{W^{3, \infty}} + \|v_{tx}\|_{W^{1, \infty}}+\|Lv_x\|_{W^{2, \infty}}) \|u-v\|_{H^4} \|\vp-\psi\|_{H^k}\|\psi\|_{H^{k+2}},
\end{align*}
where $F_1, F_2$ are two positive continuous functions.
Since $\vp=\psi$ at $t = 0$, we have, by Gr\"{o}nwall's inequality,
\[
\|(\vp-\psi)(t)\|_{H^k} \lesssim \int_0^t  e^{\int_\tau^t F_1(\cdots)(s)\diff s} F_2\left(\cdots\right)(\tau)\|(u-v)(\cdot, \tau)\|_{H^3}\|\psi(\cdot, \tau)\|_{H^{k+2}} \diff{\tau},
\]
where $(\cdots)$ are the same arguments as above.
By the $H^s$-energy estimate with $s \geq 5$, the function
\[
e^{\int_\tau^t F_1(\cdots)(s)\diff s} F_2\left(\cdots\right)(\tau)\|\psi(\cdot, \tau)\|_{H^{k+2}}
\]
is bounded on $[0,T]$ for $k\leq 3$. Thus, by taking $T$ small enough, we deduce that
\[
\|\vp-\psi\|_{L_t^\infty H_x^{3}}\leq \lambda \|u-v\|_{L_t^\infty H_x^3}
\]
for some $0 < \lambda<1$.
\end{proof}
So for any $s\geq 5$ and $\vp_0\in H^s(\R)$, $\{\vp^{(0)}, \vp^{(1)}, \dotsc, \vp^{(i)}, \dotsc\}$ is a bounded sequence in $C([0,T]; H^s(\R))$. We claim that this is a Cauchy sequence with respect to $\|\cdot\|_{L^\infty_tH^3_x}$. In fact, for any $0<\ve<1$, there is a positive integer $N=\log_\lambda \frac{(1-\lambda)\ve}{3\bar C}$ such that, for any $i>j>N$,
\begin{align*}
\|\vp^{(i)}-\vp^{(j)}\|_{L^\infty_tH^3_x}&\leq \|\vp^{(i)}-\vp^{(i-1)}\|_{L^\infty_tH^3_x}+\dotsb+\|\vp^{(j+1)}-\vp^{(j)}\|_{L^\infty_tH^3_x}\\
&\leq (\lambda^{i-N-1}+\dotsb+\lambda^{j-N})\|\vp^{(N+1)}-\vp^{(N)}\|_{L^\infty_tH^3_x}\\
&\leq (\lambda^{i-N-1}+\dotsb+\lambda^{j-N})\lambda^{N}\|\vp^{(1)}-\vp^{(0)}\|_{L^\infty_tH^3_x}\\
&\leq \frac{3\bar C}{1-\lambda}\lambda^N<\ve.
\end{align*}
So $\lim_{j\to\infty} \vp^{(j)}$ exists and is unique. The regularity of the solution follows from the {\it a priori} estimate.
Therefore, we obtain the following existence theorem with a blow-up criterion.
\begin{theorem}
\label{th:loc_exist}
Let $s \ge 5$ be an integer. Suppose that $\vp_0 \in H^s(\R)$ satisfies
\[
\|T_{B^{\log}[\vp_0]}\|_{L^2 \to L^2} \leq C,\qquad
\sum_{n = 1}^\infty \tilde{C}^n |c_n| \Big(\|\px\varphi_0\|_{W^{3, \infty}}^{2n} + \|L\px \varphi_0\|_{W^{2, \infty}}^{2n}\Big) < \infty
\]
for some constant $0 < C < 2$, where $\tilde C$ is the same constant as the one in Proposition~\ref{apriori} and the symbol $B^{\log}[\vp_0]$ is defined in \eqref{defBlog}.
Then there exists a maximal time of existence $0 < \Time_{\max} \le \infty$ depending only on $\|\vp_0\|_{H^s}$,  $C$, and $\tilde{C}$ such that the initial value problem \eqref{fsqgivp} has a unique solution with $\vp \in C([0, \Time_{\max}); H^s(\R))$. If $\Time_{\max} < \infty$, then  either
\begin{equation*}
\lim_{t \uparrow
 \Time_{\max}} \sum_{n = 1}^\infty \tilde{C}^n |c_n| \Big(\|\varphi_x(t)\|_{W^{3, \infty}}^{2n} + \|L \varphi_x(t)\|_{W^{2, \infty}}^{2n}\Big) = \infty \quad \text{or}\quad \lim_{t \uparrow \Time_{\max}} \big\|T_{B^{\log}[\vp(\cdot, t)]}\big\|_{L^2 \to L^2} = 2.
\end{equation*}
\end{theorem}

We remark that, by interpolation, one can also replace $\|\varphi_x(t)\|_{W^{3, \infty}}^{2n} + \|L \varphi_x(t)\|_{W^{2, \infty}}^{2n}$ by $\|\partial_x^{4}\varphi(t)\|_{L^\infty}^{2n} + \|L \varphi_x(t)\|_{L^\infty}^{2n}$.

The front equation is invariant under $(x,t)\mapsto(-x,-t)$, so the same result holds backward in time.
One could use a Bona-Smith argument, as in \cite{HSZ1}, to prove that the solution depends continuously on the initial data, but we will not carry out the details here.

\section{Global solution for small initial data}
\label{sec-global}

Beginning with this section, we address the global well-posedness of \eqref{fsqgivp} with small initial data. From now on, we fix the following parameter values
\begin{equation}\label{param_vals}
s = 1200,\qquad r = 1,\qquad  p_0 = 10^{-4}.
\end{equation}
The front equation \eqref{fsqgivp} is invariant under the transformation
\[
x\mapsto \lambda(x + 2\log |\lambda| t),\qquad t\mapsto \lambda t,\qquad \vp\mapsto \lambda\vp.
\]
The scaling-Galilean part of this transformation is generated by the
vector-field
\begin{equation}\label{defS}
\S=(x+2t)\partial_x+t\partial_t,
\end{equation}
and the linearized equation $\vp_t = 2 \log|\px| \varphi_x$ commutes with $\S$ (\emph{cf.} Lemma~\ref{S-comm}).
We also introduce the notation
\begin{equation}
\label{Lambda-h}
\begin{aligned}
h(x,t) = e^{-2t\px \log|\px|}\vp(x,t),\qquad  \hat{h}(\xi,t) &= e^{-2it\xi\log|\xi| } \hat\vp(\xi,t)
\end{aligned}
\end{equation}
for the function $h$ obtained by removing the action of the linearized evolution group on $\vp$.
When convenient, we write $h(\cdot,t) = h(t)$, $\vp(\cdot,t) = \vp(t)$.

Our global existence theorem is as follows.

\begin{theorem}\label{global}
Let $s$, $r$, $p_0$ be defined as in \eqref{param_vals}. There exists a constant $0<\ve \ll 1$, such that if $\vp_0 \in H^s(\R)$ satisfies
\[
\|\vp_0\|_{H^s}  + \|x \px \vp_0\|_{H^r} \leq \ve_0
\]
for some $0 < \ve_0 \leq \ve$, then there exists a unique global solution $\vp\in C([0,\infty); H^s(\R))$ of \eqref{fsqgivp}.
Moreover, this solution satisfies
\[
 \|\vp(t)\|_{H^s}+\|\S\vp(t)\|_{H^{r}}  \lesssim \ve_0(t+1)^{p_0},
\]
where $\S$ is the vector field in \eqref{defS}.
\end{theorem}

Given local existence, we only need to prove the global \emph{a priori} bound. In order to do this, we introduce the $Z$-norm of a function $f\in L^2(\R)$, defined by
\begin{equation}
\|f\|_{Z} = \left\|(|\xi|+|\xi|^{r + 4}) \hat f(\xi)\right\|_{L^\infty_\xi},
\label{defZ}
\end{equation}
and prove the global bound by use of the following bootstrap argument.

\begin{proposition}[Bootstrap]\label{bootstrap}
Let $T>1$ and suppose that $\vp\in C([0,T]; H^s)$ is a solution of \eqref{fsqgivp}, where the initial data satisfies
\[
\|\vp_0\|_{H^s} + \|x \px \vp_0\|_{H^r}\leq \ve_0
\]
for some $0 < \ve_0 \ll 1$. If there exists $\ve_0 \ll \ve_1 \lesssim \ve_0^{1/3}$ such that the solution satisfies
\[
(t+1)^{-p_0}\left( \|\vp(t)\|_{H^s}+\|\S\vp(t)\|_{H^{r}} \right)+\|\vp\|_{Z}\leq \ve_1
\]
for every $t\in [0,T]$, then the solution satisfies an improved bound
\[
(t+1)^{-p_0}\left(\|\vp(t)\|_{H^s}+\|\S\vp(t)\|_{H^{r}} \right)+ \|\vp\|_{Z} \lesssim\ve_0.
\]
\end{proposition}

Theorem \ref{global} then follows from combining this bootstrap proposition with the local existence and blow-up result in Theorem~\ref{th:loc_exist}. We call the assumptions in Proposition \ref{bootstrap} the \emph{bootstrap assumptions}. To prove Proposition \ref{bootstrap}, we need the following lemmas, some of whose proofs are deferred to the next sections.

\begin{lemma}[Sharp pointwise decay]\label{sharp}
Under the bootstrap assumptions,
\[
\||\px|^{r + 2}\vp_x(t)\|_{L^{\infty}}+\|L\vp_x(t)\|_{L^{\infty}}\lesssim \ve_1(t+1)^{-1/2}.
\]
\end{lemma}

\begin{lemma}[Scaling-Galilean estimate]\label{weightedE}
Under the bootstrap assumptions,
\[
(t + 1)^{-p_0}\|\S\vp(t)\|_{H^{r}}\lesssim \ve_0.
\]
\end{lemma}

\begin{lemma}\label{lem:xdh}
Under the bootstrap assumptions,
\[
(t + 1)^{-p_0} (\|\vp(t)\|_{H^s} + \|x \px \vp(t)\|_{H^r}) \lesssim \ve_0.
\]
\end{lemma}
\begin{proof}
Recall the energy estimate \eqref{apest}
\[
\tilde E^{(s)}(t)\lesssim \tilde E^{(s)}(0) e^{\int_0^t F(\|\vp_x(\tau)\|_{W^{3,\infty}} + \|L\vp_x(\tau)\|_{W^{2,\infty}})(\|\vp_x(\tau)\|_{W^{3,\infty}}+\|L\vp_x(\tau)\|_{W^{2,\infty}})^2 \diff{\tau}}.
\]
From Lemma \ref{sharp}, we have
\[\begin{aligned}
F(\|\vp_x(\tau)\|_{W^{3,\infty}} + \|L\vp_x(\tau)\|_{W^{2,\infty}}) &\lesssim 1,\\
\|\vp_x(\tau)\|_{W^{3,\infty}}+\|L\vp_x(\tau)\|_{W^{2,\infty}}&\lesssim (t + 1)^{-1/2} \ve_1,
\end{aligned}
\]
which implies that
\[
\tilde E^{(s)}(t)\lesssim \ve_0^2(t + 1)^{C\ve_1^2}
\]
for some constant $C$, so once $\ve_1^2 \ll p_0$, we have
\[
(t + 1)^{-p_0} \|\vp\|_{H^s} \lesssim \ve_0.
\]

Next, we observe that we can use $\|\S \vp\|_{H^r}$ to control $\|x \px h\|_{H^r}$.
It follows from \eqref{Lambda-h}, the definition of $\S$, and \eqref{Tn-sqg} that
\begin{align}
\label{xDh}
\begin{split}
\F_x[x \px h](\xi) &= - \partial_\xi \big(\xi \hat{h}(\xi)\big) = - \hat{h}(\xi) - \xi \partial_\xi \hat{h}(\xi),\\
\xi \partial_\xi \hat h(\xi,t) &= \xi e^{-2i t\xi\log|\xi|} \left(-2it (\log|\xi| + 1) \hat\vp(\xi,t) + \partial_\xi \hat\vp(\xi,t)\right)\\
&= e^{-2 i t \xi \log|\xi|} \left[\xi \partial_\xi \hat \vp(\xi,t) - (2it\xi - 1) \hat\vp(\xi,t) - t \hat\vp_t(\xi,t) - t \widehat{\mathcal{N}}(\xi, t) - \hat\vp(\xi,t)\right]\\
&= e^{-2 i t \xi \log|\xi|} \left[- \widehat{\S \vp}(\xi,t) - \hat\vp(\xi,t) - t\widehat{\mathcal{N}}(\xi, t)\right],
\end{split}
\end{align}
where $\mathcal{N}$
 denotes the nonlinear term in \eqref{Tn-sqg}, which satisfies the estimate
 \begin{equation}\label{estimate_N}
 \||\px|^j\mathcal{N}\|_{L^2}\lesssim \sum\limits_{n=1}^\infty (\|\vp_x\|_{W^{3,\infty}}^{2n}+\|L\vp_x\|_{W^{2,\infty}}^{2n})\|\vp\|_{H^{j+2}} \qquad \text{for all}\ j = 0, \dotsc, r.
 \end{equation}

By the bootstrap assumptions, Lemma \ref{sharp}, and Lemma \ref{weightedE} we then find that
\[
(t + 1)^{-p_0} \|x \px h(t)\|_{H^r} \lesssim \ve_0,
\]
and the same estimate holds for $\vp$ in view of \eqref{Lambda-h}.
\end{proof}

\begin{lemma}[Nonlinear dispersive estimate]\label{nonDisp}
Under the bootstrap assumptions, the solution of \eqref{fsqgivp} satisfies
\[
\|\vp(t)\|_{Z}\lesssim \ve_0.
\]
\end{lemma}

Proposition \ref{bootstrap} then follows by
combining Lemmas \ref{sharp}--\ref{nonDisp}.

\section{Linear dispersive estimate}
\label{sec-sharp}

In this section, we prove a dispersive estimate for the linearized evolution operator $e^{2t\partial_x\log|\partial_x|}$ defined in \eqref{Lambda-h} and use it to prove Lemma~\ref{sharp}. We recall that $P_k$ and $\tilde{P}_k$ are the frequency-localization operators with symbols $\psi_k$ and $\tilde{\psi}_k$, respectively (see \eqref{defpsik}).

\begin{lemma}\label{disp}
For $t> 0$ and $f \in L^2$, we have the linear dispersive estimates
\begin{align}\label{LocDis}
\|e^{2t\partial_x\log|\partial_x|}P_k f \|_{L^\infty} &\lesssim (t+1)^{-1/2} 2^{k/2}\|\widehat{P_k f }\|_{L^\infty_\xi}+(t+1)^{-3/4}2^{-k/4}\left[\|P_k (x\px f)\|_{L^2}+\|\tilde P_kf \|_{L^2}\right].
\end{align}
\end{lemma}

\begin{proof} Using the inverse Fourier transform, we can write the solution as
\begin{align*}
e^{2t\partial_x\log|\partial_x|}P_k  f
= \int_\R e^{ix\xi+2i(\xi\log|\xi|)t}\psi_k(\xi)\hat f (\xi)\diff{\xi}.
\end{align*}
Since
\begin{equation}\label{exp}
\partial_\xi e^{ix\xi+2i(\xi\log|\xi|)t}=[ix+2it(\log|\xi|+1)]  e^{ix\xi+2i(\xi\log|\xi|)t},
\end{equation}
we can integrate by parts to get
\begin{align*}
\|e^{2t\partial_x\log|\partial_x|}P_k  f \|_{L^\infty}&= \left\|\int_\R e^{ix\xi+2i(\xi\log|\xi|)t}\hat f (\xi)\psi_k(\xi)\diff{\xi}\right\|_{L^\infty}\\
&=\left\|\int_\R \frac1{ix+2it(\log|\xi|+1)} \partial_\xi e^{ix\xi+2i(\xi\log|\xi|)t}\hat f (\xi)\psi_k(\xi)\diff\xi\right\|_{L^\infty}\\
&=\left\|\int_\R e^{ix\xi+2i(\xi\log|\xi|)t} \partial_\xi \left( \frac1{ix+2it(\log|\xi|+1)}\hat f (\xi)\psi_k(\xi)\right)\diff\xi\right\|_{L^\infty}\\
&=\bigg\|\int_\R e^{ix\xi+2i(\xi\log|\xi|)t}\Big(\frac{-2it}{\xi[ix+2it(\log|\xi|+1)]^2}\hat f(\xi) \psi_k(\xi)\\*
&\qquad+\frac1{ix+2it(\log|\xi|+1)}\psi_k(\xi) \partial_\xi \hat f (\xi)
+\frac1{ix+2it(\log|\xi|+1)}\hat f (\xi)\psi'_k(\xi) \Big)\diff\xi \bigg\|_{L^\infty}.
\end{align*}

{\bf 1.} If $|ix+2it(\log|\xi|+1)|\gtrsim (t+1)$, we use \eqref{psi-L2} and get
\begin{align*}
\|e^{2t\partial_x\log|\partial_x|}P_k  f \|_{L^\infty}&\lesssim \frac1{t+1} \int_\R \left|\xi^{-1}\hat f(\xi) \psi_k(\xi)+\psi_k(\xi) \partial_\xi \hat f(\xi) +\hat f (\xi)\psi'_k(\xi) \right| \diff\xi\\
&\lesssim \frac1{t+1} \left[2^{-k}\|\widehat{P_{k} f }\|_{L^2_\xi}+2^{-k/2}\|P_k\F^{-1}(\xi\partial_\xi \hat f )\|_{L^2}+2^{-k/2}\|\tilde P_k f \|_{L^2}\right].
\end{align*}
Then \eqref{LocDis} follows when $(t + 1)^{-1} \lesssim 2^k$. Otherwise, when $t + 1 \lesssim 2^{-k}$, we have
\[
\|e^{2t\partial_x\log|\partial_x|}P_k  f \|_{L^\infty} \lesssim 2^k \|\widehat{P_k f }\|_{L^\infty_\xi} \lesssim (t + 1)^{-1 / 2} 2^{k / 2} \|\widehat{P_k f }\|_{L^\infty_\xi}.
\]

{\bf  2.} Next we prove estimates for the case when $|ix+2it(\log|\xi|+1)|\ll (t+1)$. Let
\[
\xi_0^\pm=\pm e^{-1-{x}/{2t}}
\]
be the solutions of $x+2t(\log |\xi|+1)=0$. Since $\psi_k$ is supported in an annulus with radius around $2^k$, we only need to consider the case when $|\xi_0^\pm|\approx 2^k$ and $\psi_k$ is supported on the neighborhood of the stationary phase point $\xi_0^\pm$. We decompose the integral and estimate it as
\begin{align*}
\left|\int_{\R}  e^{ix\xi+2i(\xi\log|\xi|)t}\hat f (\xi)\psi_k(\xi)\diff\xi\right|
\lesssim \sum\limits_{l\leq k+N}\left[|J^+_l|+|J^-_l|\right],
\end{align*}
with
\[
J^\pm_l=\int_{\R} e^{ix\xi+2i(\xi\log|\xi|)t}\hat f (\xi)\psi_k(\xi)\mathbf 1_\pm(\xi) \psi_l(\xi-\xi_0^\pm)\diff\xi,
\]
where ${\mathbf 1}_\pm$ is the indicator function supported on $\R_\pm$ and $N$ is large enough  that the support of $\psi_k$ is covered by the set $\bigcup_{l\leq k+N}\{\xi \mid \psi_l(\xi-\xi_0^\pm)=1\}$.

When $2^l \leq 2^{k/2}(t+1)^{-1/2}$, we have
\[
\sum_{2^l \leq 2^{k/2}(t+1)^{-1/2}} |J_l^\pm|\lesssim \sum_{2^l \leq 2^{k/2}(t+1)^{-1/2}} 2^l\|\widehat{P_k f } \|_{L^\infty}\leq 2^{k/2}(t+1)^{-1/2}\|\widehat{P_k f} \|_{L^\infty}.
\]
When $2^{k/2}(t+1)^{-1/2}\leq 2^l \leq 2^{k+N}$, since $|\xi-\xi_0|\approx 2^l$ and $|\xi_0|\approx 2^k$, we get the estimate
\[
x+2t(\log|\xi|+1)= 2t\log\left|\frac\xi{\xi_0}\right|\approx 2t\log \bigg|1\pm \frac{2^l}{2^k}\bigg|.
\]
Using \eqref{exp} and integration by parts, we have
\begin{align*}
|J_l^\pm|
&\lesssim \frac{2^{k-l}}{(t+1)}\int_{\R} \big(|\partial_\xi\hat f (\xi)|+2^{-l}|\hat  f (\xi)|\big) \psi_l(\xi-\xi_0^\pm) \diff\xi\\
&\lesssim \frac{2^{k-l}}{(t+1)}\|\hat f \|_{L^\infty}+\frac{2^{k-\frac l2}}{(t+1)}\|\partial_\xi \hat f \|_{L^2.}
\end{align*}
Then we take the sum of $J_l$ over $2^l\geq 2^{k/2}(t+1)^{-1/2}$ to get the estimates \eqref{LocDis}.
\end{proof}

\begin{proof}[Proof of Lemma \ref{sharp}]
After splitting into high-frequency and low-frequency parts, it suffices to bound the terms
\[
\bigg\|\sum\limits_{k\leq 0}P_k L\vp_x\bigg\|_{L^\infty}, \qquad\bigg\|\sum\limits_{k> 0}P_k\partial_x^{3+\epsilon}\vp\bigg\|_{L^\infty}.
\]
Take the function $f$ in Lemma \ref{disp} to be $L\px h$. Since $e^{2t\px\log|\px|}$ and $P_k$ commute, and
\[
x\px^{2}Lh=\px(x\px L h)-\px Lh=\px[x\px, L] h+\px L(x\px h)-\px Lh=-\px h+\px L(x\px h)-\px Lh,
\]
we have that
\begin{align*}
\|P_k L\partial_x \vp\|_{L^\infty}\lesssim (t+1)^{-1/2}\|\F(P_k L|\partial_x|^{\frac32}\vp)\|_{L^\infty_\xi}+(t+1)^{-3/4}[2^{\frac34 k}\|P_kL(x\px h) \|_{L^2}\\
+2^{\frac34 k}(1+|k|)\|P_k  h \|_{L^2}
+\|\tilde P_k (|\px|^{\frac34}L\vp)\|_{L^2}].
\end{align*}
It follows from \eqref{xDh} that
\[
\|P_k(x\px h) \|_{L^2}\lesssim \|  P_k \vp\|_{L^2}+\|  P_k \S\vp\|_{L^2}+t\|P_k\mathcal{N}\|_{L^2}.
\]

We first observe that $k\leq 0$ automatically leads to $(t+1)^{-1/4+p_0}2^{\frac34 k}|k|\lesssim 1$, and then we have
\begin{align}\label{sharp1-}
\begin{split}
\|P_k L\px\vp\|_{L^\infty} & \lesssim (t+1)^{-1/2} 2^{k/2}|k|\|\psi_k(\xi)|\xi|^{}\hat \vp(\xi)\|_{L^\infty_\xi} +(t+1)^{-1/2-p_0}[\|\tilde P_k\vp\|_{L^2}+\|P_k\mathcal S\vp\|_{L^2}+t\|P_k\mathcal N\|_{L^2}].
\end{split}
\end{align}
Summing over for $k \leq 0$, using \eqref{estimate_N}, the bootstrap assumptions, and \eqref{sharp1-} in the corresponding ranges of $k$, and we obtain that
\[
\bigg\|\sum\limits_{k\leq 0}P_kL\vp_x\bigg\|_{L^\infty} \lesssim \ve_1 (t + 1)^{- 1 / 2}.
\]

To estimate $\|P_k |\px|^{r + 2} \vp_x\|_{L^{\infty}}$, we take the function $f$ in Lemma \ref{disp} to be $|\px|^{r + 2} h_x$ and obtain
\[
\|P_k |\partial_x|^{r + 2} \vp_x\|_{L^\infty}\lesssim (t+1)^{-1/2}\|\F(P_k |\partial_x|^{r + \frac{7}{2}}\vp)\|_{L^\infty_\xi}+(t+1)^{-3/4}[2^{-k/4}\|P_k(x |\px|^{r + 4}h) \|_{L^2}+\|\tilde P_k (|\px|^{r + \frac{7}{4}}\vp_x)\|_{L^2}].
\]
Using
\[
- x |\px|^{r + 4} h = [x \px, \px |\px|^{r + 2}]h + \px |\px|^{r + 2}(x \px h) = -(r + 3) \px |\px|^{r + 2} h + \px |\px|^{r + 2}(x \px h),
\]
 and \eqref{xDh}, we get that
\begin{align*}
\|P_k|\px|^{r + 2} \vp_x\|_{L^\infty} \lesssim (t + 1)^{-1 / 2} \|\psi_k(\xi)|\xi|^{r + \frac72}\hat \vp(\xi)\|_{L^\infty_\xi}+(t+1)^{-3/4} 2^{\frac34k} \big[\||\px|^{r + 2} \tilde P_k\vp\|_{L^2}\\
+\||\px|^{r + 2} P_k\mathcal S\vp\|_{L^2}+\||\px|^{r + 2}P_k\vp\|_{L^2}+t\||\px|^{r + 2}P_k\mathcal N\|_{L^2}\big].
\end{align*}

For  $k\in \Z_+$ and $(t+1)^{-1/4+p_0}2^{(r + \frac{11}4)k}\lesssim 1$, we have
\begin{align}
\begin{split}
\|P_k|\px|^{r + 2} \vp_x\|_{L^\infty} \lesssim &(t+1)^{-1/2} 2^{-\frac{k}{2}}\|\psi_k(\xi)|\xi|^{r+4}\hat \vp(\xi)\|_{L^\infty_\xi}\\
& \qquad +(t+1)^{-1/2-p_0}[\|\tilde P_k\vp\|_{L^2}+\|P_k\mathcal S\vp\|_{L^2}+t\|P_k\mathcal N\|_{L^2}].
\end{split}
\end{align}
Finally, for $k\in \Z_+$ and $(t+1)^{-1/4+p_0}2^{(r + \frac{11}4) k}\gtrsim 1$, we have
\begin{align}
\label{sharp2+}
\begin{split}
\|P_k\px^{r + 2} \vp_x\|_{L^\infty}&\lesssim \||\xi|^{r + 3} \psi_k(\xi)\hat\vp(\xi)\|_{L^1_\xi}\lesssim \||\xi|^{r + 3 - s}\psi_k(\xi)\|_{L^2}\|\tilde P_k\vp\|_{H^s}
\\
&\lesssim 2^{(r + 3 - s + \frac12)k}\|\tilde P_k\vp\|_{H^s}\lesssim (t+1)^{-(s-\frac72-r)\frac{1-4p_0}{11+4r}}\|\tilde P_k\vp\|_{H^s}.
\end{split}
\end{align}
Summing over $k \in \Z_+$, using \eqref{estimate_N}, the bootstrap assumptions, and \eqref{sharp2+} in the corresponding range of $k$, we obtain that
\[
\bigg\|\sum\limits_{k> 0}P_k|\partial_x|^{r + 2}\vp_x\bigg\|_{L^\infty} \lesssim \ve_1 (t + 1)^{- 1 / 2},
\]
which completes the proof.
\end{proof}

\section{Scaling-Galilean estimate}

In this section, we prove the scaling-Galilean estimate in Lemma~\ref{weightedE}.

First, we summarize some commutator identities for the scaling-Galilean operator $\S$ defined in \eqref{defS} and
$L = \log|\partial_x|$. The straightforward proofs follow
by use of the Fourier transform and are omitted.
\begin{lemma}
\label{S-comm}
Let $\varphi(x,t)$ be a Schwartz distribution on $\R^2$ such that $L\varphi(x,t)$ is a Schwartz distribution. Then
\begin{align*}
&[\S, \partial_x]\vp = -\partial_x\vp,\qquad
[\S, L]\vp = -\vp,\qquad
[\S, L\partial_x]\vp = -\vp_x-L\partial_x\vp,
\\
&[\S, \partial_t]\vp = -2\partial_x\vp-\partial_t\vp,\qquad
[\S, \partial_t-2L\partial_x]\vp = -\partial_t\vp+2L\partial_x\varphi.
\end{align*}
\end{lemma}

Next, we prove a weighted energy estimate for $\S \vp$.
\begin{proof}[Proof of Lemma \ref{weightedE}]
Applying $\S$ to equation \eqref{fsqgeq} and using Lemma \ref{S-comm}, we get
\[
(\S \vp)_t-2L\partial_x(\S\vp)+\partial_x  T_{B^0[\vp]} \S\vp + L[ T_{B^{\log}[\vp]}\S\vp]_x+\S\Rc=\text{commutators},
\]
where the commutators are
\[
\partial_x [\S, T_{B^0[\vp]}]\vp,\qquad
[\S, \partial_x] T_{B^0[\vp]}\vp,\qquad
[\S, L\partial_x]\big(T_{B^{\log}[\vp]}\vp\big),\qquad
L\partial_x\Big([\S, T_{B^{\log}[\vp]}]\vp\Big).
\]
The first commutator can be written as
\begin{align*}
[\S, T_{B^0[\vp]}]\vp&=[(x+2t)\px+t\partial_t, T_{B^0[\vp]}]\vp\\
&=(x+2t)\px T_{B^0[\vp]}\vp-T_{B^0[\vp]}[(x+2t)\px \vp]+t\partial_t T_{B^0[\vp]}\vp-T_{B^0[\vp]} (t\partial_t \vp)\\
&=(x+2t)T_{\px B^0[\vp]}\vp+[(x+2t), T_{ B^0[\vp]}]\px\vp+T_{t\partial_t B^0[\vp]}\vp\\
&=T_{(x+2t) \px B^0[\vp]}\vp+\left(xT_{\px B^0[\vp]}\vp-T_{x\px B^0[\vp]}\vp\right)+[x, T_{ B^0[\vp]}]\px\vp+T_{t\partial_t B^0[\vp]}\vp.
\end{align*}
By the commutator estimates in Lemma~\ref{Commu} and Theorem~\ref{Linftybound}, we obtain for $0\leq k\leq r$ that
\begin{align*}
\|\partial_x [\S, T_{B^0[\vp]}]\vp\|_{H^k}&\lesssim \|[x, T_{B^0[\vp]}]\px \vp\|_{H^{k+1}}+\|xT_{\px B^0[\vp]}\vp-T_{x\px B^0[\vp]}\vp\|_{H^{k+1}}+ \|T_{\S B^0[\vp]}\vp_x\|_{H^{k}}\\
&\lesssim \|B^0[\vp]\|_{\Mc_{(1,2)}}\|\vp\|_{H^{k+2}}+\|B^0[\vp]\|_{\Mc_{(2,2)}}\|\vp\|_{H^{k+1}}+\|\S B^0[\vp]\|_{\mathcal L^2_1}  \| \vp_x\|_{W^{k+1,\infty}}.
\end{align*}
Using \eqref{B-est}, together with Lemma~\ref{multilinear} and similar estimates for $\|\S B^0[\vp]\|_{\mathcal L^2_1}$, we find that
\begin{align*}
\|\partial_x [\S, T_{B^0[\vp]}]\vp\|_{H^k}
&\lesssim F(\|L\vp_x\|_{W^{2,\infty}}+\|\vp_x\|_{W^{2,\infty}}) (\|L\vp_x\|_{W^{2,\infty}}+\|\vp_x\|_{W^{2,\infty}}) \|\vp_x\|_{W^{r+1, \infty}}(\|\S\vp\|_{H^{r}}+\|\vp\|_{H^s}).
\end{align*}
Similarly, we have
\begin{align*}
\bigg\| L\partial_x\Big([\S, T_{B^{\log}[\vp]}]\vp\Big)\bigg\|_{H^k}&\lesssim F(
\|L\vp_x\|_{W^{2,\infty}}+\|\vp_x\|_{W^{2,\infty}}) (\|L\vp_x\|_{W^{2,\infty}}+\|\vp_x\|_{W^{2,\infty}}) \|\vp_x\|_{W^{r+1, \infty}}(\|\S\vp\|_{H^{r}}+\|\vp\|_{H^s}).
\end{align*}

By Lemma \ref{S-comm}, Lemma~\ref{lem:taest}, and \eqref{B-est}, the second and third commutators satisfy
\begin{align*}
\|[\S, \partial_x] T_{B^0[\vp]}\vp\|_{H^k}&=\|T_{B^0[\vp]}\vp\|_{H^{k+1}}\lesssim F(\|L\vp_x\|_{W^{2,\infty}}+\|\vp_x\|_{W^{2,\infty}}) (\|L\vp_x\|_{W^{2,\infty}}+\|\vp_x\|_{W^{2,\infty}})^2\|\vp\|_{H^{k+1}},\\[.5ex]
\big\|[\S, L\partial_x]\big(T_{B^{\log}[\vp]}\vp\big)\big\|_{H^k}&=\big\|\big(T_{B^{\log}[\vp]}\vp\big)_x
+L\partial_x\big(T_{B^{\log}[\vp]}\vp\big)\big\|_{H^k}
\\
&\lesssim F(\|L\vp_x\|_{W^{1,\infty}}+\|\vp_x\|_{W^{1,\infty}}) (\|L\vp_x\|_{W^{1,\infty}}+\|\vp_x\|_{W^{1,\infty}})^2(\|\vp\|_{H^{k+1}}+\|L\vp\|_{H^{k+1}}).
\end{align*}

Thus, the evolution equation for $\S\vp$ can be written as
\[
(\S \vp)_t +\partial_xT_{B^0[\vp]}\S\vp+\Rc_{\S}=L\big[(2-T_{B^{\log}[\vp]})\S\vp\big],
\]
where the remainder $\Rc_\S$ satisfies
\[
\|\Rc_\S\|_{H^{k}} \lesssim (\|\vp_x\|_{W^{2, \infty}} + \|L \vp_x\|_{W^{2 , \infty}})^2 \left(\|\S \vp\|_{H^r} + \|\vp\|_{H^{r+1}}+ \|L\vp\|_{H^{r+1}}\right).
\]
As in \eqref{weighted_energy1}, we define a weighted energy for $\S \vp$ by
\begin{align*}
E_\S^{(j)}(t) &= \int_{\R} |D|^{j} \S \vp(x, t) \cdot \left(2 - T_{B^{\log}[\vp]}\right)^{2j+1} |D|^{j} \S \vp(x,t) \diff{x},\qquad j=0, 1, \dotsb, r,
\\
\tilde{E}_\S^{(r)}(t) &= \sum_{j=0}^r E_\S^{(j)}(t)
\end{align*}
and repeat similar estimates to the ones in the proof of Proposition \ref{apriori} to get
\begin{align*}
\frac {\diff}{\diff{t}} E_\S^{(j)}(t) \lesssim&  F(
\|L\vp_x\|_{W^{2,\infty}}+\|\vp_x\|_{W^{2,\infty}}) (\|L\vp_x\|_{W^{2,\infty}}+\|\vp_x\|_{W^{3,\infty}})^2\|\S\vp\|_{H^{j}}^2\\
&+(\|\vp_x\|_{W^{2, \infty}} + \|L \vp_x\|_{W^{2 , \infty}})^2 \left(\|\S \vp\|_{H^r} + \|\vp\|_{H^{r+1}}+ \|L\vp\|_{H^{r+1}}\right)\|\S\vp\|_{H^{j}}.
\end{align*}
Using Lemma~\ref{sharp} and the equivalence of $\tilde{E}_\S^{(r)}$ and $\|\S \vp\|_{H^r}^2$ when $\|2 - T_{B^{\log}[\vp]}\|_{L^2 \to L^2}$ is bounded away from zero, we find by integrating in $t$ that
\[
\tilde{E}_\S^{(r)}(t)\lesssim \ve_0^2 (t + 1)^{2 p_0},
\]
which proves the lemma.
\end{proof}

\section{Nonlinear dispersive estimate}
\label{sec-Znorm}

In this section, we prove the estimate in Lemma~\ref{nonDisp} for the $Z$-norm $\|\vp\|_Z$ defined in \eqref{defZ}.

When $|\xi|<(t + 1)^{-p_0}$, Lemma \ref{interpolation} and the bootstrap assumptions give
\begin{align*}
|(|\xi|+|\xi|^{r+4})\hat\vp(\xi,t)|^2&\lesssim (|\xi|+|\xi|^{r+4})^2|\xi|^{-1}\|\hat\vp\|_{L^2_\xi}(|\xi|\|\partial_\xi\hat\vp\|_{L^2_\xi}+\|\hat\vp\|_{L^2_\xi})\\
&\lesssim (|\xi|+|\xi|^{r+4})\|\vp\|_{L^2}(\|\S\vp\|_{L^2}+\|\vp\|_{L^2})\\
&\lesssim \ve_0^2 .
\end{align*}
Let $p_1=10^{-6}$. When $|\xi|\geq (t + 1)^{p_1}$,  Lemma \ref{interpolation} and the bootstrap assumptions, with the parameter values \eqref{param_vals}, give
\begin{align*}
|(|\xi|+|\xi|^{r + 4})\hat\vp(\xi,t)|^2&\lesssim \frac{(|\xi|+|\xi|^{r + 4})^2}{|\xi|^{s+1}}\|\vp\|_{H^s}(\|\S\vp\|_{L^2}+\|\vp\|_{L^2})\\
&\lesssim |\xi|^{2r+7-s}\ve_0^2(t+1)^{2p_0} \\
&\lesssim \ve_0^2 .
\end{align*}
Thus, we only need to consider the frequency range
\begin{equation}
(t + 1)^{-p_0}\leq |\xi|\leq (t + 1)^{p_1}.
\label{xirange}
\end{equation}
In the following, we fix $\xi$ in this range, and denote by $\cutoffxi(\xi, t)$ a smooth cutoff function compactly supported on a small neighborhood of $\left\{(\xi, t): (t+1)^{-p_0}<|\xi|<(t+1)^{p_1}\right\}$.

Taking the Fourier transform of \eqref{cub-sqg}, we obtain that
\begin{equation}
\hat \vp_t(\xi) + \frac16 i\xi \iint_{\R^2} \Tb_1(\eta_1,\eta_2, \xi - \eta_1 - \eta_2) \hat\vp(\xi - \eta_1 - \eta_2) \hat\vp(\eta_1) \hat\vp(\eta_2) \diff{\eta_1} \diff{\eta_2}+\widehat{\Nc_{\geq5}(\vp)}(\xi) = 2i\xi\log|\xi|\hat\vp(\xi),
\label{phihateq}
\end{equation}
where $\Nc_{\geq5}(\vp)$ is given by \eqref{N>=5}. From \eqref{Tnexp},
\[
\begin{aligned}
\Tb_1(\eta_1,\eta_2, \eta_3) &= -\eta_1^2 \log|\eta_1| - \eta_2^2 \log|\eta_2| - \eta_3^2 \log|\eta_3| - (\eta_1 + \eta_2 + \eta_3)^2 \log|\eta_1 + \eta_2 + \eta_3|\\
& \qquad + \left\{(\eta_1 + \eta_2)^2 \log|\eta_1 + \eta_2| + (\eta_1 + \eta_3)^2 \log|\eta_1 + \eta_3| + (\eta_2 + \eta_3)^2 \log|\eta_2 + \eta_3|\right\}.
\end{aligned}
\]
with
\begin{align}
\label{partialT}
\begin{split}
\partial_{\eta_1} [\Tb_1(\eta_1, \eta_2, \xi - \eta_1 - \eta_2)] &= -2 \Big\{\eta_1 \log|\eta_1| - (\eta_1 + \eta_2) \log|\eta_1 + \eta_2|\\
&\qquad + (\xi - \eta_1) \log|\xi - \eta_1| - (\xi - \eta_1 - \eta_2) \log|\xi - \eta_1 - \eta_2|\Big\},\\
\partial_{\eta_2} [\Tb_1(\eta_1, \eta_2, \xi - \eta_1 - \eta_2)] &= -2 \Big\{\eta_2 \log|\eta_2| - (\eta_1 + \eta_2) \log|\eta_1 + \eta_2|\\
&\qquad + (\xi - \eta_2) \log|\xi - \eta_2| - (\xi - \eta_1 - \eta_2) \log|\xi - \eta_1 - \eta_2|\Big\}.
\end{split}
\end{align}

\subsection{Modified scattering}

Nonlinearity leads to a cumulative frequency shift in the long-time behavior of the Fourier components of the solution due to space-time resonances of the form $\xi+\xi -\xi = \xi$. To account for this
effect, we use the method of modified scattering and introduce a phase correction
\[
\Theta(\xi, t)=-2t\xi\log|\xi|+\xi \int_0^t [\beta_1(t)\Tb_1(\xi, \xi, -\xi)+\beta_2(t)\Tb_1(\xi,-\xi,\xi)+\beta_3(t)\Tb_1(-\xi, \xi, \xi)]|\hat \vp(\xi, \tau)|^2 \diff{\tau},
\]
where $\beta_1(t), \beta_2(t)$, and $\beta_3(t)$ are real-valued functions of $t$ to be determined later. We then let
\[
\hat v(\xi, t)=e^{i\Theta(\xi, t)}\hat \vp(\xi, t).
\]

Using \eqref{phihateq}, we find that
\begin{align}
\begin{split}
\hat v_t(\xi, t) &= e^{i \Theta(\xi, t)} [\hat\vp_t(\xi, t) + i \Theta_t(\xi, t) \hat\vp(\xi, t)]
\\
&= U_1(\xi, t) + U_2(\xi, t) - e^{i \Theta(\xi, t)} \widehat{\Nc_{\geq 5}(\vp)}(\xi, t),
\end{split}
\label{vteq}
\end{align}
where
\begin{align*}
U_1(\xi, t) &=e^{i\Theta(\xi, t)}\bigg\{ -\frac16 i\xi \iint_{\R^2}  \Tb_1(\eta_1,\eta_2, \xi - \eta_1 - \eta_2) \hat\vp(\xi-\eta_1-\eta_2, t)\hat\vp(\eta_1, t)\hat\vp(\eta_2, t) \diff{\eta_1} \diff{\eta_2}\\
&\qquad +i\xi \left[\beta_1(t)\Tb_1(\xi, \xi, -\xi)+\beta_2(t)\Tb_1(\xi, -\xi, \xi)+\beta_3(t)\Tb_1(-\xi,\xi,\xi)\right]|\hat \vp(\xi, t)|^2\hat \vp(\xi, t) \bigg\},\\
U_2(\xi, t) &= \hat v(\xi,t)\bigg\{  i\xi \int_0^t \left[\beta_1'(t)\Tb_1(\xi, \xi, -\xi)+\beta'_2(t)\Tb_1(\xi, -\xi, \xi)+\beta'_3(t)\Tb_1(-\xi,\xi,\xi)\right]|\hat \vp(\xi, \tau)|^2 \diff{\tau}\bigg\}.
\end{align*}
The coefficient of $\hat v$ in the term $U_2$ is purely imaginary, so it leads to a phase shift in $\hat{v}$ that does not affect its norm, and we get from \eqref{vteq} that
\begin{align*}
\|\vp\|_{Z}&=\|(|\xi|+|\xi|^{r+4}) \hat\vp(\xi, t)\|_{L^\infty_\xi}=\|(|\xi|+|\xi|^{r+4}) \hat v(\xi, t)\|_{L^\infty_\xi}
\\
&\lesssim \int_0^t \|(|\xi|+|\xi|^{r+4})U_1(\xi, \tau)\|_{L^\infty_\xi}+\|(|\xi|+|\xi|^{r+4})\widehat{\mathcal{N}_{\geq5}(\vp)}(\xi, \tau)\|_{L^\infty_\xi} \diff \tau.
\end{align*}
We will estimate the cubic terms involving $U_1$ in Sections \ref{sec:high}--\ref{sec:res} and the higher-degree terms involving
$\widehat{\Nc_{\geq5}(\vp)}$ in Section \ref{higherorder}. We do not need to consider the terms in $U_1$ that involve
the $\beta_j$ until we come to an analysis of the space-time resonances in Section~\ref{sec:res}.

To begin with, we recall that $h =e^{-2t\partial_x\log|\partial_x|}\vp$ is defined in \eqref{Lambda-h}.
 From \eqref{phihateq}, we find that $\hat h$ satisfies
\begin{align}\label{eqhhat}
\begin{split}
&\hat h_t(\xi, t) +\frac16 i\xi\iint_{\R^2} \Tb_1(\eta_1, \eta_2, \xi - \eta_1 - \eta_2) e^{it\Phi(\xi,\eta_1,\eta_2)} \hat h(\xi-\eta_1-\eta_2, t) \hat h(\eta_1, t)\hat h(\eta_2, t)  \diff{\eta_1} \diff{\eta_2}
\\
&\hskip2in+e^{-2it\xi\log|\xi|}\widehat{\Nc_{\geq5}(\vp)}(\xi, t)=0,
\end{split}
\end{align}
where
\begin{equation}
\Phi(\xi,\eta_1,\eta_2)=2(\xi-\eta_1-\eta_2)\log|\xi-\eta_1-\eta_2|+2\eta_1\log|\eta_1|+2\eta_2\log|\eta_2|-2\xi\log|\xi|.
\label{defPhi}
\end{equation}
Suppressing the dependence on the time variable $t$, we can write the integral in $U_1$ involving $\vp$ in terms of $h$ as
\begin{align*}
&\iint_{\R^2}  \Tb_1(\eta_1,\eta_2, \xi - \eta_1 - \eta_2) \hat\vp(\xi-\eta_1-\eta_2)\hat\vp(\eta_1)\hat\vp(\eta_2) \diff{\eta_1} \diff{\eta_2}\\
&\qquad =\iint_{\R^2} \Tb_1(\eta_1, \eta_2, \xi - \eta_1 - \eta_2) e^{it\Phi(\xi,\eta_1,\eta_2)} \hat h(\xi-\eta_1-\eta_2) \hat h(\eta_1)\hat h(\eta_2)  \diff{\eta_1} \diff{\eta_2}.
\end{align*}

Carrying out a dyadic decomposition, with $ h_j=P_j h$ and $ \vp_j=P_j\vp$ where $P_j$ is the Fourier multiplier with symbol $\psi_j$ defined in \eqref{defpsik}, we rewrite this integral in each dyadic block as
\begin{align}
\label{cubdyadic}
 \iint_{\R^2}   \Tb_1(\eta_1, \eta_2, \xi - \eta_1 - \eta_2) e^{it\Phi(\xi,\eta_1,\eta_2)} \hat h_{j_1}(\eta_1)\hat h_{j_2}(\eta_2) \hat h_{j_3}(\xi-\eta_1-\eta_2) \diff{\eta_1} \diff{\eta_2}.
\end{align}

In the following subsections, we estimate this integral in various regions of frequency-space.

In Section~\ref{sec:high}, we estimate the integral for high frequencies (large $j_1$, $j_2$, and $j_3$). In Section~\ref{sec:nonres}, we estimate the integral for nonresonant frequencies, using oscillatory integral estimates with respect to the frequency variables together with multilinear estimates to get sufficient time decay.

In Section~\ref{sec:near}, we consider frequencies that are close to the resonant frequencies. In that case, the bounds for the multilinear symbols are worse, so we cannot obtain sufficient time decay by the method used for the nonresonant frequencies. We resolve this issue by an additional dyadic decomposition centered at each resonant point and a refinement of the symbol estimates.

Finally, in Section~\ref{sec:res}, we consider frequencies that are at the space resonance or space-time resonances. For the space resonance, we estimate the integral in a region about the space resonance point that shrinks in time, using an oscillatory integral estimate with respect to time and the equation to eliminate the time-derivative of the solution. For the space-time resonances, we take advantage of the modified scattering phase correction and estimate the integral on shrinking regions about the space-time resonance points.

\subsection{High frequencies}\label{sec:high}
When $\max\{j_1,j_2,j_3\} \gtrsim 10^{-3}\log_2 |t+1|>0$, we can estimate the nonlinear terms \eqref{cubdyadic} by using Lemma \ref{multilinear},  with the $L^\infty$-norm placed on the lowest derivative term. There are, in total, $r+6=13$ derivatives shared by three factors of $\vp$. Thus, we can ensure that the term with least derivatives has at most four derivatives, with or without a logarithmic derivative.

To be more specific, using H\"{o}lder's inequality, Sobolev embedding, and the bootstrap assumptions, we obtain the estimate
\begin{align*}
 & \Bigg\| \xi (|\xi|+|\xi|^{r + 4})\cutoffxi(\xi, t)  \iint_{\R^2}   \Tb_1(\eta_1, \eta_2, \xi - \eta_1 - \eta_2) e^{it\Phi(\xi,\eta_1,\eta_2)} \hat h_{j_1}(\eta_1)\hat h_{j_2}(\eta_2) \hat h_{j_3}(\xi-\eta_1-\eta_2)  \diff{\eta_1} \diff{\eta_2} \Bigg\|_{L^\infty_\xi}\\
 \lesssim ~~& (t + 1)^{(r+8-s)10^{-3}}\|\vp_{\min}\|_{L^2}(\|\vp_{\med}\|_{L^{\infty}}+\|L\partial_x\vp_{\med}\|_{W^{r,\infty}})\|\vp_{\max}\|_{H^s}\\
\lesssim ~~& (t + 1)^{(r+8-s)10^{-3}} \|\vp_{j_1}\|_{H^s}\|\vp_{j_2}\|_{H^s}\|\vp_{j_3}\|_{H^s},
\end{align*}
where $\max$, $\med$, $\min$ represent the maximum, median, and the minimum of  $j_1$, $j_2$, $j_3$, and $\cutoffxi$ is the cutoff function for the frequency range \eqref{xirange}. From \eqref{param_vals}, we have $(r+8-s)10^{-3}<-1.1$, so the right-hand-side is summable over $j_1$, $j_2$, $j_3$, and the sum is integrable for $t\in (0,\infty)$.

\subsection{Nonresonant frequencies}\label{sec:nonres}
We now only need to consider when $ \max\{j_1,j_2,j_3\}<10^{-3}\log_2(t+1)$.
The regions $|j_1-j_3|>1$ or $|j_2-j_3|>1$ correspond to nonresonant frequencies. Without loss of generality, we assume $|j_1-j_3|>1$.

Notice that by \eqref{defPhi}, we have
\begin{align}\label{denom1}
\partial_{\eta_1}\Phi = 2\log|\eta_1|-2\log|\xi-\eta_1-\eta_2|.
\end{align}
Since $|\eta_1|$ and $|\xi-\eta_1-\eta_2|$ are in different dyadic blocks, we have $\big||\eta_1|-|\xi-\eta_1-\eta_2|\big|\gtrsim \max\{|\eta_1|, |\xi-\eta_1-\eta_2|\}$.
Therefore, $|\partial_{\eta_1}\Phi| \gtrsim 1$.

After integrating by parts, we have
\begin{align*}
&  \iint_{\R^2}    \Tb_1(\eta_1, \eta_2, \xi - \eta_1 - \eta_2) e^{it\Phi(\xi,\eta_1,\eta_2)} \hat h_{j_1}(\eta_1)\hat h_{j_2}(\eta_2) \hat h_{j_3}(\xi-\eta_1-\eta_2) \diff{\eta_1} \diff{\eta_2}\\
 =~&   \iint_{\R^2}    \frac{\Tb_1(\eta_1, \eta_2, \xi - \eta_1 - \eta_2)}{it\partial_{\eta_1}\Phi(\xi,\eta_1,\eta_2)}\partial_{\eta_1} e^{it\Phi(\xi,\eta_1,\eta_2)} \hat h_{j_1}(\eta_1)\hat h_{j_2}(\eta_2) \hat h_{j_3}(\xi-\eta_1-\eta_2) \diff{\eta_1} \diff{\eta_2}\\
 =~& -W_1-W_2-W_3,
\end{align*}
where
\[
W_1(\xi, t)=  \iint_{\R^2} \partial_{\eta_1}\left[ \frac{\Tb_1(\eta_1, \eta_2, \xi - \eta_1 - \eta_2)}{it\partial_{\eta_1}\Phi(\xi,\eta_1,\eta_2)}\right] e^{it\Phi(\xi,\eta_1,\eta_2)} \hat h_{j_1}(\eta_1)\hat h_{j_2}(\eta_2) \hat h_{j_3}(\xi-\eta_1-\eta_2) \diff{\eta_1} \diff{\eta_2},
\]
\[
W_2(\xi, t)=  \iint_{\R^2} \left[ \frac{\Tb_1(\eta_1, \eta_2, \xi - \eta_1 - \eta_2)}{it\partial_{\eta_1}\Phi(\xi,\eta_1,\eta_2)}\right] e^{it\Phi(\xi,\eta_1,\eta_2)} \hat h_{j_1}(\eta_1)\hat h_{j_2}(\eta_2) \partial_{\eta_1}\hat h_{j_3}(\xi-\eta_1-\eta_2) \diff{\eta_1} \diff{\eta_2},
\]\[
W_3(\xi, t)=  \iint_{\R^2} \left[ \frac{\Tb_1(\eta_1, \eta_2, \xi - \eta_1 - \eta_2)}{it\partial_{\eta_1}\Phi(\xi,\eta_1,\eta_2)}\right] e^{it\Phi(\xi,\eta_1,\eta_2)} \partial_{\eta_1} \hat h_{j_1}(\eta_1)\hat h_{j_2}(\eta_2) \hat h_{j_3}(\xi-\eta_1-\eta_2) \diff{\eta_1} \diff{\eta_2}.
\]

{\bf Estimate of $W_1$.}
Since
\begin{align}
\label{F-1W1}
\|W_1\|_{L^\infty_\xi}\lesssim \|\F^{-1}({W}_1)\|_{L^1},
\end{align}
it suffices to estimate the $L^1_x$ norm of
\[
\iiint_{\R^3} e^{i\xi x}   \partial_{\eta_1}\left[ \frac{\Tb_1(\eta_1, \eta_2, \xi - \eta_1 - \eta_2)}{it\partial_{\eta_1}\Phi(\xi,\eta_1,\eta_2)}\right] e^{it\Phi(\xi,\eta_1,\eta_2)} \hat h_{j_1}(\eta_1)\hat h_{j_2}(\eta_2) \hat h_{j_3}(\xi-\eta_1-\eta_2) \diff{\eta_1} \diff{\eta_2}\diff \xi.
\]

Notice that by \eqref{denom1}
\begin{align*}
\partial_{\eta_1}\frac{\Tb_1(\eta_1, \eta_2, \xi - \eta_1 - \eta_2)}{\partial_{\eta_1}\Phi(\xi, \eta_1, \eta_2)}
= \kappa_1(\eta_1,\eta_2,\xi - \eta_1 - \eta_2) - \frac{\kappa_2(\eta_1,\eta_2,\xi - \eta_1 - \eta_2)}{2},
\end{align*}
where
\begin{align*}
\kappa_1(\eta_1,\eta_2,\eta_3)&=\frac{\partial_{\eta_1}\Tb_1(\eta_1,\eta_2,\eta_3)-\partial_{\eta_3}\Tb_1(\eta_1,\eta_2,\eta_3)}{\log|\eta_1|-\log|\eta_3|},\\
\kappa_2(\eta_1,\eta_2,\eta_3)&=\Tb_1(\eta_1,\eta_2,\eta_3)  \frac{\frac1{\eta_1}+\frac1{\eta_3}}{(\log|\eta_1|-\log|\eta_3|)^2}.
\end{align*}

Making a change of variable $\eta_3=\xi-\eta_1-\eta_2$, we need to estimate the trilinear form
\begin{align*}
\frac1{it}\iiint_{\R^3}e^{i\xi x}   \left[\kappa_1(\eta_1,\eta_2,\eta_3)+\kappa_2(\eta_1,\eta_2,\eta_3)\right] \hat \vp_{j_1}(\eta_1)\hat \vp_{j_2}(\eta_2) \hat \vp_{j_3}(\eta_3) \diff{\eta_1} \diff{\eta_2}\diff\eta_3,
\end{align*}
with symbol
\[
   \left[\kappa_1(\eta_1,\eta_2,\eta_3)+\kappa_2(\eta_1,\eta_2,\eta_3)\right]\psi_{j_1}(\eta_1)\psi_{j_2}(\eta_2)\psi_{j_3}(\eta_3).
\]

According to Lemma~\ref{multilinear}, this trilinear operator is bounded on $L^2\times L^2 \times L^\infty \to L^1$ by
\begin{align}
\label{kappaest}
\begin{split}
&\left\|\left[\kappa_1(\eta_1,\eta_2,\eta_3)+\kappa_2(\eta_1,\eta_2,\eta_3)\right]\psi_{j_1}(\eta_1)\psi_{j_2}(\eta_2)\psi_{j_3}(\eta_3)\right\|_{S^\infty}\\
\lesssim~&\Big(\|  \partial_{\eta_1}\Tb_1(\eta_1,\eta_2,\eta_3)\tilde\psi_{j_1}(\eta_1)\tilde\psi_{j_2}(\eta_2)\tilde\psi_{j_3}(\eta_3)\|_{S^\infty}\\
&\qquad +\|\partial_{\eta_3}\Tb_1(\eta_1,\eta_2,\eta_3)\tilde\psi_{j_1}(\eta_1)\tilde\psi_{j_2}(\eta_2)\tilde\psi_{j_3}(\eta_3)\|_{S^\infty}\Big) \cdot\left\|\frac{\psi_{j_1}(\eta_1)\psi_{j_2}(\eta_2)\psi_{j_3}(\eta_3)}{\log|\eta_1|-\log|\eta_3|}\right\|_{S^\infty}\\
&+\bigg(\left\|\frac{\Tb_1(\eta_1,\eta_2,\eta_3)}{\eta_1}\tilde\psi_{j_1}(\eta_1)\tilde\psi_{j_2}(\eta_2)\tilde\psi_{j_3}(\eta_3)\right\|_{S^\infty}\\
& \qquad +\left\|  \frac{\Tb_1(\eta_1,\eta_2,\eta_3)}{\eta_3}\tilde\psi_{j_1}(\eta_1)\tilde\psi_{j_2}(\eta_2)\tilde\psi_{j_3}(\eta_3)\right\|_{S^\infty}\bigg)\cdot\left\|\frac{\psi_{j_1}(\eta_1)\psi_{j_2}(\eta_2)\psi_{j_3}(\eta_3)}{(\log|\eta_1|-\log|\eta_3|)^2}\right\|_{S^\infty}.
\end{split}
\end{align}
\begin{lemma}\label{log-log}
Suppose that $|j_1-j_3|>1$. Then for any $m \in \Z_+$,
\[
\left\|\frac{1}{(\log|\eta_1|-\log|\eta_3|)^m}\psi_{j_1}(\eta_1)\psi_{j_2}(\eta_2)\psi_{j_3}(\eta_3)\right\|_{S^\infty}\lesssim 1.
\]
\end{lemma}
\begin{proof}
By the definition of the $S^\infty$-norm \eqref{Sinf} and the definition of $\psi_k$  \eqref{defpsik}, we have that
\begin{align*}
&\left\|\frac{\psi_{j_1}(\eta_1)\psi_{j_2}(\eta_2)\psi_{j_3}(\eta_3)}{(\log|\eta_1|-\log|\eta_3|)^m}\right\|_{S^\infty}\\
=~&\left\|\iiint_{\R^3} \frac{\psi_{j_1}(\eta_1)\psi_{j_2}(\eta_2)\psi_{j_3}(\eta_3)}{(\log|\eta_1|-\log|\eta_3|)^m} e^{i(y_1\eta_1+y_2\eta_2+y_3\eta_3)}\diff \eta_1\diff\eta_2\diff\eta_3\right\|_{L^1}\\
=~&\iiint_{\R^3}\left|\iiint_{\R^3} \frac{\psi_0(2^{-j_1}\eta_1)\psi_0(2^{-j_2}\eta_2)\psi_0(2^{-j_3}\eta_3)}{(\log|\eta_1|-\log|\eta_3|)^m} e^{i(y_1\eta_1+y_2\eta_2+y_3\eta_3)}\diff \eta_1\diff\eta_2\diff\eta_3\right|\diff y_1\diff y_2\diff y_3\\
\lesssim~&1,
\end{align*}
where the last inequality comes from oscillatory integral estimates, using the fact that $|j_1-j_3|>1$
and the support of $\psi_0$ is $(-\frac85, -\frac58)\cup(\frac58,\frac85)$.
\end{proof}

For the estimates of other symbols in \eqref{kappaest}, we have the following lemma.
\begin{lemma} \label{estT1}
For any $j_1, j_2, j_3 \in \Z$, we have
\begin{align}
& \|  \partial_{\eta_1}\Tb_1(\eta_1,\eta_2,\eta_3)\tilde\psi_{j_1}(\eta_1)\tilde\psi_{j_2}(\eta_2)\tilde\psi_{j_3}(\eta_3)\|_{S^\infty}\lesssim 2^{\max\{j_2, j_3\}},\label{est_symb1}\\
& \|  \Tb_1(\eta_1,\eta_2,\eta_3)\tilde\psi_{j_1}(\eta_1)\tilde\psi_{j_2}(\eta_2)\tilde\psi_{j_3}(\eta_3)\|_{S^\infty}\lesssim
2^{\max\{j_1,j_2,j_3\}+\min\{j_1,j_2,j_3\}},\label{est_symb1'}
\end{align}
and
\begin{align}\label{est_symb2}
\left\|  \frac{\Tb_1(\eta_1,\eta_2,\eta_3)}{\eta_1}\tilde\psi_{j_1}(\eta_1)\tilde\psi_{j_2}(\eta_2)\tilde\psi_{j_3}(\eta_3)\right\|_{S^\infty}\lesssim 2^{\max\{j_2,j_3\}}.
\end{align}
Furthermore, since $\Tb_1$ is symmetric, we also have
\begin{align*}
& \|  \partial_{\eta_3}\Tb_1(\eta_1,\eta_2,\eta_3)\tilde\psi_{j_1}(\eta_1)\tilde\psi_{j_2}(\eta_2)\tilde\psi_{j_3}(\eta_3)\|_{S^\infty}\lesssim 2^{\max\{j_1, j_2\}},\\
& \left\|  \frac{\Tb_1(\eta_1,\eta_2,\eta_3)}{\eta_3}\tilde\psi_{j_1}(\eta_1)\tilde\psi_{j_2}(\eta_2)\tilde\psi_{j_3}(\eta_3)\right\|_{S^\infty}\lesssim 2^{\max\{j_1,j_2\}}.
\end{align*}
\end{lemma}

\begin{proof} {\bf 1.} We prove \eqref{est_symb1} first. Using inverse Fourier transform in $(\eta_1, \eta_2, \eta_3)$, we obtain
\begin{align*}
&\F^{-1}[  \partial_{\eta_1}\Tb_1(\eta_1,\eta_2,\eta_3)\tilde\psi_{j_1}(\eta_1)\tilde\psi_{j_2}(\eta_2) \tilde\psi_{j_3}(\eta_3)]\\
=~&\iiint_{\R^3}e^{i(y_1\eta_1+y_2\eta_2+y_3\eta_3)}   \partial_{\eta_1}\left[\int_{\R}\frac{\prod_{j=1}^3(1-e^{i\eta_j\zeta})}{|\zeta|^{3}}\diff \zeta\right] \tilde\psi_{j_1}(\eta_1) \tilde\psi_{j_2}(\eta_2) \tilde\psi_{j_3}(\eta_3) \diff \eta_1\diff \eta_2\diff \eta_3\\
=~&\iiint_{\R^3}  \left[\int_{\R}\frac{-i\zeta e^{i\eta_1(\zeta+y_1)}(e^{iy_2\eta_2}-e^{i\eta_2(\zeta+y_2)})(e^{iy_3\eta_3}-e^{i\eta_3(\zeta+y_3)})}{|\zeta|^{3}}\diff \zeta\right] \tilde\psi_{j_1}(\eta_1) \tilde\psi_{j_2}(\eta_2) \tilde\psi_{j_3}(\eta_3) \diff \eta_1\diff \eta_2\diff \eta_3\\
=~& \int_{\R}\frac{-i\zeta }{|\zeta|^{3}} \cdot\left[\F^{-1}[\tilde\psi_{j_1}](y_1+\zeta)\right]\cdot \left[\F^{-1}[\tilde\psi_{j_2}](y_2)-\F^{-1}[\tilde\psi_{j_2}](\zeta+y_2)\right] \left[\F^{-1}[\tilde\psi_{j_3}](y_3)-\F^{-1}[\tilde\psi_{j_3}](\zeta+y_3)\right] \diff \zeta.
\end{align*}

Notice that
\begin{align*}
\left|\F^{-1}[\tilde\psi_{j_1}](y_1+\zeta)\right| &= 2^{j_1}\left|\F^{-1}[\tilde\psi_0](2^{j_1}(y_1+\zeta))\right|,\\
\left|\F^{-1}[\tilde\psi_{j_2}](y_2)-\F^{-1}[\tilde\psi_{j_2}](\zeta+y_2)\right| &= 2^{j_2} \left|\F^{-1}[\tilde\psi_0](2^{j_2} y_2)-\F^{-1}[\tilde\psi_0](2^{j_2}(\zeta+y_2))\right|,\\
\left|\F^{-1}[\tilde\psi_{j_3}](y_3)-\F^{-1}[\tilde\psi_{j_3}](\zeta+y_3)\right| &= 2^{j_3} \left|\F^{-1}[\tilde\psi_0](2^{j_3} y_3)-\F^{-1}[\tilde\psi_0](2^{j_3}(\zeta+y_3))\right|,
\end{align*}
and that
\begin{align*}
& \int_{\R} \left|\F^{-1}[\tilde\psi_0](2^{j_1}(y_1+\zeta))\right|\diff y_1\lesssim 2^{-j_1},\\
& \int_{\R} \left|\F^{-1}[\tilde\psi_{j_2}](2^{j_2} y_2)-\F^{-1}[\tilde\psi_{j_2}](2^{j_2}(\zeta+y_2))\right| \diff y_2 \lesssim \min\{2^{-j_2}, |\zeta|\},\\
& \int_{\R} \left|\F^{-1}[\tilde\psi_{j_3}](2^{j_3} y_3)-\F^{-1}[\tilde\psi_{j_3}](2^{j_3}(\zeta+y_3))\right|\diff y_3 \lesssim \min\{2^{-j_3}, |\zeta|\}.
\end{align*}
Therefore, we have
\begin{align*}
&\left\|\F^{-1}[  \partial_{\eta_1}\Tb_1(\eta_1,\eta_2,\eta_3)\tilde\psi_{j_1}(\eta_1)\tilde\psi_{j_2}(\eta_2)\tilde\psi_{j_3}(\eta_3)]\right\|_{L^1}\\
\lesssim~&\int_{\R}\frac{1}{|\zeta|^2} 2^{j_2+j_3}\min\{2^{-j_2}, |\zeta|\} \min\{2^{-j_3}, |\zeta|\} \diff\zeta\\
=~&2^{j_2+j_3}\bigg(\int_{|\zeta|>\max\{2^{-j_2}, 2^{-j_3}\}}\frac1{|\zeta|^2}2^{-j_2-j_3}\diff\zeta+\int_{\min\{2^{-j_2}, 2^{-j_3}\}<|\zeta|<\max\{2^{-j_2}, 2^{-j_3}\}}\frac1{|\zeta|}\min\{2^{-j_2}, 2^{-j_3}\}\diff\zeta\\
&\quad+\int_{|\zeta|<\min\{2^{-j_2}, 2^{-j_3}\}} 1\diff\zeta \bigg)\\
\lesssim~& 2^{\max\{j_2, j_3\}}.
\end{align*}

{\bf 2.} Next, we prove \eqref{est_symb1'} and \eqref{est_symb2}. The estimate of \eqref{est_symb1'} is similarly to \eqref{est_symb1}. We first use inverse Fourier transform and write
\begin{align*}
&\F^{-1}\left[  \Tb_1(\eta_1,\eta_2,\eta_3)\tilde\psi_{j_1}(\eta_1)\tilde\psi_{j_2}(\eta_2)\tilde\psi_{j_3}(\eta_3)\right]\\
=~&\iiint_{\R^3}e^{i(y_1\eta_1+y_2\eta_2+y_3\eta_3)}   \left[\int_{\R}\frac{\prod_{j=1}^3(1-e^{i\eta_j\zeta})}{|\zeta|^{3}}\diff \zeta\right]\tilde\psi_{j_1}(\eta_1)\tilde\psi_{j_2}(\eta_2)\tilde\psi_{j_3}(\eta_3) \diff \eta_1\diff \eta_2\diff \eta_3\\
=~&\iiint_{\R^3}  \left[\int_{\R}\frac{ (e^{iy_1\eta_1}-e^{i\eta_1(\zeta+y_1)})(e^{iy_2\eta_2}-e^{i\eta_2(\zeta+y_2)})(e^{iy_3\eta_3}-e^{i\eta_3(\zeta+y_3)})}{ |\zeta|^{3}}\diff \zeta\right]\tilde\psi_{j_1}(\eta_1)\tilde\psi_{j_2}(\eta_2) \tilde\psi_{j_3}(\eta_3) \diff \eta_1\diff \eta_2\diff \eta_3\\
=~& \int_{\R}\frac{1 }{|\zeta|^{3}} \left[\F^{-1}[\tilde\psi_{j_1}](y_1)-\F^{-1}[\tilde\psi_{j_1}](\zeta+y_1)\right]\\
& \qquad \cdot \left[\F^{-1}[\tilde\psi_{j_2}](y_2)-\F^{-1}[\tilde\psi_{j_2}](\zeta+y_2)\right] \cdot \left[\F^{-1}[\tilde\psi_{j_3}](y_3)-\F^{-1}[\tilde\psi_{j_3}](\zeta+y_3)\right] \diff \zeta.
\end{align*}
Taking the $L^1$-norm, we obtain
\begin{align*}
&\left\|\F^{-1}[\Tb_1(\eta_1,\eta_2,\eta_3)\tilde\psi_{j_1}(\eta_1)\tilde\psi_{j_2}(\eta_2)\tilde\psi_{j_3}(\eta_3)]\right\|_{L^1}\\
\lesssim~&\int_{\R} 2^{j_1+j_2+j_3}\frac1{|\zeta|^3}\min\{2^{-j_1},|\zeta|\}\min\{2^{-j_2},|\zeta|\}\min\{2^{-j_3},|\zeta|\}\diff\zeta\\
\lesssim~& \int_{|\zeta|>\max\{2^{-j_1},2^{-j_2},2^{-j_3}\}}\frac1{|\zeta|^3}\diff\zeta+\int_{|\zeta|<\min\{2^{-j_1},2^{-j_2},2^{-j_3}\}}2^{j_1+j_2+j_3}\diff\zeta\\
&+\int_{\min\{2^{-j_1},2^{-j_2},2^{-j_3}\}<|\zeta|<\med\{2^{-j_1},2^{-j_2},2^{-j_3}\}}2^{\med\{j_1,j_2,j_3\}+\min\{j_1,j_2,j_3\}}\frac1{|\zeta|}\diff \zeta\\
&+\int_{\med\{2^{-j_1},2^{-j_2},2^{-j_3}\}<|\zeta|<\max\{2^{-j_1},2^{-j_2},2^{-j_3}\}}2^{\min\{j_1,j_2,j_3\}}\frac1{|\zeta|^2}\diff \zeta\\
\lesssim~&2^{2\min\{j_1,j_2,j_3\}}+2^{\max\{j_1,j_2,j_3\}+\med\{j_1,j_2,j_3\}}+2^{\max\{j_1,j_2,j_3\}+\min\{j_1,j_2,j_3\}}+2^{\min\{j_1,j_2,j_3\}+\med\{j_1,j_2,j_3\}}\\
\lesssim~&2^{\max\{j_1,j_2,j_3\}+\min\{j_1,j_2,j_3\}},
\end{align*}
which proves \eqref{est_symb1'}.

As for \eqref{est_symb2}, we define
\[
\tilde{\tilde\psi}_k(\eta):=\sum\limits_{j=k-3}^{k+3}\psi_j(\eta).
\]
Then it follows from the support of $\psi_k$ and the fact that $\psi_k$ forms a partition of unity that
\[
 \frac{\Tb_1(\eta_1,\eta_2,\eta_3)}{\eta_1}\tilde\psi_{j_1}(\eta_1)\tilde\psi_{j_2}(\eta_2)\tilde\psi_{j_3}(\eta_3)= [\Tb_1(\eta_1,\eta_2,\eta_3)\tilde\psi_{j_1}(\eta_1)\tilde\psi_{j_2}(\eta_2)\tilde\psi_{j_3}(\eta_3)]\cdot\left[ \frac{1}{\eta_1}\tilde{\tilde\psi}_{j_1}(\eta_1)\tilde{\tilde\psi}_{j_2}(\eta_2)\tilde{\tilde\psi}_{j_3}(\eta_3)\right].
\]
By Lemma~\ref{multilinear}, we have
\begin{align}\label{lm81eq1}
\left\|\frac{\Tb_1(\eta_1,\eta_2,\eta_3)}{\eta_1}\tilde\psi_{j_1}(\eta_1)\tilde\psi_{j_2}(\eta_2)\tilde\psi_{j_3}(\eta_3)\right\|_{S^\infty}\lesssim \left\|\Tb_1(\eta_1,\eta_2,\eta_3)\tilde\psi_{j_1}(\eta_1)\tilde\psi_{j_2}(\eta_2)\tilde\psi_{j_3}(\eta_3)\right\|_{S^\infty}\bigg\| \frac{\tilde{\tilde\psi}_{j_1}(\eta_1)\tilde{\tilde\psi}_{j_2}(\eta_2) \tilde{\tilde\psi}_{j_3}(\eta_3)}{\eta_1}\bigg\|_{S^\infty}.
\end{align}

In view of \eqref{est_symb1'}, we only need to estimate the second term. To this end, we have
\begin{align*}
&\left\|\frac{1}{\eta_1}\tilde{\tilde\psi}_{j_1}(\eta_1)\tilde{\tilde\psi}_{j_2}(\eta_2)\tilde{\tilde\psi}_{j_3}(\eta_3)\right\|_{S^\infty}=\left\|\int_{\R}\eta_1^{-1} \tilde{\tilde\psi}_{j_1}(\eta_1) e^{i\eta_1 y_1}\diff \eta_1 \F^{-1}[\tilde{\tilde\psi}_{j_2}](y_2) \F^{-1}[\tilde{\tilde\psi}_{j_3}](y_3) \right\|_{L^1}\lesssim 2^{-j_1}.
\end{align*}
Therefore, by \eqref{lm81eq1} and considering all the possible relations between $j_1$, $j_2$, and $j_3$, we obtain \eqref{est_symb2}.
\end{proof}

Applying the above lemmas to \eqref{kappaest} and \eqref{F-1W1}, we obtain
\begin{align*}
\|W_1\|_{L^\infty_\xi}\lesssim (t+1)^{-1} \left[\|\px\vp_{\max\{j_1,j_2\}}\|_{L^\infty}\|\vp_{j_3}\|_{L^2}\|\vp_{\min\{j_1,j_2\}}\|_{L^2} + \|\px\vp_{\max\{j_2,j_3\}}\|_{L^\infty}\|\vp_{j_1}\|_{L^2}\|\vp_{\min\{j_2,j_3\}}\|_{L^2}\right].
\end{align*}
Since the two terms are symmetric in $j_1$ and $j_3$, it suffices to estimate one of them, as the other one is similar. We use lemma \ref{disp} and get
\begin{multline*}
\|\px\vp_{\max\{j_1,j_2\}}\|_{L^\infty}\lesssim
(t+1)^{-1/2}\||\xi|^{1.5}\hat{h}_{\max\{j_1,j_2\}}\|_{L^\infty_\xi}\\
+(t+1)^{-3/4}\left[\||\px|^{0.75}P_{\max\{j_1,j_2\}}(x\px h)\|_{L^2}+\||\px|^{0.75} h_{\max\{j_1,j_2\}}\|_{L^2}  \right].
\end{multline*}
Therefore,
\begin{align*}
\|W_1\|_{L^\infty_\xi}&\lesssim (t+1)^{-1.5}\Big(\mathbf1_{\max\{j_1,j_2\}\leq 0}2^{0.5\max\{j_1,j_2\}}\||\xi|\hat{h}_{\max\{j_1,j_2\}}\|_{L^\infty_\xi} \\
&\qquad+\mathbf1_{\max\{j_1,j_2\}> 0}2^{(-1.5-r)\max\{j_1,j_2\}}\||\xi|^{r+3}\hat{h}_{\max\{j_1,j_2\}}\|_{L^\infty_\xi} \Big) \cdot\|\vp_{j_3}\|_{L^2}\|\vp_{\min\{j_1,j_2\}}\|_{L^2}\\
&\qquad+(t+1)^{-1.75}\left[\||\px|^{0.75}P_{\max\{j_1,j_2\}}(x\px h)\|_{L^2}+\||\px|^{0.75} h_{\max\{j_1,j_2\}}\|_{L^2}  \right] \|\vp_{j_3}\|_{L^2}\|\vp_{\min\{j_1,j_2\}}\|_{L^2}\\
& \qquad +(t+1)^{-1.5}\Big(\mathbf1_{\max\{j_2,j_3\}\leq 0}2^{0.5\max\{j_2,j_3\}}\||\xi|\hat{h}_{\max\{j_2,j_3\}}\|_{L^\infty_\xi} \\
&\qquad+\mathbf1_{\max\{j_2,j_3\}> 0}2^{(-1.5-r)\max\{j_2,j_3\}}\||\xi|^{r+3}\hat{h}_{\max\{j_2,j_3\}}\|_{L^\infty_\xi} \Big) \cdot\|\vp_{j_1}\|_{L^2}\|\vp_{\min\{j_2,j_3\}}\|_{L^2}\\
&\qquad+(t+1)^{-1.75}\left[\||\px|^{0.75}P_{\max\{j_2,j_3\}}(x\px h)\|_{L^2}+\||\px|^{0.75} h_{\max\{j_2,j_3\}}\|_{L^2}  \right] \|\vp_{j_1}\|_{L^2}\|\vp_{\min\{j_2,j_3\}}\|_{L^2}\\
&\lesssim (t+1)^{-1.5}\Big(\mathbf1_{\max\{j_1,j_2\}\leq 0}2^{0.5\max\{j_1,j_2\}}+\mathbf1_{\max\{j_1,j_2\}> 0}2^{(-1.5-r)\max\{j_1,j_2\}} \Big)\\
& \hspace{3in} \cdot \|h_{\max\{j_1,j_2\}}\|_Z\|\vp_{j_3}\|_{L^2}\|\vp_{\min\{j_1,j_2\}}\|_{L^2}\\
&\qquad+(t+1)^{-1.75}\left[\||\px|^{0.75}P_{\max\{j_1,j_2\}}(x\px h)\|_{L^2}+\||\px|^{0.75} h_{\max\{j_1,j_2\}}\|_{L^2}  \right] \|\vp_{j_3}\|_{L^2}\|\vp_{\min\{j_1,j_2\}}\|_{L^2}\\
& \qquad+(t+1)^{-1.5}\Big(\mathbf1_{\max\{j_2,j_3\}\leq 0}2^{0.5\max\{j_2,j_3\}}+\mathbf1_{\max\{j_2,j_3\}> 0}2^{(-1.5-r)\max\{j_2,j_3\}} \Big)\\*
& \hspace{3in} \cdot \|h_{\max\{j_2,j_3\}}\|_Z\|\vp_{j_1}\|_{L^2}\|\vp_{\min\{j_2,j_3\}}\|_{L^2}\\
&\qquad+(t+1)^{-1.75}\left[\||\px|^{0.75}P_{\max\{j_2,j_3\}}(x\px h)\|_{L^2}+\||\px|^{0.75} h_{\max\{j_2,j_3\}}\|_{L^2}  \right] \|\vp_{j_1}\|_{L^2}\|\vp_{\min\{j_2,j_3\}}\|_{L^2}.
\end{align*}

{\bf Estimate of $W_2$ and $W_3$.}
We rewrite $W_2$ as
\[
\iint_{\R^2} \left[ \frac{\Tb_1(\eta_1, \eta_2, \xi - \eta_1 - \eta_2)}{it\partial_{\eta_1}\Phi(\xi,\eta_1,\eta_2) (\xi - \eta_1 - \eta_2)}\right] e^{it\Phi(\xi,\eta_1,\eta_2)} \hat h_{j_1}(\eta_1)\hat h_{j_2}(\eta_2) \left[(\xi - \eta_1 - \eta_2)\partial_{\eta_1}\hat h_{j_3}(\xi-\eta_1-\eta_2)\right] \diff{\eta_1} \diff{\eta_2}.
\]
In view of the multilinear estimate Lemma \ref{multilinear}, we need to estimate the $S^\infty$-norm of the symbol
\[
\frac{\Tb_1(\eta_1,\eta_2,\eta_3)}{(\log|\eta_1|-\log|\eta_3|)\eta_3}\psi_{j_1}(\eta_1)\psi_{j_2}(\eta_2)\tilde\psi_{j_3}(\eta_3).
\]
Using Lemma \ref{log-log} and Lemma \ref{estT1}, as in the estimates of $W_1$, we obtain
\begin{align*}
\|W_2\|_{L^\infty_\xi}\lesssim (t+1)^{-1}\|\px\vp_{\max\{j_1,j_2\}}\|_{L^\infty}\|\xi\partial_\xi\hat{h}_{j_3}\|_{L^2_\xi}\|\vp_{\min\{j_1,j_2\}}\|_{L^2}.
\end{align*}
Using Lemma \ref{disp}, we have
\begin{align*}
& \lesssim (t+1)^{-1.5}\Big(\mathbf1_{\max\{j_1,j_2\}\leq 0}2^{0.5\max\{j_1,j_2\}}+\mathbf1_{\max\{j_1,j_2\}> 0}2^{(-1.5-r)\max\{j_1,j_2\}} \Big)\\
& \hspace{3in} \cdot\|h_{\max\{j_1,j_2\}}\|_Z  \|\xi\partial_\xi\hat{h}_{j_3}\|_{L^2_\xi}\|\vp_{\min\{j_1,j_2\}}\|_{L^2}\\
& \qquad +(t+1)^{-1.75}\left[\||\px|^{0.75}P_{\max\{j_1,j_2\}}(x\px h)\|_{L^2}+\||\px|^{0.75} h_{\max\{j_1,j_2\}}\|_{L^2}  \right] \|\xi\partial_\xi\hat{h}_{j_3}\|_{L^2_\xi}\|\vp_{\min\{j_1,j_2\}}\|_{L^2}.
\end{align*}

Similarly, we have
\begin{align*}
\|W_3\|_{L^\infty_\xi}
&\lesssim (t+1)^{-1.5}\Big(\mathbf1_{\max\{j_2,j_3\}\leq 0}2^{0.5\max\{j_2,j_3\}}+\mathbf1_{\max\{j_2,j_3\}> 0}2^{(-1.5-r)\max\{j_2,j_3\}} \Big)\\
& \hspace{3in}\cdot \|{h}_{\max\{j_2,j_3\}}\|_{Z}  \|\xi\partial_\xi\hat{h}_{j_1}\|_{L^2_\xi}\|\vp_{\min\{j_2,j_3\}}\|_{L^2}\\
& \qquad +(t+1)^{-1.75}\left[\||\px|^{0.75}P_{\max\{j_2,j_3\}}(x\px h)\|_{L^2}+\||\px|^{0.75} h_{\max\{j_2,j_3\}}\|_{L^2}  \right] \|\xi\partial_\xi\hat{h}_{j_1}\|_{L^2_\xi}\|\vp_{\min\{j_2,j_3\}}\|_{L^2}.
\end{align*}

In conclusion, for nonresonant frequencies,
\begin{align*}
&\bigg\|\xi(|\xi|+|\xi|^{r+4})\cutoffxi(\xi,t)\iint_{\R^2} \Tb_1(\eta_1,\eta_2,\xi-\eta_1-\eta_2)e^{it\Phi(\xi,\eta_1,\eta_2)}\hat h_{j_1}(\eta_1)\hat h_{j_2}(\eta_2)\hat h_{j_3}(\xi-\eta_1-\eta_2)\diff\eta_1\diff\eta_2\bigg\|_{L^\infty_\xi}\\
\lesssim~ &(t+1)^{(r+5)p_1}\left(\|W_1\|_{L^\infty_\xi}+\|W_2\|_{L^\infty_\xi}+\|W_3\|_{L^\infty_\xi}\right)\\
\lesssim~ & (t+1)^{-1.5+(r+5)p_1}\Big(  \|\vp_{j_1}\|_{L^2}\|\vp_{\min\{j_2,j_3\}}\|_{L^2} + \|\xi\partial_\xi\hat{h}_{j_1}\|_{L^2_\xi}\|\vp_{\min\{j_2,j_3\}}\|_{L^2}\Big)\\
&\hspace{3cm}\cdot\Big(\mathbf1_{\max\{j_2,j_3\}\leq 0}2^{0.5\max\{j_2,j_3\}}+\mathbf1_{\max\{j_2,j_3\}> 0}2^{(-1.5-r)\max\{j_2,j_3\}} \Big)\|{h}_{\max\{j_2,j_3\}}\|_{Z}\\[1ex]
& +(t+1)^{-1.5+(r+5)p_1}\left(\|\vp_{j_3}\|_{L^2} \|\vp_{\min\{j_1, j_2\}}\|_{L^2} + \|\xi\partial_\xi\hat{h}_{j_3}\|_{L^2_\xi}\|\vp_{\min\{j_1,j_2\}}\|_{L^2}\right)\\
&\hspace{3cm}\cdot\Big(\mathbf1_{\max\{j_1,j_2\}\leq 0}2^{0.5\max\{j_1,j_2\}}+\mathbf1_{\max\{j_1,j_2\}> 0}2^{(-1.5-r)\max\{j_1,j_2\}} \Big)\|h_{\max\{j_1,j_2\}}\|_Z \\[1ex]
&+(t+1)^{-1.75+(r+5)p_1}\bigg[\Big(\||\px|^{0.75}P_{\max\{j_1,j_2\}}(x\px h)\|_{L^2}+\||\px|^{0.75} h_{\max\{j_1,j_2\}}\|_{L^2} \Big) \|\vp_{j_3}\|_{L^2}\|\vp_{\min\{j_1,j_2\}}\|_{L^2}\\
&\hspace{3cm}+\Big(\||\px|^{0.75}P_{\max\{j_2,j_3\}}(x\px h)\|_{L^2}+\||\px|^{0.75} h_{\max\{j_2,j_3\}}\|_{L^2} \Big) \|\vp_{j_1}\|_{L^2}\|\vp_{\min\{j_2,j_3\}}\|_{L^2}\\
&\hspace{3cm}+\Big(\||\px|^{0.75}P_{\max\{j_1,j_2\}}(x\px h)\|_{L^2}+\||\px|^{0.75} h_{\max\{j_1,j_2\}}\|_{L^2} \Big) \|\xi\partial_\xi\hat{h}_{j_3}\|_{L^2_\xi}\|\vp_{\min\{j_1,j_2\}}\|_{L^2}\\
&\hspace{3cm}+\Big(\||\px|^{0.75}P_{\max\{j_2,j_3\}}(x\px h)\|_{L^2}+\||\px|^{0.75} h_{\max\{j_2,j_3\}}\|_{L^2}  \Big) \|\xi\partial_\xi\hat{h}_{j_1}\|_{L^2_\xi}\|\vp_{\min\{j_2,j_3\}}\|_{L^2}\bigg].
\end{align*}
By the bootstrap assumptions and Lemma~\ref{lem:xdh}, the right-hand-side is summable for $j_1,j_2,j_3$ and the sum is integrable for $t\in (0,\infty)$.

\subsection{Close to the resonance}\label{sec:near}
When
\begin{align}
\label{resregion}
\max\{j_1,j_2,j_3\}<10^{-3}\log_2(t+1),\qquad |j_3-j_2|\leq 1, \qquad|j_3-j_1|\leq 1,
\end{align}
we need to consider the following two cases:
\begin{enumerate}[(i)]
\item Frequencies $\eta_1, \eta_2$ and $\xi-\eta_1-\eta_2$ have the same sign.

By the definition of cutoff function $\psi$, we have
\[
\frac58 2^{j_1} \leq |\eta_1|\leq\frac85 2^{j_1},\quad \frac58 2^{j_2} \leq |\eta_2|\leq\frac85 2^{j_2},\quad \frac58 2^{j_3} \leq |\xi-\eta_1-\eta_2|\leq\frac85 2^{j_3},
\]
and thus,
\[
\frac58 (2^{j_1}+2^{j_2}+ 2^{j_3})\leq |\xi| \leq \frac85 (2^{j_1}+2^{j_2}+ 2^{j_3}).
\]
This corresponds to the region near the {\it space resonance} $\eta_1 = \eta_2 = \xi - \eta_1 - \eta_2 = \xi / 3$.

\item Frequencies $\eta_1, \eta_2$ and $\xi-\eta_1-\eta_2$ do not have the same sign.

This corresponds to the region near the {\it space-time resonances} $(\eta_1,\eta_2)=(\xi,\xi)$, $(\xi,-\xi)$, or $(-\xi,\xi)$ separately. Since the symbol $\Tb'_1(\eta_1, \eta_2, \eta_3)$ is symmetric in $\eta_1$, $\eta_2$, and $\eta_3$, it suffices to \eqref{cubdyadic} in the region near $(\xi, \xi)$.
\end{enumerate}

To estimate \eqref{cubdyadic} in the region \eqref{resregion}, we decompose the region further. Denoting $(\xi_1,\xi_2, \xi_3)=(\xi,\xi, -\xi)$ or $(\frac\xi3,\frac\xi3, \frac\xi3)$, we decompose \eqref{resregion} using the new cutoff functions $\psi_{k_1}$ and $\psi_{k_2}$.
Using the fact that
\[
\sum\limits_{(k_1,k_2)\in\Z^2}\psi_{k_1}(\eta_1-\xi_1)\psi_{k_2}(\eta_2-\xi_2)=1,
\]
we write the integral \eqref{cubdyadic} as
\begin{multline*}
\iint_{\R^2}   \Tb_1(\eta_1, \eta_2, \xi - \eta_1 - \eta_2) e^{it\Phi(\xi,\eta_1,\eta_2)} \hat h_{j_1}(\eta_1)\hat h_{j_2}(\eta_2) \hat h_{j_3}(\xi-\eta_1-\eta_2) \\
\cdot \bigg[\sum\limits_{k_1=-\infty}^{\max\{j_1,j_3\}+1}\psi_{k_1}(\eta_1-\xi_1) \bigg] \cdot \bigg[\sum\limits_{k_2=-\infty}^{\max\{j_2,j_3\}+1}\psi_{k_2}(\eta_2-\xi_2)\bigg] \diff{\eta_1} \diff{\eta_2},
\end{multline*}
where
\[
\bigg[\sum\limits_{k_1=-\infty}^{\max\{j_1,j_3\}+1}\psi_{k_1}(\eta_1-\xi_1) \bigg] \cdot \bigg[\sum\limits_{k_2=-\infty}^{\max\{j_2,j_3\}+1}\psi_{k_2}(\eta_2-\xi_2)\bigg]=1
\]
on the support of $\hat h_{j_1}(\eta_1)\hat h_{j_2}(\eta_2) \hat h_{j_3}(\xi-\eta_1-\eta_2)$.
Thus, we need to consider
\begin{align}
\label{resInt1}
\begin{split}
 & \iint_{\R^2}   \Tb_1(\eta_1, \eta_2, \xi - \eta_1 - \eta_2)  e^{it\Phi(\xi,\eta_1,\eta_2)} \hat h_{j_1}(\eta_1)\hat h_{j_2}(\eta_2) \hat h_{j_3}(\xi-\eta_1-\eta_2)\\
& \qquad\cdot \psi_{k_1}(\eta_1-\xi_1) \psi_{k_2}(\eta_2-\xi_2)  \diff{\eta_1} \diff{\eta_2}.
\end{split}
\end{align}
In this subsection, we restrict our attention to
\[
k_1\geq \log_2[\varrho(t)] \qquad \text{or} \qquad k_2\geq \log_2[\varrho(t)],
\]
 where
 \begin{align}\label{rho}
 \varrho(t)=(t+1)^{-0.49}.
 \end{align}
The case of $k_1< \log_2[\varrho(t)]$ and $k_2< \log_2[\varrho(t)]$, related to the resonant frequencies, will be discussed in Section \ref{sec:res}.

Since these expressions are symmetric in $\eta_1$ and $\eta_2$, we assume without loss of generality that $j_1 \geq k_1\geq k_2 \geq \log_2[\varrho(t)]$. The other case can be discussed in the samiliar way.

Using integrating by parts, we can write \eqref{resInt1} as
\begin{align*}
& \iint_{\R^2}   \frac{\Tb_1(\eta_1, \eta_2, \xi - \eta_1 - \eta_2)}{2it(\log|\eta_1|-\log|\xi-\eta_1-\eta_2|)}\partial_{\eta_1} e^{it\Phi(\xi,\eta_1,\eta_2)} \hat h_{j_1}(\eta_1) \hat h_{j_2}(\eta_2) \hat h_{j_3}(\xi-\eta_1-\eta_2)\\
& \hspace{2in} \cdot \psi_{k_1}(\eta_1-\xi_1) \psi_{k_2}(\eta_2-\xi_2)  \diff{\eta_1} \diff{\eta_2}\\
  =~&\frac{i}{2 t} (V_1+V_2+V_3+V_4),
\end{align*}
where
\begin{align*}
V_1(\xi, t)=&\iint_{\R^2}   \partial_{\eta_1}\left[\frac{\Tb_1(\eta_1, \eta_2, \xi - \eta_1 - \eta_2)}{\log|\eta_1|-\log|\xi-\eta_1-\eta_2|}\right] e^{it\Phi(\xi,\eta_1,\eta_2)} \hat h_{j_1}(\eta_1)\hat h_{j_2}(\eta_2) \hat h_{j_3}(\xi-\eta_1-\eta_2)\\
& \hspace{2in} \cdot \psi_{k_1}(\eta_1-\xi_1)  \psi_{k_2}(\eta_2-\xi_2)  \diff{\eta_1} \diff{\eta_2},\\[1ex]
   V_2(\xi, t)=& \iint_{\R^2}   \left[\frac{\Tb_1(\eta_1, \eta_2, \xi - \eta_1 - \eta_2)}{\log|\eta_1|-\log|\xi-\eta_1-\eta_2|}\right] e^{it\Phi(\xi,\eta_1,\eta_2)} \partial_{\eta_1}\hat h_{j_1}(\eta_1)\hat h_{j_2}(\eta_2) \hat h_{j_3}(\xi-\eta_1-\eta_2)\\
   & \hspace{2in} \cdot \psi_{k_1}(\eta_1-\xi_1) \psi_{k_2}(\eta_2-\xi_2)  \diff{\eta_1} \diff{\eta_2},\\[1ex]
  V_3(\xi, t)=& \iint_{\R^2}   \left[\frac{\Tb_1(\eta_1, \eta_2, \xi - \eta_1 - \eta_2)}{\log|\eta_1|-\log|\xi-\eta_1-\eta_2|}\right] e^{it\Phi(\xi,\eta_1,\eta_2)} \hat h_{j_1}(\eta_1)\hat h_{j_2}(\eta_2) \partial_{\eta_1}\hat h_{j_3}(\xi-\eta_1-\eta_2)\\
  & \hspace{2in} \cdot \psi_{k_1}(\eta_1-\xi_1) \psi_{k_2}(\eta_2-\xi_2)  \diff{\eta_1} \diff{\eta_2},\\[1ex]
  V_4(\xi, t)=& \iint_{\R^2}   \left[\frac{\Tb_1(\eta_1, \eta_2, \xi - \eta_1 - \eta_2)}{\log|\eta_1|-\log|\xi-\eta_1-\eta_2|}\right] e^{it\Phi(\xi,\eta_1,\eta_2)} \hat h_{j_1}(\eta_1)\hat h_{j_2}(\eta_2) \hat h_{j_3}(\xi-\eta_1-\eta_2)\\
  & \hspace{2in} \cdot \partial_{\eta_1} \psi_{k_1}(\eta_1-\xi_1) \psi_{k_2}(\eta_2-\xi_2) \diff{\eta_1} \diff{\eta_2}.
\end{align*}

{\bf Estimate of $V_1$.} We first denote the symbol for $V_1$ as
\begin{align*}
& m(\eta_1,\eta_2,\xi)\\
= ~& \frac{-2}{\log|\eta_1|-\log|\xi-\eta_1-\eta_2|} \cdot \Big[\eta_1 \log|\eta_1| - (\eta_1 + \eta_2) \log|\eta_1 + \eta_2|\\
&\hspace{2in} + (\xi - \eta_1) \log|\xi - \eta_1| - (\xi - \eta_1 - \eta_2) \log|\xi - \eta_1 - \eta_2|\Big]\\
& - \frac{\eta_1^{-1} + (\xi - \eta_1 - \eta_2)^{-1}}{(\log|\eta_1|-\log|\xi-\eta_1-\eta_2|)^2} \cdot \Big[-\eta_1^2 \log|\eta_1| - \eta_2^2 \log|\eta_2| - \eta_3^2 \log|\eta_3| - (\eta_1 + \eta_2 + \eta_3)^2 \log|\eta_1 + \eta_2 + \eta_3|\\
& \qquad + (\eta_1 + \eta_2)^2 \log|\eta_1 + \eta_2| + (\eta_1 + \eta_3)^2 \log|\eta_1 + \eta_3| + (\eta_2 + \eta_3)^2 \log|\eta_2 + \eta_3|\Big].
\end{align*}

Denote $\upsilon_i=\eta_i-\xi_i$, $i=1,2$, and it suffices to estimate
\begin{align*}
\bigg\|\iint_{\R^2} & m(\upsilon_1+\xi_1,\upsilon_2+\xi_2,\xi) e^{i t \Phi(\xi,\upsilon_1+\xi_1,\upsilon_2+\xi_2)} \hat h_{j_1}(\upsilon_1+\xi_1)\hat h_{j_2}(\upsilon_2+\xi_2) \hat h_{j_3}(\xi_3-\upsilon_1-\upsilon_2) \\
& \hspace{3.5in} \cdot \psi_{k_1}(\upsilon_1)  \psi_{k_2}(\upsilon_2) \diff{\upsilon_1} \diff{\upsilon_2}\bigg\|_{L^\infty_\xi}.
\end{align*}

Using Lemma \ref{multilinear}, we have
\begin{align*}
\|V_1\|_{L^\infty_\xi}
\lesssim~&\|\chi_{j_1,j_3}^{k_1,k_2}(\upsilon_1, \upsilon_2, \xi) m(\upsilon_1+\xi_1,\upsilon_2+\xi_2,\xi)\|_{S^\infty_{\upsilon_1,\upsilon_2}L^\infty_\xi} \\
&\cdot\| \hat \vp_{j_1}(\upsilon_1+\xi_1)\psi_{k_1}(\upsilon_1)\|_{L^2_{\upsilon_1}L^\infty_\xi} \left\|\hat \vp_{j_2}(\upsilon_2+\xi_2)\psi_{k_2}(\upsilon_2)\right\|_{L^2_{\upsilon_2}L^\infty_\xi} \|\vp_{j_3}\|_{L^\infty}
\end{align*}
where
\[
\chi_{j_1,j_3}^{k_1,k_2}(\upsilon_1,\upsilon_2,\xi)=\tilde\psi_{k_1}(\upsilon_1)\tilde\psi_{k_2}(\upsilon_2)\tilde\psi_{j_1}(\upsilon_1+\xi_1) \tilde\psi_{j_2}(\upsilon_2+\xi_2) \tilde\psi_{j_3}(\xi_3-\upsilon_1-\upsilon_2)\chi(\xi).
\]

\noindent (i) If $(\xi_1,\xi_2,\xi_3)=(\xi/3, \xi/3, \xi/3)$, since $S^\infty$-norm is rotational and scaling invariant, setting $w_1=\upsilon_1$, $w_2=-2\upsilon_1-\upsilon_2$, and using \eqref{SymEst}, we have
\begin{align*}
&\|\chi_{j_1,j_3}^{k_1,k_2}(\upsilon_1,\upsilon_2,\xi) m(\upsilon_1+\xi_1,\upsilon_2+\xi_2,\xi)\|_{S^\infty_{\upsilon_1,\upsilon_2}L^\infty_\xi}\\
 =~&\|\chi_{j_1,j_3}^{k_1,k_2}(w_1, -2w_1-w_2,\xi) m(w_1+\xi_1,-2w_1-w_2+\xi_2,\xi)\|_{S^\infty_{w_1,w_2}L^\infty_\xi}\\
 \lesssim~&\|\chi_{j_1,j_3}^{k_1,k_2}(w_1, -2w_1-w_2,\xi) m(w_1+\xi_1,-2w_1-w_2+\xi_2,\xi)\|_{L^1_{w_1w_2}L^\infty_\xi}^{1/4}\\
 &\cdot\|\partial_{w_1}^2[\chi_{j_1,j_3}^{k_1,k_2}(w_1, -2w_1-w_2,\xi) m(w_1+\xi_1,-2w_1-w_2+\xi_2,\xi)]\|_{L^1_{w_1w_2}L^\infty_\xi}^{1/2}\\
 &\cdot\|\partial_{w_1}^2\partial_{w_2}^2[\chi_{j_1,j_3}^{k_1,k_2}(w_1, -2w_1-w_2,\xi) m(w_1+\xi_1,-2w_1-w_2+\xi_2,\xi)]\|_{L^1_{w_1w_2}L^\infty_\xi}^{1/4}\\
 \lesssim~& (1 + |j_1|) (2^{j_1+k_1}2^{j_1})^{1/4} (2^{- j_1 +k_1}2^{j_1})^{1/2} (2^{- j_1 - k_1}2^{j_1})^{1/4}\\
 =~& (1 + |j_1|) \cdot 2^{(j_1+k_1)/2},
\end{align*}
where we have used the estimate
\begin{align*}
\left|\frac{\chi_{j_1,j_3}^{k_1,k_2}}{\log|w_1+\frac\xi3| - \log|\frac\xi3+w_1+w_2|}\right| &\lesssim 2^{j_1-k_1},\\
\left|\chi_{j_1,j_3}^{k_1,k_2} \partial_{w_1}^2\frac{1}{\log|w_1+\frac\xi3| - \log|\frac\xi3+w_1+w_2|}\right|& \lesssim 2^{3(j_1-k_1)}2^{2(-2 j_1 + k_1)} =2^{- j_1 - k_1},\\
\left|\chi_{j_1,j_3}^{k_1,k_2}  \partial_{w_1}^2\partial_{w_2}^2\frac{1}{\log|w_1+\frac\xi3| - \log|\frac\xi3+w_1+w_2|}\right| & \lesssim 2^{5(j_1-k_1)}2^{-2j_1} 2^{2(-2 j_1+k_1)} =2^{-j_1-3k_1}.
\end{align*}
Therefore, using \eqref{psi-L2}, \eqref{LocDis}, and \eqref{resregion}, we obtain
\begin{align}
\label{est_V1}
\begin{split}
\|V_1\|_{L^\infty_\xi}&\lesssim (1 + |j_1|) 2^{ (j_1+k_1)/2 - j_3}(t+1)^{-1}\| \hat \vp_{j_1}(\upsilon_1+\xi_1)\psi_{k_1}(\upsilon_1)\|_{L^2_{\upsilon_1}L^\infty_\xi} \left\|\hat \vp_{j_2}(\upsilon_2+\xi_2)\psi_{k_2}(\upsilon_2)\right\|_{L^2_{\upsilon_2}L^\infty_\xi} \|\px \vp_{j_3}\|_{L^\infty}\\
&\lesssim (1 + |j_1|) 2^{- 0.5 j_1+k_1+0.5k_2}\|\psi_{k_1} \hat\vp_{j_1}\|_{L^\infty_{\xi}}\|\psi_{k_2} \hat\vp_{j_2}\|_{L^\infty_\xi} \Big\{(t+1)^{-1.5}\||\xi|^{3 / 2}\hat{h}_{j_3}^{k_1,k_2}\|_{L^\infty_\xi} \\
& \hspace{1in} +(t+1)^{-1.75}\big[\||\px|^{3 / 4} P_{j_3}^{k_1,k_2}(x\px h)\|_{L^2}+\||\px|^{3 / 4} h_{j_3}^{k_1,k_2}\|_{L^2}  \big] \Big\}.
\end{split}
\end{align}

\noindent (ii) If $(\xi_1,\xi_2,\xi_3)=(\xi,\xi,-\xi)$, we use \eqref{SymEst} to obtain
\begin{align*}
&\|\chi_{j_1,j_3}^{k_1,k_2}(\upsilon_1,\upsilon_2,\xi) m(\upsilon_1+\xi_1,\upsilon_2+\xi_2,\xi)\|_{S^\infty_{\upsilon_1,\upsilon_2}L^\infty_\xi}\\
 \lesssim~& \|\chi_{j_1,j_3}^{k_1,k_2}(\upsilon_1,\upsilon_2,\xi) m(\upsilon_1+\xi_1,\upsilon_2+\xi_2,\xi)\|_{L^1_{\upsilon_1\upsilon_2}}^{1/4}\|\partial_{\upsilon_1}^2[\chi_{j_1,j_3}^{k_1,k_2}(\upsilon_1,\upsilon_2,\xi) m(\upsilon_1+\xi_1,\upsilon_2+\xi_2,\xi)]\|_{L^1_{\upsilon_1\upsilon_2}}^{1/2}\\
 &\cdot\|\partial_{\upsilon_1}^2\partial_{\upsilon_2}^2[\chi_{j_1,j_3}^{k_1,k_2}(\upsilon_1,\upsilon_2,\xi) m(\upsilon_1+\xi_1,\upsilon_2+\xi_2,\xi)]\|_{L^1_{\upsilon_1\upsilon_2}}^{1/4}
\\
 \lesssim~ & (1 + |j_1|) (2^{j_1+k_1}2^{j_1})^{1/4} (2^{-j_1+k_1}2^{j_1})^{1/2} (2^{-j_1-2k_2+k_1}2^{j_1})^{1/4} \\
 =~& (1 + |j_1|) \cdot 2^{ \frac12j_1+k_1-\frac12k_2},
\end{align*}
where we have used the estimates
\begin{align*}
\left|\frac{\chi_{j_1,j_3}^{k_1,k_2}}{\log|\upsilon_1+\xi| - \log|-\xi-\upsilon_1-\upsilon_2|}\right| &\lesssim 2^{j_1-k_2},\\
\left|\chi_{j_1,j_3}^{k_1,k_2}\partial_{\upsilon_1}^2\frac{1}{\log|\upsilon_1+\xi| - \log|-\xi-\upsilon_1-\upsilon_2|}\right| &\lesssim 2^{3(j_1-k_2)}2^{2(-2 j_1+k_2)}=2^{-j_1-k_2},\\
 \left|\chi_{j_1,j_3}^{k_1,k_2}\partial_{\upsilon_1}^2\partial_{\upsilon_2}^2\frac{1}{\log|\upsilon_1+\xi| - \log|-\xi-\upsilon_1-\upsilon_2|}\right| &\lesssim 2^{5(j_1-k_2)}2^{-2 j_1}2^{2(-2 j_1 + k_2)}=2^{- j_1-3k_2}.
\end{align*}

Therefore, using \eqref{psi-L2}, \eqref{LocDis}, and \eqref{resregion}
\begin{align}
\begin{split}
\|V_1\|_{L^\infty_\xi}
&\lesssim (1 + |j_1|) 2^{0.5 j_1 - j_3 + k_1 - 0.5k_2}(t+1)^{-1}\| \hat \vp_{j_1}(\upsilon_1+\xi_1)\psi_{k_1}(\upsilon_1)\|_{L^2_{\upsilon_1}L^\infty_\xi} \left\|\hat \vp_{j_2}(\upsilon_2+\xi_2)\psi_{k_2}(\upsilon_2)\right\|_{L^2_{\upsilon_2}L^\infty_\xi} \|\px \vp_{j_3}\|_{L^\infty}\\
&\lesssim (1 + |j_1|) 2^{- 0.5 j_1 + 1.5k_1}\|\psi_{k_1} \hat\vp_{j_1}\|_{L^\infty_{\xi}} \|\psi_{k_2} \hat\vp_{j_2}\|_{L^\infty_\xi} \Big\{(t+1)^{-1.5}\||\xi|^{3 / 2}\hat{h}_{j_3}\|_{L^\infty_\xi} \\
& \hspace{1in} +(t+1)^{-1.75}\big[\||\px|^{3 / 4} P_{j_3}(x\px h)\|_{L^2}+\||\px|^{3 / 4} h_{j_3}\|_{L^2}  \big] \Big\}.
\end{split}
\end{align}

{\bf Estimates of $V_2$--$V_4$.}
The estimates for $V_2$--$V_4$ are similar to $V_1$. We omit the details here. The resulting estimates are as follows.

\noindent (i) If $(\xi_1,\xi_2, \xi_3)=(\frac\xi3,\frac\xi3, \frac\xi3)$, the symbol can be estimated as
\begin{align*}
&\left\|\frac{\Tb'_1(\xi_1+\upsilon_1, \xi_2+\upsilon_2, \xi_3-\upsilon_1-\upsilon_2)}{\log|\xi_1+\upsilon_1| - \log|\xi_3-\upsilon_1-\upsilon_2|}\right\|_{S^\infty_{\upsilon_1\upsilon_2}L^\infty_\xi} \lesssim 2^{1.5 j_1+0.5k_1}.
\end{align*}

\noindent (ii) If $(\xi_1,\xi_2, \xi_3)=(\xi,\xi, -\xi)$, the symbol can be estimated as
\begin{align*}
&\left\|\chi_{j_1, j_3}^{k_1, k_2}(\upsilon_1, \upsilon_2, \xi) \frac{\Tb'_1(\xi_1+\upsilon_1, \xi_2+\upsilon_2, \xi_3-\upsilon_1-\upsilon_2)}{\log|\xi_1+\upsilon_1| - \log|\xi_3-\upsilon_1-\upsilon_2|}\right\|_{S^\infty_{\upsilon_1\upsilon_2}L^\infty_\xi}\\
\lesssim~&(2^{j_1+k_1}2^{2j_1})^{1/4} (2^{- j_1+k_1}2^{2j_1})^{1/2} (2^{- j_1-2k_2+k_1}2^{2j_1})^{1/4} \\
 =~& (1 + |j_1|) \cdot 2^{1.5 j_1+k_1-0.5k_2}.
\end{align*}

In either case, we have the following estimates
\begin{align}
\begin{split}
 \|V_2\|_{L^\infty_\xi}
&\lesssim (1 + |j_1|) 2^{-0.5 j_1+k_1}\|\eta_1 \partial_{\eta_1}\hat\vp_{j_1}(\eta_1)\|_{L^2_{\eta_1}} \|\psi_{k_2} \hat\vp_{j_2}\|_{L^\infty_\xi}\\*
& \qquad \cdot  \Big\{(t+1)^{-1.5}\||\xi|^{3 / 2}\hat{h}_{j_3}\|_{L^\infty_\xi} +(t+1)^{-1.75}\big[\||\px|^{3 / 4} P_{j_3}(x\px h)\|_{L^2}+\||\px|^{3 / 4} h_{j_3}\|_{L^2}  \big] \Big\},
\end{split}\\
\begin{split}
\|V_3\|_{L^\infty_\xi}
&\lesssim (1 + |j_1|) 2^{-0.5 j_1+k_1}\|\eta_3 \partial_{\eta_3}\hat\vp_{j_3}(\eta_3)\|_{L^2_{\eta_3}}\|\psi_{k_2} \hat\vp_{j_2}\|_{L^\infty_\xi}\\*
& \qquad \cdot \Big\{(t+1)^{-1.5}\||\xi|^{3 / 2}\hat{h}_{j_1}\|_{L^\infty_\xi} +(t+1)^{-1.75}\big[\||\px|^{3 / 4} P_{j_1}(x\px h)\|_{L^2}+\||\px|^{3 / 4} h_{j_1}\|_{L^2}  \big] \Big\},
\end{split}\\
\begin{split}
\|V_4\|_{L^\infty_\xi}
&\lesssim (1 + |j_1|) 2^{0.5 j_1+0.5k_1} \|\psi_{k_1} \hat\vp_{j_1}\|_{L^\infty_{\xi}}\|\psi_{k_2} \hat\vp_{j_2}\|_{L^\infty_\xi}\\
& \qquad \cdot \Big\{(t+1)^{-1.5}\||\xi|^{3 / 2}\hat{h}_{j_3}\|_{L^\infty_\xi} +(t+1)^{-1.75}\big[\||\px|^{3 / 4} P_{j_3}(x\px h)\|_{L^2}+\||\px|^{3 / 4} h_{j_3}\|_{L^2}  \big] \Big\}.\label{est_V4}
\end{split}
\end{align}

Now we take the summation over $\log_2[\varrho(t)] \leq k_1, k_2 \leq \max\{j_1, j_3\} + 1$, and combine the estimates \eqref{est_V1}--\eqref{est_V4} to get
\begin{align*}
& \Bigg\|\xi (|\xi| + |\xi|^{r + 4}) \cutoffxi(\xi, t) \iint_{\R^2} \Tb_1'(\eta_1, \eta_2, \xi - \eta_1 - \eta_2) e^{i A t \Phi(\xi,\eta_1,\eta_2)} \hat h_{j_1}(\eta_1)\hat h_{j_2}(\eta_2) \hat h_{j_3}(\xi-\eta_1-\eta_2) \\*
& \hspace{2in} \cdot \bigg[\sum\limits_{k_1= \log_2[\varrho(t)]}^{\max\{j_1,j_3\}+1}\psi_{k_1}(\eta_1-\xi_1) \bigg] \cdot \bigg[\sum\limits_{k_2=\log_2[\varrho(t)]}^{\max\{j_2,j_3\}+1}\psi_{k_2}(\eta_2-\xi_2)\bigg] \diff{\eta_1} \diff{\eta_2}\Bigg\|_{L^\infty_\xi}\\
& \lesssim (1 + |j_1|) \left[\max\{j_1, j_3\} - \log_2[\varrho(t)]\right]^2 (t + 1)^{(r + 4) p_1}\\
& \qquad \cdot \left[\||\xi| \psi_{k_1} \hat\vp_{j_1}\|_{L^\infty_{\xi}}\||\xi| \psi_{k_2} \hat\vp_{j_2}\|_{L^\infty_\xi} + \|\eta_1 \partial_{\eta_1} \hat{\vp}_{j_1}(\eta_1)\|_{L^2_{\eta_1}} \||\xi| \psi_{k_2} \hat\vp_{j_2}\|_{L^\infty_\xi} +  \||\xi| \psi_{k_1} \hat\vp_{j_1}\|_{L^\infty_\xi} \|\eta_2 \partial_{\eta_2} \hat{\vp}_{j_2}(\eta_2)\|_{L^2_{\eta_2}}\right]\\
& \qquad \cdot \Big\{(t+1)^{-1.5}\||\xi|^{3 / 2}\hat{h}_{j_3}\|_{L^\infty_\xi} +(t+1)^{-1.75}\big[\||\px|^{3 / 4} P_{j_3}(x\px h)\|_{L^2}+\||\px|^{3 / 4} h_{j_3}\|_{L^2}  \big] \Big\}.
\end{align*}

The right-hand-side is summable with respect to $j_1, j_2, j_3$ under $|j_3-j_2|\leq 1$ and $|j_3-j_1|\leq 1$, since we can write
\begin{align*}
\begin{split}
& \||\xi|^{3 / 2} \hat h_j\|_{L^\infty_\xi}\lesssim  (\mathbf1_{j\leq 0}2^{j / 2}+\mathbf1_{j>0}2^{(-r - 3 / 2)j})\|h_j\|_{Z},
\end{split}
\end{align*}
and the resulting sum is integrable for $t\in (0,\infty)$.

\subsection{Resonant frequencies} \label{sec:res}
In this section, we estimate \eqref{resInt1} in the region
\[
|j_1-j_3|\leq 1, \qquad |j_2-j_3|\leq 1, \qquad k_1 < \log_2[\varrho(t)],\qquad k_2 < \log_2[\varrho(t)],
\]
and then sum over $k_1, k_2 < \log_2(\varrho(t))$.
After taking the sum, the cutoff function of the integrand is
\begin{align*}
\cutoff(\xi,\eta_1,\eta_2,t):=\psi\bigg(\frac{\eta_1 - \xi_1}{\varrho(t)}\bigg) \cdot \psi\bigg(\frac{\eta_2 - \xi_2}{\varrho(t)}\bigg).
\end{align*}
The support of this cutoff function is
\[
\left\{(\eta_1,\eta_2)\in \R^2\mid |\eta_1 - \xi_1|<\frac85 \varrho(t), ~|\eta_2 - \xi_2\big|<\frac85 \varrho(t)\right\},
\]
which can be written as the union of four disjoint sets $A_1\cup A_2\cup A_3\cup A_4$, where
\begin{align*}
A_1&=\bigg\{(\eta_1,\eta_2)~~\bigg{|}~~ \bigg|\eta_1-\frac{\xi}3\bigg|<\frac85 \varrho(t),\quad \bigg|\eta_2-\frac{\xi}3\bigg|<\frac85 \varrho(t)\bigg\},\\
A_2&=\bigg\{(\eta_1,\eta_2)~~\bigg{|}~~ \big|\eta_1-\xi\big|<\frac85 \varrho(t),\quad \big|\eta_2-\xi\big|<\frac85 \varrho(t)\bigg\},\\
A_3&=\bigg\{(\eta_1,\eta_2)~~\bigg{|}~~\big|(\eta_1-\xi)\big|<\frac85 \varrho(t),\quad \big|\eta_2-(-\xi)\big|<\frac85 \varrho(t) \bigg\},\\
A_4&=\bigg\{(\eta_1,\eta_2)~~\bigg{|}~~\big|\eta_1-(-\xi)\big|<\frac85 \varrho(t),\quad \big|\eta_2-\xi\big|<\frac85 \varrho(t) \bigg\}.
\end{align*}
We notice that the regions $A_1$, $A_2$, $A_3$, $A_4$ are discs centered at $(\xi/3, \xi/3)$, $(\xi, \xi)$, $(\xi, -\xi)$, and $(-\xi, \xi)$, respectively. The region $A_1$ corresponds to space resonances $\xi = \xi/3 + \xi/3 + \xi/3$, while $A_2$, $A_3$, $A_4$ correspond to space-time resonances $\xi = \xi + \xi -\xi$.

\subsubsection{Space resonances}
When $(\eta_1,\eta_2)\in A_1$,
we can expand $\Tb_1 / \Phi$ around $(\xi, \xi/3, \xi/3)$ as
\begin{equation}\label{T1Phi}
\frac{\Tb_1(\eta_1, \eta_2, \xi - \eta_1 - \eta_2)}{\Phi(\xi,\eta_1,\eta_2) }= \bigg(\frac12 - \frac{2 \log 2}{3 \log 3}\bigg) \xi + O\bigg(\left|\eta_1-\frac\xi3\right|^2+\left|\eta_2-\frac\xi3\right|^2\bigg).
\end{equation}

After writing
\begin{align*}
e^{i\tau \Phi(\xi, \eta_1,\eta_2)}=\frac{1}{i\Phi(\xi, \eta_1,\eta_2)} \left[\partial_\tau e^{i\tau\Phi(\xi, \eta_1,\eta_2)}\right],
\end{align*}
and integrating by parts with respect to $\tau$, we get that
\[
\begin{aligned}
&\int_1^t i\xi\iint_{\R^2}  \Tb_1(\eta_1, \eta_2, \xi - \eta_1 - \eta_2)e^{i\tau\Phi(\xi,\eta_1,\eta_2)} \hat h_{j_1}(\eta_1, \tau)\hat h_{j_2}(\eta_2, \tau) \hat h_{j_3}(\xi-\eta_1-\eta_2, \tau) \cutoff(\xi,\eta_1,\eta_2,\tau)  \diff{\eta_1} \diff{\eta_2} \diff{\tau}\\[1ex]
=& \int_1^t \xi\iint_{\R^2}  \frac{\Tb_1(\eta_1, \eta_2, \xi - \eta_1 - \eta_2)}{\Phi(\xi, \eta_1,\eta_2)} \partial_\tau \left[e^{i\tau\Phi(\xi, \eta_1,\eta_2)}\right] \hat h_{j_1}(\eta_1, \tau)\hat h_{j_2}(\eta_2, \tau) \hat h_{j_3}(\xi-\eta_1-\eta_2, \tau) \cutoff(\xi,\eta_1,\eta_2,\tau)  \diff{\eta_1} \diff{\eta_2} \diff{\tau}\\[1ex]
=&  J_1-\int_1^t J_2(\tau)+J_3(\tau) \diff{\tau},
\end{aligned}
\]
where
\begin{align*}
J_1 &=  \iint_{\R^2}  \xi \frac{\Tb_1(\eta_1, \eta_2, \xi - \eta_1 - \eta_2)}{\Phi(\xi, \eta_1,\eta_2)} \hat h_{j_1}(\eta_1, \tau)\hat h_{j_2}(\eta_2, \tau) \hat h_{j_3}(\xi-\eta_1-\eta_2, \tau)  e^{i\tau\Phi(\xi, \eta_1,\eta_2)}\cutoff(\xi,\eta_1,\eta_2,\tau) \diff{\eta_1} \diff{\eta_2}\Big|_{\tau=1}^{\tau=t},\\[1ex]
J_2(\tau) &= \xi \iint_{\R^2} \frac{\Tb_1(\eta_1, \eta_2, \xi - \eta_1 - \eta_2)}{\Phi(\xi, \eta_1,\eta_2)}  e^{i\tau\Phi(\xi, \eta_1,\eta_2)} \partial_\tau \left[\hat h_{j_1}(\eta_1, \tau)\hat h_{j_2}(\eta_2, \tau) \hat h_{j_3}(\xi-\eta_1-\eta_2, \tau)\right]\cutoff(\xi,\eta_1,\eta_2,\tau) \diff{\eta_1} \diff{\eta_2},\\[1ex]
J_3(\tau)& = \xi\iint_{\R^2} \partial_\tau\cutoff(\xi,\eta_1,\eta_2,\tau) \frac{\Tb_1(\eta_1, \eta_2, \xi - \eta_1 - \eta_2)}{\Phi(\xi, \eta_1,\eta_2)} e^{i\tau\Phi(\xi, \eta_1,\eta_2)} \hat h_{j_1}(\eta_1, \tau)\hat h_{j_2}(\eta_2, \tau) \hat h_{j_3}(\xi-\eta_1-\eta_2, \tau)  \diff{\eta_1} \diff{\eta_2} .
\end{align*}
For $J_1$, for any $\tau\geq1$, we have from \eqref{T1Phi} that
\[
\begin{aligned}
&\bigg|(|\xi|+|\xi|^{r + 4})\iint_{\R^2} \cutoff(\xi,\eta_1,\eta_2,t)  \xi \frac{\Tb_1(\eta_1, \eta_2, \xi - \eta_1 - \eta_2)}{\Phi(\xi, \eta_1,\eta_2)} \hat h_{j_1}(\eta_1, \tau)\hat h_{j_2}(\eta_2, \tau) \hat h_{j_3}(\xi-\eta_1-\eta_2, \tau)   e^{i\tau\Phi(\xi, \eta_1,\eta_2)}\diff{\eta_1} \diff{\eta_2}\bigg|\\
\lesssim~~& \bigg|(|\xi|+|\xi|^{r + 4})\iint_{\R^2} \cutoff(\xi,\eta_1,\eta_2,t)  \xi^2 \hat h_{j_1}(\eta_1, \tau)\hat h_{j_2}(\eta_2, \tau) \hat h_{j_3}(\xi-\eta_1-\eta_2, \tau) e^{i\tau\Phi(\xi, \eta_1,\eta_2)} \diff{\eta_1} \diff{\eta_2}\bigg|\\
&+ (|\xi|+|\xi|^{r + 4})\iint_{\R^2} \cutoff(\xi,\eta_1,\eta_2,t)   [\varrho(\tau)]^2 \left| \hat h_{j_1}(\eta_1, \tau)\hat h_{j_2}(\eta_2, \tau) \hat h_{j_3}(\xi-\eta_1-\eta_2, \tau) \right|\diff{\eta_1} \diff{\eta_2}\\
\lesssim~~&(\tau+1)^{2p_0+(r+3)p_1}\||\xi| \hat h_{j_1}\|_{L^\infty_\xi}\||\xi| \hat h_{j_2}\|_{L^\infty_\xi}\||\xi| \hat h_{j_3}\|_{L^\infty_\xi}\big([\varrho(\tau)]^{2}+[\varrho(\tau)]^{4}\big).
\end{aligned}
\]
Notice that in $A_1$, the number of summations of $j_1$, $j_2$, and $j_3$ is or order $\log(t+1)$, and therefore, the right-hand-side is bounded for $\tau \geq 1$ after the summation in $j_1$, $j_2$, and $j_3$.

After taking the time derivative, the term $J_2$ can be written as a sum of three terms.
\begin{align*}
&\xi \iint_{\R^2} \frac{\Tb_1(\eta_1, \eta_2, \xi - \eta_1 - \eta_2)}{\Phi(\xi, \eta_1,\eta_2)}  e^{i\tau\Phi(\xi, \eta_1,\eta_2)}  \left[\partial_\tau\hat h_{j_1}(\eta_1, \tau)\hat h_{j_2}(\eta_2, \tau) \hat h_{j_3}(\xi-\eta_1-\eta_2, \tau)\right]\cutoff(\xi,\eta_1,\eta_2,\tau) \diff{\eta_1} \diff{\eta_2},\\[1ex]
&\xi \iint_{\R^2} \frac{\Tb_1(\eta_1, \eta_2, \xi - \eta_1 - \eta_2)}{\Phi(\xi, \eta_1,\eta_2)}  e^{i\tau\Phi(\xi, \eta_1,\eta_2)}  \left[\hat h_{j_1}(\eta_1, \tau) \partial_\tau \hat h_{j_2}(\eta_2, \tau) \hat h_{j_3}(\tau, \xi-\eta_1-\eta_2)\right]\cutoff(\xi,\eta_1,\eta_2,\tau) \diff{\eta_1} \diff{\eta_2},
\end{align*}
and
\begin{align*}
&\xi \iint_{\R^2} \frac{\Tb_1(\eta_1, \eta_2, \xi - \eta_1 - \eta_2)}{\Phi(\xi, \eta_1,\eta_2)}  e^{i\tau\Phi(\xi, \eta_1,\eta_2)}  \left[\hat h_{j_1}(\eta_1, \tau)\hat h_{j_2}(\eta_2, \tau) \partial_\tau\hat h_{j_3}(\xi-\eta_1-\eta_2, \tau)\right]\cutoff(\xi,\eta_1,\eta_2,\tau) \diff{\eta_1} \diff{\eta_2}.
\end{align*}
Notice that by \eqref{eqhhat},and the bootstrap assumptions and Lemma \ref{sharp}, we have
\begin{align*}
&\|\partial_t\hat h\|_{L^\infty_\xi}\lesssim \bigg\|\xi\iint_{\R^2}\Tb_1(\eta_1,\eta_2,\xi-\eta_1-\eta_2) e^{it\Phi(\xi,\eta_1,\eta_2)}\hat h(\xi-\eta_1-\eta_2)\hat h(\eta_1)\hat h(\eta_2)\diff\eta_1\diff\eta_2\bigg\|_{L^\infty_\xi}+\left\|\widehat{\mathcal N_{\geq 5}(\vp)}\right\|_{L^\infty_\xi}\\
 \lesssim~~ &\left\|\partial_x \bigg\{\varphi^2 \log|\partial_x| \varphi_{xx}
- \varphi \log|\partial_x| (\varphi^2)_{xx} + \frac 13 \log|\partial_x| (\varphi^3)_{xx}\bigg\}\right \|_{L^1}+\left\|\mathcal N_{\geq 5}(\vp)\right\|_{L^1}\\
\lesssim~~ &\|\vp\|_{H^s}^2 \cdot \sum\limits_{j=0}^{\infty} \left(\|\vp_x\|_{W^{3,\infty}}^{2 j + 1} + \|L \vp_x\|_{W^{3, \infty}}^{2 j + 1}\right)\\
\lesssim~~ & \ve_1^3 (t + 1)^{2p_0-\frac12}.
\end{align*}
Therefore, we obtain
\[
\begin{aligned}
\left| (|\xi|+|\xi|^{r + 4}) J_2(\tau)\right|
\lesssim \sum \| h_{\ell_1}\|_{Z}\|\partial_\tau\hat h_{\ell_2}\|_{L^\infty_\xi}\| h_{\ell_3}\|_{Z}[\varrho(\tau)]^{2}
\lesssim \ve_1^3 (\tau + 1)^{p_0-\frac12}[\varrho(t)]^{2}\sum \| h_{\ell_1}\|_{Z}\| h_{\ell_3}\|_{Z},
\end{aligned}
\]
where we sum over all permutations $(\ell_1, \ell_2, \ell_3)$ of $(j_1, j_2,j_3)$ in the space resonance region $A_1$. Again, we notice that the number of summations is of order $\log(\tau + 1)$, and the resulting sum is integrable for $\tau \in (1, \infty)$.

As for the term $J_3$, by the definition of the cutoff function, we have
\[
\left|\partial_\tau\left[\psi_{\leq \log_2(\varrho(\tau))}(|\eta_1|-|\xi-\eta_1-\eta_2|) \cdot \psi_{\leq \log_2(\varrho(\tau))}(|\eta_2|-|\xi-\eta_1-\eta_2|)\right] \right| \lesssim \varrho'_1(t)(\varrho(t))^{-1}\lesssim \frac{1}{t+1}.
\]
The area of its support is of the order of $[\varrho(t)]^2$.
Then, using \eqref{T1Phi}, we get that
\begin{align*}
\left|(|\xi|+|\xi|^{r + 4})J_3(\tau)\right|\lesssim (\tau+1)^{2p_0+(r+3)p_1-1} [\varrho(\tau)]^{2} \sum \|\xi\hat h_{\ell_1}\|_{L^\infty_\xi}\|\xi\hat h_{\ell_2}\|_{L^\infty_\xi}\|\xi\hat h_{\ell_3}\|_{L^\infty_\xi},
\end{align*}
where the summation is taken over permutations $(\ell_1, \ell_2, \ell_3)$ of $(j_1, j_2,j_3)$, so the sum converges and is integrable for $\tau \in (1, \infty)$.

\subsubsection{Space-time resonances}
We now use modified scattering to consider the term
\begin{align*}
\frac16 \iint_{A_2\bigcup A_3\bigcup A_4}  &i \xi \cutoff(\xi,\eta_1,\eta_2,t) \Tb_1(\eta_1, \eta_2, \xi - \eta_1 - \eta_2) \hat\vp_{j_1}(\eta_1)\hat\vp_{j_2}(\eta_2) \hat\vp_{j_3}(\xi-\eta_1-\eta_2) \diff{\eta_1} \diff{\eta_2}\\
& -i\xi \left[\beta_1(t)\Tb_1(\xi, \xi, -\xi)+\beta_2(t)\Tb_1(\xi, -\xi, \xi)+\beta_3(t)\Tb_1(-\xi,\xi,\xi)\right]|\hat \vp(\tau,\xi)|^2\hat \vp(\tau,\xi).
\end{align*}

For $A_2$, we take
\[
\beta_1(t)=\frac16 \iint_{A_2}\cutoff(\xi,\eta_1,\eta_2,t)\diff\eta_1\diff\eta_2.
\]
Therefore, using a Taylor expansion and \eqref{partialT}, we obtain
\begin{align*}
&\bigg|(|\xi|+|\xi|^{r+4})\frac16 i\xi \iint_{A_2} \cutoff(\xi,\eta_1,\eta_2,t)\\
&\qquad\Big[\Tb_1(\eta_1, \eta_2, \xi - \eta_1 - \eta_2) \hat\vp_{j_1}(\eta_1)\hat\vp_{j_2}(\eta_2) \hat\vp_{j_3}(\xi-\eta_1-\eta_2) - \Tb_1(\xi, \xi, -\xi)|\hat \vp(\tau,\xi)|^2\hat \vp(\tau,\xi) \Big] \diff{\eta_1} \diff{\eta_2}\bigg|\\
\lesssim~& (|\xi|+|\xi|^{r+4}) (i\xi) \iint_{A_2} \left|\partial_{\eta_1}\left[\Tb_1(\eta_1, \eta_2, \xi - \eta_1 - \eta_2) \hat\vp_{j_1}(\eta_1)\hat\vp_{j_2}(\eta_2) \hat\vp_{j_3}(\xi-\eta_1-\eta_2)\right] \bigg|_{\eta_1 = \eta_1'} (\xi-\eta_1)\right|\\
& \hspace {.75in}+\left|\partial_{\eta_2}\left[\Tb_1(\eta_1, \eta_2, \xi - \eta_1 - \eta_2) \hat\vp_{j_1}(\eta_1)\hat\vp_{j_2}(\eta_2) \hat\vp_{j_3}(\xi-\eta_1-\eta_2)\right] \bigg|_{\eta_2 = \eta_2'} (\xi-\eta_2)\right| \diff{\eta_1} \diff{\eta_2}\\
\lesssim~&(t+1)^{(r+3)p_1}  \|\xi\hat\vp_{j_1}\|_{L^\infty_\xi}\|\xi\hat\vp_{j_2}\|_{L^\infty_\xi}\|\xi\hat\vp_{j_3}\|_{L^\infty_\xi} [\varrho(t)]^{3}+ \sum\|\xi\hat\vp_{\ell_1}\|_{L^\infty_\xi}\|\xi\hat\vp_{\ell_2}\|_{L^\infty_\xi}\|\S\vp_{\ell_3}\|_{H^r}[\varrho(t)]^{5/2},
\end{align*}
where $\eta_1'$ (or $\eta_2'$) in the first inequality is some number between $\xi$ and $\eta_1$ (or $\eta_2$), and the summation in the second inequality is over permutations $(\ell_1, \ell_2, \ell_3)$ of $(j_1, j_2,j_3)$. The estimates for $A_3$ and $A_4$ follow by a similar argument.

Taking a summation over $j_1, j_2, j_3$ and using the estimates in the above subsections together with the time-decay of $\varrho(t)$ in \eqref{rho}, we conclude that
\[
\int_0^{\infty} \|(|\xi|+|\xi|^{r+4})U_1\|_{L^\infty_\xi} \diff t\lesssim \ve_0.
\]

\subsection{Higher-degree terms}\label{higherorder}
In this subsection, we prove that
\[
\left\|(|\xi|+|\xi|^{r+4})\widehat{\mathcal{N}_{\geq5}(\vp)}\right\|_{L^\infty_\xi}
\]
is integrable in time. We begin by proving an estimate for the symbol $\Tb_n$. We have
\begin{align*}
&\F^{-1}\left[  \Tb_n(\eta_1,\eta_2,\dotsc, \eta_{2n+1}) \psi_{j_1}(\eta_1)\psi_{j_2}(\eta_2)\dotsm\psi_{j_{2n+1}}(\eta_{2n+1})\right]\\
=~&\iiint_{\R^{2n+1}}e^{i(y_1\eta_1+y_2\eta_2+\dotsb+y_{2n+1}\eta_{2n+1})}   \bigg[\int_{\R}\frac{\prod_{j=1}^{2n+1}(1-e^{i\eta_j\zeta})}{|\zeta|^{2n+1}}\diff \zeta\bigg]\psi_{j_1}(\eta_1)\psi_{j_2}(\eta_2) \dotsm \psi_{j_{2n+1}}(\eta_{2n+1}) \diff{\etab_n}\\
=~&\iiint_{\R^{2n+1}}  \bigg[\int_{\R}\frac{ (e^{iy_1\eta_1}-e^{i\eta_1(\zeta+y_1)})\dotsm(e^{iy_{2n+1}\eta_{2n+1}}-e^{i\eta_{2n+1}(\zeta+y_{2n+1})})}{ |\zeta|^{2n+1}}\diff \zeta\bigg]\psi_{j_1}(\eta_1)\dotsm\psi_{j_{2n+1}}(\eta_{2n+1}) \diff{\etab_n}\\
=~& \int_{\R}\frac{1 }{|\zeta|^{2n+1}} \left[\F^{-1}[\psi_{j_1}](y_1)-\F^{-1}[\psi_{j_1}](\zeta+y_1)\right]\dotsm \left[\F^{-1}[\psi_{j_{2n+1}}](y_{2n+1})-\F^{-1}[\psi_{j_{2n+1}}](\zeta+y_{2n+1})\right] \diff \zeta,
\end{align*}
and it follows that
\begin{align*}
&\left\|\F^{-1}\left[  \Tb_n(\eta_1,\eta_2,\dotsc, \eta_{2n+1}) \psi_{j_1}(\eta_1)\psi_{j_2}(\eta_2)\dotsm\psi_{j_{2n+1}}(\eta_{2n+1})\right]\right\|_{L^1}\\
&\qquad\lesssim \int_{\R} 2^{j_1+\dotsb+j_{2n+1}}\frac1{|\zeta|^{2n+1}}\min\{2^{-j_1},|\zeta|\}\min\{2^{-j_2},|\zeta|\}
\dotsm\min\{2^{-j_{2n+1}},|\zeta|\}\diff\zeta.\\
\end{align*}
Let $\ell_1, \ell_2, \dotsc, \ell_{2n+1}$ be a permutation of $j_1, j_2, \dotsc, j_{2n+1}$ satisfying $2^{-\ell_1}\leq 2^{-\ell_2}\leq \dotsb \leq 2^{-\ell_{2n+1}}$. Then
\begin{align*}
&\left\|\F^{-1}\left[  \Tb_n(\eta_1,\eta_2,\dotsc, \eta_{2n+1}) \psi_{j_1}(\eta_1)\psi_{j_2}(\eta_2)\dotsm\psi_{j_{2n+1}}(\eta_{2n+1})\right]\right\|_{L^1}\\
&\qquad \lesssim \int_{|\zeta|> 2^{-\ell_{2n+1}}}\frac1{|\zeta|^{2n+1}}\diff\zeta+\int_{2^{-\ell_{2n}}<|\zeta|<2^{-\ell_{2n+1}}}\frac{2^{\ell_1}}{|\zeta|^{2n}}\diff\zeta\\
& \qquad +\dotsb+\int_{2^{-\ell_1}<|\zeta|<2^{-\ell_{2}}}\frac{2^{\ell_1+\dotsb+\ell_{2n}}}{|\zeta|}\diff \zeta+\int_{|\zeta|<2^{-\ell_1}}2^{\ell_1+\dotsb+\ell_{2n+1}}\diff \zeta\\
&\qquad \lesssim 2^{\ell_2+\dotsb+\ell_{2n+1}}.
\end{align*}

Therefore, by Lemma \ref{multilinear}, we have
\begin{align*}
\left\|(|\xi|+|\xi|^{r+4})\widehat{\mathcal{N}_{\geq5}(\vp)}\right\|_{L^\infty_\xi}\lesssim& ~(t+1)^{(r+4)p_1}\|\mathcal{N}_{\geq 5}(\vp)\|_{L^1} \lesssim \|\vp\|_{H^1}^2\sum\limits_{n=2}^\infty \left(\|\vp_x\|_{L^\infty}^{2n-1} + \|L \vp_x\|_{L^\infty}^{2n-1}\right).
\end{align*}
Using the dispersive estimate Lemma~\ref{sharp}, we see that the right-hand-side is integrable in $t$, which leads to
\[
\int_0^\infty \left\| (|\xi|+|\xi|^{r+3})\widehat{\mathcal N_{\geq5}(\vp)} \right\|_{L_\xi^\infty}\diff t\lesssim \ve_0.
\]
This completes the proof of Theorem~\ref{global}.

\appendix
\section{Alternative formulation of the SQG front equation}
We first prove an algebraic identity that will be used in deriving \eqref{expd-sqg}.
\begin{lemma}
\label{p-cancel}
Let $N \ge 2$ be an integer. Then for any integer $1\le p \le N - 1$ and any $\eta_j \in \R$, $j = 1, 2, \dotsc, N$
\begin{equation}
\label{eta-cancel}
\sum_{\ell = 1}^{N} \sum_{1 \leq m_1 < m_2 < \dotsb < m_\ell \leq N} (-1)^\ell (\eta_{m_1} + \eta_{m_2} + \dotsb + \eta_{m_\ell})^p = 0.
\end{equation}
\end{lemma}
\begin{proof}
A general term in the expansion of left-hand-side of \eqref{eta-cancel} is proportional to
\begin{equation}
\label{etageneralterm}
\eta_1^{\alpha_1} \eta_2^{\alpha_2} \dotsm \eta_N^{\alpha_N},
\end{equation}
where $\alpha_1, \alpha_2, \dotsc, \alpha_N$ are nonnegative integers such that
$\alpha_1 + \alpha_2 + \dotsb +\alpha_N = p$.
It suffices to show that the coefficients of the monomials \eqref{etageneralterm} are zero. Let $1\le M\le N-1$ denote the number of nonzero terms in the list $(\alpha_1, \alpha_2, \dotsc, \alpha_N)$. Using the multinomial theorem, we see that the coefficient of \eqref{etageneralterm} is
\[
\begin{pmatrix}p\\ \alpha_1, \dotsc, \alpha_N\end{pmatrix} \cdot \sum_{j = 0}^{N - M} (-1)^{M + j} \begin{pmatrix}N - M\\ j\end{pmatrix} = \begin{pmatrix}p\\ \alpha_1, \dotsc, \alpha_N\end{pmatrix} \cdot (-1)^M (1 - 1)^{N - M} = 0.
\]
\end{proof}

To compute $\Tb_n(\etab_n)$ in \eqref{Tnintdef}, we first expand the product
\[\begin{aligned}
\Re~\prod_{j = 1}^{2n+1} (1 - e^{ i \eta_j \zeta}) &= 1 + \sum_{\ell = 1}^{2n + 1} \sum_{1 \leq m_1 < m_2 < \dotsb < m_\ell \leq 2n + 1} (-1)^\ell \cos\big((\eta_{m_1} + \eta_{m_2} + \dotsb + \eta_{m_\ell})\zeta\big)\\
&= \sum_{\ell = 1}^{2n + 1} \sum_{1 \leq m_1 < m_2 < \dotsb < m_\ell \leq 2n + 1} (-1)^{\ell + 1} \left[1 - \cos\big((\eta_{m_1} + \eta_{m_2} + \dotsb + \eta_{m_\ell})\zeta\big)\right].
\end{aligned}\]

We replace the integral over $\R$ in \eqref{Tnintdef} by an integral over $\R\setminus(-\epsilon,\epsilon)$, where $\epsilon \ll 1$, and decompose the expression for $\Tb_n$ into a sum of terms of the form
\[\begin{aligned}
\int_{\epsilon <|\zeta| <\infty}\frac{1 - \cos(\eta\zeta)}{|\zeta|^{2n + 1}} \diff{\zeta} &= \int_{\epsilon < |\zeta| \leq 1/|\eta|} \frac{1 + \sum_{j = 1}^n \frac{(-1)^j (\eta\zeta)^{2j}}{(2j)!} - \cos(\eta\zeta)}{|\zeta|^{2n+1}} \diff{\zeta} + \int_{|\zeta| > 1/|\eta|} \frac{1 - \cos(\eta\zeta)}{|\zeta|^{2n + 1}} \diff{\zeta}\\
& \qquad - \sum_{j = 1}^n \frac{(-1)^{j} \eta^{2j}}{(2j)!} \int_{\epsilon < |\zeta| \leq 1/|\eta|} \frac{1}{|\zeta|^{2n-2j+1}} \diff{\zeta}\\
&= C_{n,1} \eta^{2n} - \sum_{j = 1}^n \frac{(-1)^{j} \eta^{2j}}{(2j)!} \int_{\epsilon < |\zeta| \leq 1/|\eta|} \frac{1}{|\zeta|^{2n-2j+1}} \diff{\zeta} + o(1),
\end{aligned}\]
where
\[
C_{n, 1} = \int_{|\theta| \leq 1} \frac{1 + \sum_{j = 1}^n \frac{(-1)^j (\theta)^{2j}}{(2j)!} - \cos(\theta)}{|\theta|^{2n+1}} \diff{\theta} + \int_{|\theta| > 1} \frac{1 - \cos(\theta)}{|\theta|^{2n + 1}} \diff{\theta}
\]
is some constant that depends only on $n$.

We have
\[\begin{aligned}
\sum_{j = 1}^n \frac{(-1)^{j} \eta^{2j}}{(2j)!} \int_{\epsilon < |\zeta| \leq 1/|\eta|} \frac{1}{|\zeta|^{2n-2j+1}} \diff{\zeta}
&= C_{n, 2}^\epsilon \eta^{2n} + \sum_{j = 1}^{n - 1} C_{n, 3}^{j, \epsilon} \eta^{2j} + C_{n, 4} \eta^{2n} \log|\eta|,
\end{aligned}\]
where
\[\begin{aligned}
C_{n, 2}^\epsilon &= \sum_{j = 1}^{n - 1} \frac{(-1)^{j + 1}}{(n - j) (2j)!} + 2 \frac{(-1)^{n + 1} \log\epsilon}{(2n)!},\qquad
C_{n, 3}^{j, \epsilon} = \frac{(-1)^j \epsilon^{2j - 2n}}{(n - j) (2j)!},\qquad
C_{n, 4} = 2 \frac{(-1)^{n+1}}{(2n)!}.
\end{aligned}\]
Thus, we conclude that
\[
\int_{\epsilon < |\zeta| \le 1/|\eta|} \frac{1 - \cos(\eta\zeta)}{|\zeta|^{2n + 1}} \diff{\zeta} = \big(C_{n, 1} - C_{n, 2}^\epsilon\big) \eta^{2n} - \sum_{j = 1}^{n - 1} C_{n, 3}^{j, \epsilon} \eta^{2j} - C_{n, 4} \eta^{2n} \log|\eta|.
\]
We use these results in the expression for  $\Tb_n$ and take the limit as $\epsilon \to 0^+$. The singularity at $\epsilon =0$ does not enter into the final result because of the cancelation in Lemma~\ref{p-cancel}, and we find that
\begin{equation}
\Tb_n(\etab_n) = 2 \frac{(-1)^{n+1}}{(2n)!} \sum_{\ell = 1}^{2n + 1} \sum_{1 \leq m_1 < m_2 < \dotsb < m_\ell \leq 2n + 1} (-1)^\ell (\eta_{m_1} + \dotsb + \eta_{m_\ell})^{2n} \log\left|\eta_{m_1} + \eta_{m_2} + \dotsb + \eta_{m_\ell}\right|.
\label{Tnexp}
\end{equation}
It follows that
\[
f_n = 2 \frac{(-1)^n}{(2n)!} \sum_{\ell = 1}^{2n + 1} \begin{pmatrix}2n + 1\\ \ell\end{pmatrix} (-1)^\ell \varphi^{2n - \ell + 1} \partial^{2n} \log|\partial| \big(\varphi^l\big).
\]
Therefore, we conclude that
\[
\begin{aligned}
& \int_\R \bigg[\frac{\varphi_x(x, t) - \varphi_x(x + \zeta, t)}{|\zeta|} - \frac{\varphi_x(x, t) - \varphi_x(x + \zeta, t)}{\sqrt{\zeta^2 + (\varphi(x, t) - \varphi(x + \zeta, t))^2}}\bigg] \diff{\zeta}\\
= ~~& - \sum_{n = 1}^\infty \frac{2 c_n (-1)^n}{\Gamma(2n + 2)} \px \bigg\{ \sum_{\ell = 1}^{2n + 1} \begin{pmatrix}2n + 1\\ \ell \end{pmatrix} (-1)^\ell \varphi^{2n - \ell + 1}(x, t) \px^{2n} \log|\px| \big(\varphi^l(x, t)\big)\bigg\}\\
= ~~& \sum_{n = 1}^\infty \sum_{\ell = 1}^{2n + 1} (-1)^{\ell + 1} d_{n, \ell} \px \Big\{\varphi^{2n - \ell + 1}(x, t) \px^{2n} \log|\px| \big(\varphi^l(x, t)\big)\Big\},
\end{aligned}
\]
where
\begin{equation}
\label{ddef}
d_{n, \ell} = \frac{2 \sqrt{\pi}}{\big|\Gamma\left(\frac 12 - n\right)\big| \Gamma(\ell + 1) \Gamma(2n + 2 - \ell) \Gamma(n + 1)} > 0.
\end{equation}
Using this expansion in \eqref{full-sqg}, we get \eqref{expd-sqg}.

\section{Para-differential calculus}
\label{sec:B}

In this appendix, we use the Weyl calculus \cite{lerner} to prove some estimates for Weyl paraproducts.

\subsection{Weyl operators}
The Weyl quantization of a symbol $a \colon \R\times\R\to \C$ is the operator $a^w$  defined by
\begin{align*}
(a^wf)(x)&=\frac1{2\pi}\iint_{\R^{2}} e^{i(x-y)\xi} a\bigg(\frac{x+y}2, \xi\bigg) f(y)\diff y\diff \xi
=\int_{\R} \F_2^{-1}a\bigg(\frac{x+y}2, x-y\bigg) f(y) \diff y,
\end{align*}
where $\F_i a$ denotes the Fourier transform of $a(x_1,x_2)$ with respect to with the $i$-th variable ($i = 1, 2$).
The Fourier transform of $a^w f$ can be written as
\begin{equation}\label{weyl2}
 \begin{aligned}
\F(a^w f)(\xi)&=\frac1{2\pi}\iiint_{\R^{3}} e^{i(x-y)\eta-ix\xi} a\bigg(\frac{x+y}2, \eta\bigg) f(y)\diff y\diff \eta \diff x
=\int_{\R} \F_1^{-1} a \bigg(\xi-\eta, \frac{\xi+\eta}2\bigg)\hat f(\eta)\diff \eta.
\end{aligned}\end{equation}

For $m\in \R$, we have the symbol class
\begin{align*}
S_{1,0}^m &=\left\{a(x,\xi)\in C^{\infty}(\R\times \R) \biggm| \sup_{\xi\in\R}\|\partial_\xi^\alpha\partial_x^\beta a(\cdot,\xi)\|_{L^\infty}\leq C_{\alpha\beta}(1+|\xi|)^{m-|\alpha|},\ \forall \alpha, \beta \in \Nz\right\}.
\end{align*}
For integers $r_1, r_2\geq 0$, we define a symbol norm by
\[
M^m_{r_1,r_2}(a) = \max\limits_{0\le \alpha \leq r_2}\sup\limits_{\xi\in\R}\left\|(1+|\xi|)^{\alpha - m}\partial_\xi^\alpha a(\cdot, \xi)\right\|_{W^{r_1,\infty}},
\]
and introduce a class of symbols with finite regularity
\begin{align*}
\Gamma_{r_1, r_2}^m &= \left\{a \colon \R \times \R \to \C \bigm| M^m_{r_1, r_2} (a) < \infty\right\}.
\end{align*}
 We note that if $\Mc_{(r_1,r_2)}$ is the symbol class defined in \eqref{defMc}, then
 \[
 \left\|(1+|\xi|)^{-m} a(x,\xi)\right\|_{\Mc_{(r_1,r_2)}}\approx M^m_{r_1, r_2} (a).
 \]
 In particular, $\Mc_{(r_1,r_2)} = \Gamma_{r_1, r_2}^0$.

\subsection{Para-differential operators}
Recall from Section~\ref{sec-weyl} that $\chi \colon \R \to \R$ is a smooth function supported in the interval $\{\xi\in \R \mid |\xi|\leq 1/10\}$
and equal to $1$  on $\{\xi\in \R \mid |\xi|\leq 3/40\}$.
If $f \colon \R\to\C$ and $a \colon \R \times \R \to \C$ is a symbol, then the Weyl paraproduct $T_a f$ in \eqref{weyldef} is defined by
\[
\F \left[T_a f\right](\xi)= \int_{\R} \chi\bigg(\frac{|\xi-\eta|^2}{1+|\xi+\eta|^2}\bigg)\tilde{a}\Big(\xi-\eta, \frac{\xi + \eta}{2}\Big)\hat f(\eta)\diff\eta.
\]
Introducing the notation
\begin{align*}
\sigma_a(\cdot, \zeta_2)=\F_{1}\bigg[\chi\bigg(\frac{|\zeta_1|^2}{1+4|\zeta_2|^2}\bigg)\tilde{a}(\zeta_1,\zeta_2)\bigg],
\end{align*}
we can also write
\[
\F [T_a f](\xi)= \int_{\R} \F_1^{-1} \sigma_a\bigg(\xi-\eta, \frac{\xi+\eta}2\bigg)\hat f(\eta)\diff\eta.
\]
Comparing this result with \eqref{weyl2}, we see that $T_a = \sigma_a^w$.

\begin{lemma}
\label{lem:Gamma}
If $a\in \Gamma_{r_1,r_2}^m$, then $\sigma_a\in\Gamma_{r_1,r_2}^m$ and $M^m_{r_1,r_2}(\sigma_a)\lesssim M^m_{r_1,r_2}(a)$.
\end{lemma}
\begin{proof}
To prove that $\sigma_a\in \Gamma_{r_1,r_2}^m$, we write
\[
\partial_{\zeta_2}^\alpha\partial_x^\beta \sigma_a(x,\zeta_2) = \sum\limits_{i_1 + i_2 = \alpha} c_{i_1, i_2, \alpha}\F_{\zeta_1}\bigg[\partial_{\zeta_2}^{i_1}\chi\bigg(\frac{|\zeta_1|^2}{1+4|\zeta_2|^2}\bigg)\partial_{\zeta_2}^{i_2} \widetilde{\partial_x^\beta a}(\zeta_1,\zeta_2)\bigg],
\]
where the $c_{i_1, i_2, \alpha}$ are multinomial coefficients.

For each term, by Young's inequaility,
\begin{align*}
&\bigg| (1+|\zeta_2|)^{\alpha}\F_{\zeta_1}\bigg[\partial_{\zeta_2}^{i_1}\chi\bigg(\frac{|\zeta_1|^2}{1+4|\zeta_2|^2}\bigg)\partial_{\zeta_2}^{i_2} \widetilde{\partial_x^\beta a}(\zeta_1,\zeta_2)\bigg]\bigg|\\
=~&\bigg|\F_{\zeta_1}\bigg[(1+ |\zeta_2|)^{i_1}\partial_{\zeta_2}^{i_1}\chi\bigg(\frac{|\zeta_1|^2}{1+4|\zeta_2|^2}\bigg)\bigg]\ast \left[(1+|\zeta_2|)^{i_2}\partial_{\zeta_2}^{i_2} \partial_x^\beta a(-x,\zeta_2)\right]\bigg|\\
\lesssim~ & \bigg\|\F_{\zeta_1}\bigg[(1+|\zeta_2|)^{i_1}\partial_{\zeta_2}^{i_1}\chi\bigg(\frac{|\zeta_1|^2}{1+4|\zeta_2|^2}\bigg)\bigg](x,\zeta_2)\bigg\|_{L^1_x} \left\|(1+ |\zeta_2|)^{i_2}\partial_{\zeta_2}^{i_2} \partial_x^\beta a(-x,\zeta_2)\right\|_{L_x^\infty}.
\end{align*}

Using the Fa\`{a} di Bruno's formula, a general term of $(1+|\zeta_2|)^{i_1}\partial_{\zeta_2}^{i_1}\chi\left(|\zeta_1|^2 / (1+4|\zeta_2|^2)\right)$ is a linear combination of the terms of the form
\[
(1 + |\zeta_2|)^{i_1} \chi^{(m_1 + \dotsb + m_{i_1})}\bigg(\frac{|\zeta_1|^2}{1 + 4 |\zeta_2|^2}\bigg) \prod_{\ell = 1}^{i_1} \bigg[\partial_{\zeta_2}^\ell \bigg(\frac{|\zeta_1|^2}{1 + 4 |\zeta_2|^2}\bigg)\bigg]^{m_\ell},
\]
where $m_\ell \in \N_0$ satisfies $\sum_{\ell = 1}^{i_1} \ell m_\ell = i_1$.

Bernstein's inequality implies that
\begin{align*}
\bigg\|\F_{\zeta_1}\bigg[(1+|\zeta_2|)^{i_1}\partial_{\zeta_2}^{i_1}\chi\bigg(\frac{|\zeta_1|^2}{1+4|\zeta_2|^2}\bigg)\bigg]
\bigg\|_{L^1_x}
&\lesssim \bigg\|\int_{\R} e^{-ix\zeta_1} \bigg((1+4|\zeta_2|^2)^{\frac{i_1}{2}}\partial_{\zeta_2}^{i_1}\chi\bigg(\frac{|\zeta_1|^2}{1+4|\zeta_2|^2}\bigg)\bigg)\diff \zeta_1 \bigg\|_{L^1_x} \lesssim 1,
\end{align*}
since the middle term in this inequality is supported on the set
$\{(\zeta_1, \zeta_2) \mid |\zeta_1|\lesssim \sqrt{1+4|\zeta_2|^2}\}$.
Therefore, we have that
\begin{align*}
&\sum\limits_{\alpha\leq r_2, \beta\leq r_1}\left\|(1+|\zeta_2|)^{\alpha}\partial_{\zeta_2}^\alpha\partial_x^\beta \sigma_a(x,\zeta_2)\right \|_{L^\infty_x}\lesssim \sum\limits_{\alpha\leq r_2, \beta\leq r_1}\sum\limits_{i \leq \alpha} \left\|(1+|\zeta_2|)^{i}\partial_{\zeta_2}^{i} \partial_x^\beta a(-x,\zeta_2)\right\|_{L_x^\infty}\lesssim (1+|\zeta_2|)^{m},
\end{align*}
so $M^m_{r_1,r_2}(\sigma_a)\lesssim M^m_{r_1,r_2}(a)$, which completes the proof.
\end{proof}

\subsection{$H^s$ estimates}
The next theorem follows from Theorem 1.2 in Boulkhemair \cite{Bou99}.

\begin{theorem}
\label{L2PDO}
Let $m\in\R$.
If $a\in \Gamma_{1,1}^m $, then the Weyl operator $a^w \colon H^{s}(\R) \to H^{s-m}(\R)$ with symbol $a$ is bounded  and its operator-norm is bounded by $M_{1,1}^m(a)$.
\end{theorem}

Using Lemma~\ref{lem:Gamma} and the fact that $T_a = \sigma_a^w$, we then get the following estimate for Weyl paraproducts.
\begin{theorem}\label{L2cont}
If $a\in \Gamma_{1,1}^m$, then the Weyl paraproduct operator $T_a \colon H^s(\R) \to H^{s-m}(\R)$ is bounded for all $m,s\in\R$, and
\[
\|T_a f\|_{H^{s-m}}\lesssim M_{1,1}^m(a) \|f\|_{H^s}.
\]
\end{theorem}

In particular, setting $m=0$ and using the fact that $M_{1,1}^0(a) \approx \|a\|_{\Mc_{(1,1)}}$, we get Lemma~\ref{lem:taest}.

\subsection{$L^\infty$-$L^2$ estimates}
We also need some estimates in which we bound $\|T_af\|_{L^2}$ by $\|f\|_{L^\infty}$.
\begin{theorem}\label{Linftybound}
Let $p(\xi)=|\xi|^k$, $k\geq 0$ or $p(\xi)=|\xi|^k\log |\xi|$, $k\geq 1$.
Assume $f\in L^\infty(\R)$ with $p(\partial_x)\partial_x f\in L^\infty(\R)$, and $a(x,\xi) $ is a function such that $\|a\|_{\mathcal L_1^2}<\infty$,
where
\[
\|a\|_{\mathcal L^2_1}:=\sup_{\xi} (\|a(\cdot,\xi)\|_{L^2}+\|\partial_\xi a(\cdot,\xi)\|_{L^2}).
\]
Then we have
\[
\|p(\partial_x) T_a f\|_{L^2}\lesssim (\|f\|_{L^\infty}+\|p(\partial_x)\partial_xf\|_{L^\infty})\|a\|_{\mathcal L^2_1}.
\]
\end{theorem}

\begin{proof}
Recall that
\begin{align*}
T_a f(x)=\sigma_a^w f(x)=\int_{\R}\F_2^{-1}\sigma_a\bigg(\frac{x+y}{2}, x-y\bigg)f(y)\diff y=\int_{\R}\chi\bigg(\frac{|\xi-\eta|^2}{1+|\xi+\eta|^2}\bigg)\tilde{a}\Big(\xi-\eta, \frac{\xi + \eta}{2}\Big)\hat f(\eta)\diff\eta.
\end{align*}
We split $T_af$ into a low-frequency part
\begin{align*}
\int_{\R}\F_2^{-1}\sigma_a\bigg(\frac{x+y}{2}, x-y\bigg)[\iota(i\partial_y)f(y)]\diff y=\int_{\R}\iota(\eta) \chi\bigg(\frac{|\xi-\eta|^2}{1+|\xi+\eta|^2}\bigg)\tilde{a}\Big(\xi-\eta, \frac{\xi + \eta}{2}\Big)\hat f(\eta)\diff\eta,
\end{align*}
and a high-frequency part
\begin{align*}
\int_{\R}\F_2^{-1}\sigma_a\bigg(\frac{x+y}{2}, x-y\bigg) \left[(1-\iota(i\partial_y))f(y)\right] \diff y=\int_{\R}(1-\iota(\eta)) \chi\bigg(\frac{|\xi-\eta|^2}{1+|\xi+\eta|^2}\bigg)\tilde{a}\Big(\xi-\eta, \frac{\xi + \eta}{2}\Big)\hat f(\eta)\diff\eta.
\end{align*}
Here, the cutoff function $\iota$ is the same as the one defined in the proof of Lemma~\ref{Commu}.

The integrand in the low-frequency part is supported in $|\xi|<6$, $|\eta| < 2$. Thus, $|\xi+\eta|<10$ and $|\xi-\eta|<10$ on its support, so we can put a cut-off function $\iota(\frac{\xi+\eta}5)\iota(\frac{\xi-\eta}5)$ into the integral without changing its value:
\begin{align*}
&\int_{\R}\F_2^{-1}\sigma_a\bigg(\frac{x+y}{2}, x-y\bigg)\left[\iota(i\partial_y)f(y)\right]\diff y
=\int_{\R}\iota(\eta) \chi\bigg(\frac{|\xi-\eta|^2}{1+|\xi+\eta|^2}\bigg)\iota\bigg(\frac{\xi+\eta}5\bigg) \iota\bigg(\frac{\xi-\eta}5\bigg) \tilde{a}\Big(\xi-\eta, \frac{\xi + \eta}{2}\Big)\hat f(\eta)\diff\eta.
\end{align*}
Therefore, defining $b(x, \xi)=\iota({i}/5\partial_x) \iota(2\xi/5) a(x, \xi)$, we have
\begin{align*}
\int_{\R}\F_2^{-1}\sigma_a\bigg(\frac{x+y}{2}, x-y\bigg)[\iota(i\partial_y)f(y)]\diff y &= \int_{\R}\iota(\eta) \chi\bigg(\frac{|\xi-\eta|^2}{1+|\xi+\eta|^2}\bigg)\tilde{b}\Big(\xi-\eta, \frac{\xi + \eta}{2}\Big)\hat f(\eta)\diff\eta\\*
& =\int_{\R}\F_2^{-1}\sigma_b\bigg(\frac{x+y}{2}, x-y\bigg)[\iota(i\partial_y)f(y)]\diff y.
\end{align*}
So we obtain
\begin{align*}
&\bigg\|p(\partial_x)\int_{\R}\F_2^{-1}\sigma_b\bigg(\frac{x+y}{2}, x-y\bigg)[\iota(i\partial_y)f(y)]\diff y\bigg\|_{L^2_x}\\
\lesssim~& \bigg\|\int_{\R}\F_2^{-1}\sigma_b\bigg(\frac{x+y}{2}, x-y\bigg)[\iota(i\partial_y)f(y)]\diff y\bigg\|_{L^2_x}\\
\lesssim~& \|\iota(i\partial_y)f(y)\|_{L^\infty_y}\bigg\|\F_2^{-1}\sigma_b\bigg(\frac{x+y}{2}, x-y\bigg)\bigg\|_{L^1_yL^2_x}\\
=~& \|\iota(i\partial_y)f(y)\|_{L^\infty_y}\bigg\|\F_2^{-1}\sigma_b\bigg(x-\frac z2, z\bigg)\bigg\|_{L^1_zL_x^2},
\end{align*}
where the last term satisfies
\begin{align*}
\bigg\|\F_2^{-1}\sigma_b\bigg(x-\frac z2, z\bigg)\bigg\|_{L^1_zL_x^2}&=\bigg\|\F_\xi^{-1}\F_{\zeta_1}\bigg[\chi\bigg(\frac{|\zeta_1|^2}{1+4|\xi|^2}\bigg)\tilde{b}(\zeta_1,\xi)\bigg] \bigg(x-\frac z2, z\bigg)\bigg\|_{L^1_zL_x^2}\\
&=\bigg\|\F_\xi^{-1}\F_{\zeta_1}\bigg\{(1-\partial_\xi^2)^{1/2}\bigg[\chi\bigg(\frac{|\zeta_1|^2}{1+4|\xi|^2}\bigg)\tilde{b}(\zeta_1,\xi)\bigg]\bigg\} \bigg(x-\frac z2, z\bigg)\frac1{(1+z^2)^{1/2}}\bigg\|_{L^1_zL_x^2}\\
&\lesssim
\bigg\|\F_\xi^{-1}\F_{\zeta_1}\bigg\{(1-\partial_\xi^2)^{1/2}\bigg[\chi\bigg(\frac{|\zeta_1|^2}{1+4|\xi|^2}\bigg)\tilde{b}(\zeta_1,\xi)\bigg]\bigg\} (x, z)\bigg\|_{L^2_zL_x^2}\\
&= \bigg\|(1-\partial_\xi^2)^{1/2}\bigg[\chi\bigg(\frac{|\zeta_1|^2}{1+4|\xi|^2}\bigg)\tilde{b}(\zeta_1,\xi)\bigg] (\zeta_1, \xi)\bigg\|_{L^2_\xi L_{\zeta_1}^2}\\
&\lesssim \|b\|_{L^2_xH^{1}_\xi}\lesssim \|a\|_{\mathcal L^2_1}.
\end{align*}

For the high frequency part, we make a dyadic decomposition of $f$, after which we mainly need to estimate
\[
\int_{\R}\F_2^{-1}\sigma_a\bigg(\frac{x+y}{2}, x-y\bigg) \left[(1-\iota(i\partial_y))\psi_k(i\partial_y)f(y)\right] \diff y= \int_{\R} (1-\iota(\eta))\chi\bigg(\frac{|\xi-\eta|^2}{1+|\xi+\eta|^2}\bigg)\tilde{a}\Big(\xi-\eta, \frac{\xi + \eta}{2}\Big)\psi_k(\eta) \hat f(\eta)\diff\eta.
\]

When $|\eta| > 2$, we have
\[
\frac1{18}|\eta|\leq |\xi| \leq \frac{35}{18}|\eta|,\quad  \frac{1}2|\eta|\leq |\xi+\eta|\leq \frac{40}9|\eta|
\]
on the support of the cut-off function $\chi\left(\frac{|\xi-\eta|^2}{1+|\xi+\eta|^2}\right)$.
Therefore $|\eta|\approx|\xi+\eta|\approx |\xi|\approx 2^k$ on the support, and, since $|\eta| > 2$, we only need to consider $k\geq 0$.

By the H\"older inequality and a change of coordinates,
\begin{align*}
&\bigg\|p(\partial_x)\int_{\R}\F_2^{-1}\sigma_a\bigg(\frac{x+y}{2}, x-y\bigg)[(1-\iota(i\partial_y))f_k(y)]\diff y\bigg\|_{L^2_x}\\
\lesssim~&2^{-k}\| p(\partial_x)\px f_k\|_{L^\infty}\bigg\|\F_\xi^{-1}\F_{\zeta_1}\bigg[\chi\bigg(\frac{|\zeta_1|^2}{1+4|\xi|^2}\bigg)\tilde{a}(\zeta_1,\xi)\psi_k(\xi)\bigg]\bigg(\frac{x+y}{2}, x-y\bigg)\bigg\|_{L^1_yL^2_x}\\
\lesssim~&2^{-k}\| p(\partial_x)\px f\|_{L^\infty}\bigg\|\F_\xi^{-1}\F_{\zeta_1}\bigg[\chi\bigg(\frac{|\zeta_1|^2}{1+4|\xi|^2}\bigg)\tilde{a}(\zeta_1,\xi)\psi_k(\xi)\bigg]\bigg(x+\frac{z}{2}, z\bigg)\bigg\|_{L^1_zL^2_x}.
\end{align*}
The last term satisfies
\begin{align*}
&\bigg\|\F_\xi^{-1}\F_{\zeta_1}\bigg[\chi\bigg(\frac{|\zeta_1|^2}{1+4|\xi|^2}\bigg)\tilde{a}(\zeta_1,\xi)\psi_k(\xi)\bigg]\bigg(x+\frac{z}{2}, z\bigg)\bigg\|_{L^1_zL^2_x}\\
=~&\bigg\|\F_\xi^{-1}\F_{\zeta_1}\bigg\{(1-\partial_\xi^2)^{1/2}\bigg[\chi\bigg(\frac{|\zeta_1|^2}{1+4|\xi|^2}\bigg)\tilde{a}(\zeta_1,\xi)\psi_k(\xi)\bigg]\bigg\}\bigg(x+\frac{z}{2}, z\bigg)\frac1{(1+z^2)^{1/2}}\bigg\|_{L^1_zL^2_x}\\
\lesssim~&
\bigg\|\F_\xi^{-1}\F_{\zeta_1}\bigg\{(1-\partial_\xi^2)^{1/2}\bigg[\chi\bigg(\frac{|\zeta_1|^2}{1+4|\xi|^2}\bigg)\tilde{a}(\zeta_1,\xi)\psi_k(\xi)\bigg]\bigg\}(x, z)\bigg\|_{L^2_zL^2_x}\\
\lesssim~&\bigg\|\partial_\xi\bigg[\chi\bigg(\frac{|\zeta_1|^2}{1+4|\xi|^2}\bigg)\tilde{a}(\zeta_1,\xi)\psi_k(\xi)\bigg]\bigg\|_{L^2_{\zeta_1}L^2_\xi}+\bigg\|\chi\bigg(\frac{|\zeta_1|^2}{1+4|\xi|^2}\bigg)\tilde{a}(\zeta_1,\xi)\psi_k(\xi)\bigg\|_{L^2_{\zeta_1}L^2_\xi}\\
\lesssim~&\bigg\|\frac{8|\xi||\zeta_1|^2}{(1+4|\xi|^2)^2}\chi'\bigg(\frac{|\zeta_1|^2}{1+4|\xi|^2}\bigg)\tilde{a}(\zeta_1,\xi)\psi_k(\xi)\bigg\|_{L^2_{\zeta_1}L^2_\xi}+\bigg\|\chi\bigg(\frac{|\zeta_1|^2}{1+4|\xi|^2}\bigg)\partial_\xi\tilde{a}(\zeta_1,\xi)\psi_k(\xi)\bigg\|_{L^2_{\zeta_1}L^2_\xi}\\
&+\bigg\|\chi\bigg(\frac{|\zeta_1|^2}{1+4|\xi|^2}\bigg)\tilde{a}(\zeta_1,\xi)\partial_\xi\psi_k(\xi)\bigg\|_{L^2_{\zeta_1}L^2_\xi}+\bigg\|\chi\bigg(\frac{|\zeta_1|^2}{1+4|\xi|^2}\bigg)\tilde{a}(\zeta_1,\xi)\psi_k(\xi)\bigg\|_{L^2_{\zeta_1}L^2_\xi}\\
\lesssim~&2^{-k/2} \| a(x, \xi)\|_{L^2_xL^{\infty}_\xi}+2^{k/2} \|\partial_\xi a(x, \xi)\|_{L^2_xL^{\infty}_\xi}+2^{-k/2} \| a(x, \xi)\|_{L^2_xL^{\infty}_\xi}+2^{k/2} \| a(x, \xi)\|_{L^2_xL^{\infty}_\xi}.
\end{align*}
Summing  these inequalities over $k\geq 0$, we obtain that
\begin{align*}
\bigg\|p(\partial_x)\int_{\R}\F_2^{-1}\sigma_a\bigg(\frac{x+y}{2}, x-y\bigg)[(1-\iota(i\partial_y))f(y)]\diff y\bigg\|_{L^2_x}
\lesssim~\|p(\partial_x)\px f\|_{L^\infty}\|a\|_{\mathcal L^2_1}.
\end{align*}
The theorem then follows by combining the low and high frequency estimates.
\end{proof}

\subsection{Composition}
Finally, we state a commutator estimate for Weyl paraproducts. The composition of two symbols $a$ and $b$ is defined by
\[
a\#b(x,\xi)=\iint_{\R^2}e^{-iy\eta}a(x,\xi+\eta)b(y+x,\xi)\diff y\diff\eta.
\]
The following theorem is from \cite{lerner} (see Theorem~2.3.7).
\begin{theorem}[Composition]
Let $a_1\in S_{1,0}^{m_1}$ and $a_2\in S_{1,0}^{m_2}$. Then
\[
a_1\#a_2-a_1a_2-\frac{1}{2i}\{a_1,a_2\}\in S^{m_1+m_2-2}_{1,0},
\]
where $\{a_1, a_2\} = \partial_\xi a_1 \partial_x a_2 - \partial_\xi a_2 \partial_x a_1$ is the Poisson bracket.
\label{composition}
\end{theorem}
Using Theorem \ref{L2cont}, we therefore obtain the following estimate.
\begin{theorem}
Let $a \in \Gamma_{3, 3}^{m_1}$, $b \in \Gamma_{3, 3}^{m_2}$, and $f \in H^s(\R)$. Then
\[
T_a T_b f = T_{ab} f + \frac{1}{2 i} T_{\{a, b\}} f + \Rc,
\]
where $\{a, b\} = \partial_2 a\cdot \partial_1 b - \partial_2 a\cdot \partial_1 b$ is the Poisson bracket of $a$ and $b$, and the remainder $\Rc$ satisfies
\[
\|\Rc\|_{H^{s-(m_1+m_2-2)}} \lesssim M_{3, 3}^{m_1}(a) M_{3, 3}^{m_2}(b) \|f\|_{H^s}.
\]
\end{theorem}


\begin{thebibliography}{99}
\bibitem{BeCo} \textsc{A. L. Bertozzi and P. Constantin}. Global regularity for vortex patches. \emph{Comm. Math. Phys.}, \textbf{152}(1), 19-28, 1993.


\bibitem{BCD11} \textsc{H.~Bahouri, J.-Y.~Chemin, and R.~Danchin}. {\it Fourier Analysis and Nonlinear Partial Differential Equations}. Grundlehren der Mathematischen Wissenschaften, {\bf 343}. Springer, Heidelberg, 2011.

\bibitem{Bou99} \textsc{A.~Boulkhemair}. $L^2$ estimates for Weyl quantization, {\it Journal of Functional Analysis}, {\bf 165}, 173-204, 1999.

\bibitem{BSV19} \textsc{T. Buckmaster, S. Shkoller and V. Vicol}. Nonuniqueness of weak solutions to the SQG equation. \emph{Comm. Pure Appl. Math.}, {\bf 72}(9), 1809--1874, 2019.


\bibitem{CCGS16a} \textsc{A.~Castro, D.~C\'{o}rdoba, and J.~G\'{o}mez-Serrano}. Existence and regularity of rotating global solutions for the generalized surface quasi-geostrophic equations. {\it Duke Math. J.}, {\bf 165}(5), 93--984, 2016.

\bibitem{CCGS16b} \textsc{A.~Castro, D.~C\'{o}rdoba, and J.~G\'{o}mez-Serrano}. Uniformly rotating analytic global patch solutions for active scalars. {\it Annals of PDE}, {\bf 2}(1), 1--34, 2016.

\bibitem{CCGS18} \textsc{A.~Castro, D.~C\'{o}rdoba, and J.~G\'{o}mez-Serrano}. Global smooth solutions for the inviscid SQG equation. {\it Mem. Amer. Math. Soc.}, {\bf 266}(1292), 2020.

\bibitem{CCCGW12} \textsc{D.~Chae, P.~Constantin, D.~C\'{o}rdoba, F.~Gancedo, and J.~Wu}. Generalized surface quasi-geostrophic
    equations with singular velocities. {\it Comm. Pure Appl. Math.}, {\bf 65}(8), 1037--1066, 2012.

\bibitem{Che} \textsc{J. Y. Chemin}. Persistence of geometric structures in two-dimensional incompressible fluids. \emph{Ann. Sci. Ecole. Norm. Sup.}, {\bf 26}(4), 517-542, 1993.

\bibitem{Che1} \textsc{J. Y. Chemin}. \emph{Perfect Incompressible Fluids}, Oxford University Press, New York, 1998.

\bibitem{sqg} \textsc{P.~Constantin, A.~J.~Majda, and E.~G.~Tabak}. Singular front formation in a model for quasigeostrophic flow. \emph{Phys. Fluids}, \textbf{6}, 9--11, 1994.

\bibitem{CCG18} \textsc{A. C\'{o}rdoba, D. C\'{o}rdoba and F. Gancedo}. Uniqueness for SQG patch solutions. {\it Trans. Amer. Math. Soc.}, {\bf Ser. B}.(5), 1--31, 2018.

\bibitem{CFR04}\textsc{D.~C\'{o}rdoba, C.~Fefferman and J.~L.~Rodrigo}. Almost sharp fronts for the surface quasi-geostrophic equation. {\it Proc.
Natl. Acad. Sci. USA}, {\bf 101}(9), 2687--2691, 2004.

\bibitem{CFMR05} \textsc{D.~C\'{o}rdoba, M.~A.~Fontelos, A.~M.~Mancho, and J.~L.~Rodrigo}. Evidence of singularities for a family of contour dynamics equations. {\it Proc. Natl. Acad. Sci.}, {\bf 102}(17), 5949--5952, 2005.

\bibitem{CGI17} \textsc{D.~C\'{o}rdoba, J.~G\'{o}mez-Serrano, and A.~D.~Ionescu}. Global solutions for the generalized SQG patch equation. {\it Arch. Rational Mech. Anal.}, {\bf 233}(3), 1211--1251, 2019.

\bibitem{DIP17} \textsc{Y.~Deng, A.~D.~Ionescu, and B.~Pausader}. The Euler-Maxwell system for electrons: global solutions in 2D. {\it Arch. Rational Mech. Anal.}, {\bf 225}(2), 771--871, 2017.

\bibitem{DIPP16} \textsc{Y.~Deng, A.~D.~Ionescu, B.~Pausader, and F.~Pusateri.} Global solutions of the gravity-capillary water wave system in 3 dimensions. \emph{Acta Mathematica}, {\bf 219}, 213--402, 2017.

\bibitem{FLR12} \textsc{C.~Fefferman, G.~Luli, and J. Rodrigo}. The spine of an SQG almost-sharp front. {\it Nonlinearity}, {\bf 25}(2), 329--342, 2012.

\bibitem{FR11}\textsc{C.~Fefferman and J.~L.~Rodrigo}. Analytic sharp fronts for the surface quasi-geostrophic equation. {\it Comm. Math. Phys.},
{\bf 303} (1), 261--288, 2011.

\bibitem{FR12}\textsc{C.~Fefferman and J.~L.~Rodrigo}. Almost sharp fronts for SQG: the limit equations. {\it Comm. Math. Phys.}, {\bf 313}(1),
131--153, 2012.

\bibitem{FR15}\textsc{C.~Fefferman and J.~L.~Rodrigo}. Construction of almost-sharp fronts for the surface quasi-geostrophic equation. {\it Arch. Rational Mech. Anal.}, {\bf 218}(1), 123--162, 2015.

\bibitem{Gan08} \textsc{F. Gancedo}. Existence for the $\alpha$-patch model and the QG sharp front in Sobolev spaces. {\it Adv. Math.}, {\bf 217}(6), 2569--2598, 2008.


\bibitem{GS14} \textsc{F.~Gancedo and R.~M.~Strain}. Absence of splash singularities for SQG sharp fronts and the muskat problem. {\it
    Proc. Natl. Acad. Sci.}, {\bf 111}(2), 635--639, 2014.


\bibitem{Germain}  \textsc{P.~Germain}. Space-time resonances.
\emph{Journ\'ees \'equations aux d\'eriv\'ees partielles}, \textbf{8}, 1--10, 2010.

\bibitem{GMS09} \textsc{P.~Germain, N.~Masmoudi, and J.~Shatah}. Global solutions for 3d quadratic Schrodinger equations. {\it Int. Math. Res. Not.}, {\bf 2009}, 414--432, 2009.

\bibitem{GMS12} \textsc{P.~Germain, N.~Masmoudi, and J.~Shatah}. Global solutions for the gravity water waves equation in dimension 3. {\it Ann.  Math.}, {\bf 175}, 691--754, 2012.

\bibitem{GS18} \textsc{J.~G\'{o}mez-Serrano}. On the existence of stationary patches. {\it Adv. Math.}, {\bf 343}, 110--140, 2019.


\bibitem{Hor} \textsc{L.~H\"{o}rmander}. {\it The Analysis of Linear Partial Differential Operators. III.  Pseudo-Differential Operators. }  Grundlehren der Mathematischen Wissenschaften, {\bf 274}. Springer-Verlag, Berlin, 1985.

\bibitem{HSh17} \textsc{J.~K.~Hunter and J.~Shu}. Regularized and approximate equations for sharp fronts in the surface quasi-geostrophic equation and its generalization. \emph{Nonlinearity}, {\bf 31}(6), 2480--2517, 2018.

\bibitem{HSZ1} \textsc{J.~K.~Hunter, J.~Shu, and Q.~Zhang}. Local well-posedness of an approximate equation for SQG fronts. {\it J. Math. Fluid Mech.}, {\bf 20}(4), 1967--1984, 2018.

\bibitem{HSZ19} \textsc{J.~K.~Hunter, J.~Shu, and Q.~Zhang}. Contour dynamics for surface quasi-geostrophic fronts. \emph{Nonlinearity},
\textbf{33}, 4699--4714, 2020.

\bibitem{IP12} \textsc{A.~D.~Ionescu, and B.~Pausader}. Nonlinear fractional Schr\"odinger equations in one dimension. \emph{J. Func. Anal.},
\textbf{266}, 2012.


\bibitem{IP13}\textsc{A.~D.~Ionescu, and B.~Pausader}. The Euler-Poisson system in 2d: global stability of the constant equilibrium solution. {\it Int. Math. Res. Not.}, {\bf 2013}(4), 761--826, 2013.

\bibitem{IP14} \textsc{A.~D.~Ionescu and F.~Pusateri}. Global regularity for 2D water waves with surface tension. \emph{Memoirs of the AMS}, {\bf 256}(1277), v+124, 2018.

\bibitem{IP15} \textsc{A.~D.~Ionescu and F.~Pusateri}. Global solutions for the gravity water waves system in 2D. \emph{Invent. Math.}, {\bf 199}, 653--804, 2015.

\bibitem{IPu16} \textsc{A.~D.~Ionescu and F.~Pusateri}. Global analysis of a model for capillary water waves in two dimensions, \emph{Comm. Pure Appl. Math.}, {\bf 69}, 2016.

\bibitem{IM20a} \textsc{P.~Isett, A.~Ma}. A direct approach to nonuniqueness and failure of compactness for the SQG equation, arXiv:2007.03078.


\bibitem{KR20a}\textsc{C.~Khor, J.~L.~Rodrigo}. Local existence of analytic sharp fronts for singular SQG, arXiv:2001.10412.

\bibitem{KR20b}\textsc{C.~Khor, J.~L.~Rodrigo}. On sharp fronts and almost-sharp fronts for singular SQG, arXiv:2001.10332.

\bibitem{KiRyYaZl} \textsc{A. Kiselev, L. Ryzhik, Y. Yao, and A. Zlato\v{s}}. Finite time singularity for the modified SQG patch equation. \emph{Annals of Mathematics}, \textbf{184}(3), 909--948, 2016.

\bibitem{KYZ17} \textsc{A. Kiselev, Y. Yao and A. Zlato\v{s}}. Local Regularity for the Modified SQG Patch Equation. {\it Comm. Pure Appl. Math}, {\bf 70}(7), 1253--1315, 2017.


\bibitem{sqg_lap} \textsc{G.~Lapeyre}. Surface quasi-geostrophy, \emph{Fluids},  \textbf{2},  2017.


\bibitem{kato-ponce} \textsc{D.~Li}. On Kato-Ponce and fractional Leibniz. {\it Rev. Mat. Iberoam.}, {\bf 35}(1), 23--100, 2019.

\bibitem{lerner} \textsc{N.~Lerner}. {\it Metrics on the phase space and non-selfadjoint pseudodifferential operators}, Pseudo-Differential Operators. Theory and Applications, 3. { Birkh\"auser Verlag, Basel}, 2010. xii+397 pp.

\bibitem{majda} \textsc{A.~J.~Majda and A.~L.~Bertozzi}. \emph{Vorticity and Incompressible Flow}, Cambridge University Press, Cambridge, 2002.

\bibitem{Mar08} \textsc{F.~Marchand}. Existence and regularity of weak solutions to the quasi-geostrophic equations in the spaces $L^p$ or $\dot{H}^{-1/2}$, {\it Commun. Math. Phys.}, {\bf 277}, 45--67, 2008.

\bibitem{ozawa} \textsc{T.~Ozawa}. Long range scattering for nonlinear Schr\"{o}dinger equations in one space dimension, \emph{Comm. Math. Phys.}, {\bf 139}, 479--493, 1991.


\bibitem{Ped87} \textsc{J.~Pedlosky}. {\it Geophysical Fluid Dynamics}, 2nd ed. Springer-Verlag, New York, N.Y., 1987.

\bibitem{Res} \textsc{S. Resnick}. \emph{Dynamical Problems in Nonlinear Advective Partial Differential Equations}. Ph.D. thesis, University of Chicago, Chicago, 1995.

\bibitem{Ro05} \textsc{J.~L.~Rodrigo}. On the evolution of sharp fronts for the quasi-geostrophic equation. {\it Comm. Pure and Appl. Math.}, {\bf 58}, 0821--0866, 2005.

\bibitem{Dri} \textsc{R.~K.~Scott and D.~G.~Dritschel}. Numerical simulation of a self-similar cascade of filament instabilities in the
Surface quasigeostrophic System. \emph{Phys. Rev. Lett.}, \textbf{112}, 144505, 2014.

\bibitem{SD19} \textsc{R.~K.~Scott and D.~G.~Dritschel}. Scale-invariant singularity of the surface quasigeostrophic patch. \emph{J. Fluid Mech.}, \textbf{863}(R2), 2019.

\bibitem{Tay00} \textsc{M.~E.~Taylor}. {\it Tools for PDE: Pseudodifferential Operators, Paradifferential Operators, and Layer Potentials}. Mathematical Surveys and Monographs, {\bf 81}. American Mathematical Society, Providence, RI, 2000.


\end{thebibliography}
\end{document}